\documentclass[a4paper,reqno]{amsart}

\usepackage[a4paper,width={16cm},left=2.5cm,bottom=3cm, top=3cm]{geometry}
\usepackage{enumerate}
\usepackage{yfonts}

\usepackage[english]{babel}

\usepackage{a4wide}
\usepackage[dvipsnames]{xcolor}
\usepackage{amsmath}
\usepackage{amssymb}
\usepackage{amsfonts}
\usepackage{amsthm,graphicx}
\usepackage{comment}
\usepackage{csquotes}

\usepackage{todonotes}

\usepackage{hyperref}
\hypersetup{
	linktocpage=true,
	colorlinks=true,
	linkcolor=blue!70!black,
	citecolor=orange!40!red,
	urlcolor=magenta!80!black,
}

\usepackage{mathtools}
\mathtoolsset{showonlyrefs}

\numberwithin{equation}{section}
\newtheorem{theorem}{Theorem}[section]
\newtheorem{lemma}[theorem]{Lemma}
\newtheorem{corollary}[theorem]{Corollary}

\newtheorem{proposition}[theorem]{Proposition}
\theoremstyle{definition}

\newtheorem{example}[theorem]{Example}
\newtheorem{remark}[theorem]{Remark}

\newcommand{\mc}{\mathcal}
\newcommand{\mb}{\mathbb}

\newcommand{\la}{\lambda}
\newcommand{\norm}[1]{\left\lVert#1\right\rVert}
\newcommand{\pd}[2]{\frac{\partial#1}{\partial#2}}

\newcommand{\R}{\mb{R}}
\newcommand{\C}{\mb{C}}
\newcommand{\N}{\mb{N}}

\newcommand{\e}{\varepsilon}

\newcommand{\dive}{\mathop{\rm{div}}}

\newcommand{\Rea}{\mathop{\rm{Re}}}
\newcommand{\Ima}{\mathop{\rm{Im}}}
\newcommand{\di}{\mathop{\rm{dist}}}

\newcommand{\capa}[2]{\textnormal{\rm{Cap}}_{#1}(#2)}

\usepackage{subfig}
\usepackage{tikz}
\usepackage{xcolor}
\usepackage{pgfplots}
\pgfplotsset{compat=1.18}
\usepgfplotslibrary{fillbetween}
\usepackage{mathrsfs}
\usetikzlibrary{arrows}
\usetikzlibrary{shadings}
\makeatletter

\begin{document}
\title[Local regularity for anisotropic magnetic operators]{Local regularity for anisotropic magnetic 
operators with general codimension
singularities}


\author{Giovanni Siclari}
\address{Giovanni Siclari\newline\indent
Centro di Ricerca Matematica Ennio De Giorgi
\newline\indent
Scuola Normale Superiore di Pisa
\newline\indent
Piazza dei Cavalieri 3, 56126, Pisa, Italy}
\email{giovanni.siclari@sns.it}

\author{Stefano Vita}
\address{Stefano Vita\newline\indent
Dipartimento di Matematica ``Felice Casorati''
\newline\indent
Universit\`a degli Studi di Pavia
\newline\indent
Via Ferrata 5, 27100, Pavia, Italy}
\email{stefano.vita@unipv.it}

\date{\today}

\begin{abstract}
We study local regularity properties of solutions to stationary anisotropic magnetic Schr\"odinger equations in $\mathbb{R}^d$, $d \ge 2$, arising from singular magnetic potentials concentrated along manifolds of general codimension $2 \le n \le d$.

The magnetic interaction is modeled through a covariant gradient of the form
\[
\nabla_m u = (iM\nabla + A)u,
\]
where $M^T M$ is a uniformly elliptic matrix encoding anisotropy and $A$ is a magnetic potential with critical Hardy-type scaling along the $n$-codimensional singular set $\Sigma_0$; that is, $A\sim \mathrm{dist}(\cdot,\Sigma_0)^{-1}$.

We establish local H\"older $C^{0,\alpha}$ and Schauder $C^{1,\alpha}$ estimates for weak solutions via a blow-up analysis adapted to the magnetic structure. The regularity is deeply influenced by the combined effect of anisotropy and the singular magnetic potential, which determines the spectrum of the limiting spherical Laplace-Beltrami operator arising in the blow-up at the singular set.

Our model is motivated by the study of magnetic potentials generated by shrinking solenoids onto an axis $\Sigma_0$, in the three-dimensional setting $d=3$, $n=2$, leading to Aharonov--Bohm-type (AB) models. In this framework, we show that the geometry of the solenoidal loops plays a crucial role: in particular, any deviation from planar cross-sections orthogonal to $\Sigma_0$ induces a twofold effect. On the one hand, it breaks the ideal AB configuration, in the sense that the magnetic field outside the solenoid is no longer vanishing. On the other hand, it yields an unexpected regularizing mechanism on the wave functions, through a positive shift in the eigenvalues of the asymptotic spectral problem. This purely three-dimensional effect is consistent with our $C^{1,\alpha}$ regularity estimates, and motivates even higher Schauder regularity results,
in contrast with the two-dimensional AB model, where the optimal regularity is limited to $C^{0,\alpha}$, for some $\alpha \in (0,1/2]$.
\end{abstract}

\maketitle

\section{Introduction}\label{sec_intro}
Let $d \ge 2$ and $2\le n \le d$ be two integers. We are interested in the study of the local regularity properties of solutions to stationary anisotropic magnetic Schr\"odinger equations in $\mathbb{R}^d$, arising from singular magnetic fields concentrated along manifolds of general codimension $n$.

The magnetic interaction is modeled through a covariant gradient of the form
\[
\nabla_m u=(iM\nabla + A)u,
\]
where $A$ is the magnetic potential and $M$ is a $d\times d$-dimensional matrix.

Since our analysis is local in nature, we can confine our problem in a ball $B_R\subset\R^d$ of radius $R>0$ and centered at $0$. Let us denote the space variable by $z=(x,y) \in \R^{d-n}\times \R^n$. Then, the magnetic potential $A$ has the form
\begin{equation}\label{def_A}
A(z):=\frac{a\big(x, \frac{y}{|y|}\big)}{|y|}+ b(z),
\end{equation}
where $a\in L^\infty(B'_R\times\mathbb S^{n-1}, \R^d)$, $b \in L^d(B_R, \R^d)$ and $B_R'$ is the ball in $\R^{d-n}$ of radius $R>0$ and centered at $0$. The first term models the strong singularity of the magnetic potential along a flat thin manifold 
\begin{equation}\label{def_Sigma}
\Sigma_0:=\{(x,y) \in \R^d: |y|=0\}
\end{equation}
of dimension $d-n$ (codimension $n$). Notice that $|y|=\mathrm{dist}(z,\Sigma_0)$, and then the scaling of this term is critical (of Hardy-type). The second term instead acts as a more regular magnetic potential; that is, with subcritical scaling.

Moreover, we allow for an anisotropic diffusion, which is encoded by the presence of the matrix $M=(m_{i,j})_{i,j=1,\dots,d}$, which is invertible and with bounded measurable coefficients. We assume  that $N:=M^TM$ is uniformly elliptic; that is, there exist two constants $0<\la\leq\Lambda$ such that
\begin{equation}\label{hp_M_elliptic}
\la |\xi|^2\le N\xi\cdot \overline{\xi} \le \Lambda |\xi|^2,\qquad\forall \ \xi \in \C^d.
\end{equation}

Then, we consider the model equation
\begin{equation}\label{eq_magnetic_strong}
L_mu=0\quad  \text{ in } B_R,
\end{equation}
involving the anisotropic Schr\"odinger magnetic operator
\begin{equation}\label{def_L_Af}
L_mu:= (iM\nabla +A)^2 u= -\dive(M^TM \nabla u)+2i M A\cdot \nabla  u +i M \nabla \cdot A u+ |A|^2 u,
\end{equation}
where
\begin{equation}
 M \nabla \cdot A = \sum_{i,j=1}^d m_{i,j}\pd{A_i}{z_j}.
\end{equation}

A solution $u:B_R\to\C$ of \eqref{eq_magnetic_strong} is an element of a suitable magnetic Sobolev space $H^1_m(B_R)$, see Section \ref{subsec_magnetic_capa_Sobolev}, and must be intended in a weak sense; that is,
\begin{equation}
\int_{B_R} \nabla_{m} u \cdot \overline{\nabla_{m} \phi} \, dz=0, \quad  \text{ for any } \phi \in  C^\infty_c(B_R,\mb{C}).
\end{equation}
See Section \ref{sec_weak_sol_bounded} for a precise definition.
The variable $u$ represents the wave function of a non-relativistic charged particle in the magnetic field $B = d A$, with $|u|^2$ describing the associated probability density. Notice that $dA$ denotes the $2$-form given by the exterior derivative of $A$, seen as a $1$-form; i.e
\begin{equation}
B=dA = \sum_{1 \le i < j \le d}
\left(\frac{\partial A_j}{\partial z_i} - \frac{\partial A_i}{\partial z_j}\right)
\, dz_i \wedge dz_j,\qquad A = \sum_{i=1}^d A_i\, dz_i.  
\end{equation}

To establish an appropriate functional framework, we require certain Hardy-type inequalities, see Section \ref{subsec_Hardy}. These inequalities rely on some compatibility conditions between $a$ and $M$. The first one is the following anisotropic transversality condition
\begin{equation}\label{hp_transversality}
M^{-1} a\left(x, \frac{y}{|y|}\right) \cdot (0,y)=0,\qquad\mathrm{in} \ B_R.
\end{equation}
In order to state the other condition, let us introduce a general magnetic potential $P:\R^n \to \R^n$ of the form 
\begin{equation}\label{magneticPy}
P(y)=\frac{p\big(\frac{y}{|y|}\big)}{|y|},
\end{equation}
where $p \in L^\infty(\mb{S}^{n-1},\mb{\R}^n)$, $p\not \equiv 0$ and $h \in L^\infty(\mb{S}^{n-1})$ with $h \ge 0$. Then, let us consider
\begin{equation}\label{def:mu1}
\mu_1(p,h)=\inf_{\psi \in H^1(\mb{S}^{n-1},\mb{C}),\psi \neq 0}
\frac{\int_{\mb{S}^{n-1}} |i\nabla_{\mb{S}^{n-1}} \psi +p \psi|^2  +h |\psi|^2 d \sigma}{\int_{\mb{S}^{n-1}}  |\psi|^2 d \sigma}.
\end{equation}
The second compatibility condition regards the non-degeneracy of the constant
\begin{equation}\label{hp_hardy}
 c_{n,a,M}:=\la \left(\left(\frac{n-2}{2}\right)^2+\inf_{x \in B_R'}\mu_1(a_M''(x,\cdot),0)\right)>0,
\end{equation}
where $a_M:=M^{-1}a $, and $a_M''$ denotes the last $n$ components of the vector $a_M$ (i.e. $p=a_M''$ and $h=0$). The constant $c_{n,a,M}$ is precisely the Hardy constant in Proposition \ref{prop_hardy}. Notice that, when $n=2$, the non-degeneracy means $\inf_{x \in B_R'}\mu_1(a_M''(x,\cdot),0)>0$.

There is a vast literature on magnetic operators. We just mention a few papers and lines of investigation: a comprehensive book \cite{fournais2010spectral}, existence of solutions and qualitative properties for nonlinear equations \cite{AS_exi_non_lin,BMS_3dmodel,CS_exi_reg,C_exi,CS_semi_cla,EL_exi_non_lin}, spectral theory and inequalities \cite{AFJT,AHS_II,AHS_III,R_asym,RES_ineq,RMW_poly,LW,MOR_neg_spectr,S_density}, 
unique continuation principles \cite{NZ_unique,FFT,K_unique}. There are also many results concerning spectral stability for a special class of magnetic potentials, the Aharonov-Bohm  potentials \cite{AFaharonov,AF-SIAM,AFL,AFLeigenvar,BNHHO,BNNT_AB_eigen,FNOS_multi,FNS,FRS_mag,L2015,NMT_AB_boun} (see Appendix \ref{sec:solenoid}), also in view of its connection with the problem of optimal partitions \cite{HHOT_minimal_part_magn}.

More closely related to the present work we mention \cite{kurata1997}, where regularity theory, e.g. boundedness and continuity   for potentials in a Kato class, is discussed. Furthermore, together with unique continuation principles by monotonicity  formulas, also $C^{0,\alpha}$ regularity for magnetic potentials with a Hardy-type singularity is obtained in \cite{FFT}. The regularity results of \cite{FFT} can be seen as a special case of Theorem \ref{theorem_C0alpha_main}  for $n=d$. We also acknowledge that the unique continuation in case of potentials with general codimensional singularities was studied in \cite{FFT_multi} for the Laplacian.

\subsection{Main results and techniques}
The main target of this study is the local Schauder theory for weak solutions to \eqref{eq_magnetic_strong}, in the sense of estimates of the H\"older norms of $u$ in a ball $B_r$ in terms of its $L^2$ norm and appropriate norms of the data $M,A$ in the bigger ball $B_R$ where the equation is satisfied ($0<r<R$).

As a first step, just requiring the boundedness of variable coefficients $M$, one can prove the $L^2-L^\infty$ estimate for $u$ by a Caccioppoli inequality and Moser iterations, see Proposition \ref{prop_caccio} and Corollary \ref{corollary_bound_Harn}.

Then, we provide the H\"older $C^{0,\alpha}$ and Schauder $C^{1,\alpha}$ estimates by scaling, in the spirit of Simon's work \cite{S_reg}. To our knowledge, this approach is new in the magnetic setting.

As a first step, we establish regularity in the case of subcritical magnetic potentials; that is, $a\equiv0$ and hence $A=b$, see Section \ref{sec_a=0}. When the coefficients $M$ are at least continuous, the $C^{0,\alpha}$ regularity for $u$ is proved in two steps: a priori estimates by a contradiction argument involving a blow-up procedure, and a regularization-approximation scheme with sequences of solutions with regularized data via standard mollifiers. Here $A$ needs to belong to $L^{d+\e}$, since the magnetic potential is treated as a lower order (yet subcritical) term that disappears in the blow-up procedure. In particular, $A$ acts both as a potential term and a drift term, as clarified in \eqref{def_L_Af}. The $C^{1,\alpha}$ regularity is deduced similarly, this time requiring $A,M\in C^{0,\alpha}$. 

We would like to remark that the blow-up argument relies on the validity of classic polynomial Liouville theorems for harmonic functions with sublinear or subquadratic growth, which are applied to both the real and imaginary parts of the limiting blow-up profile.

Once the regularity for the regular case is established, we also take care of the appearance of the singular term in the magnetic potential; that is, we deal with the case $a\not\equiv0$.

For what follows, it is convenient to write the matrix $N$ in blocks form as 
\begin{equation}\label{hp_M_structure}
N=
\begin{bmatrix}
N_1& N_2 \\
N_2^T & N_3
\end{bmatrix},
\quad 
N_2(x,0)=0,
\end{equation}
where $N_1$ is $(d-n) \times (d-n)$-dimensional, $N_2$ is $(d-n) \times n$-dimensional and $N_3$ is $n\times n$-dimensional. Notice that we require the vanishing on the $N_2$-block at $\Sigma_0$.

To achieve regularity, we apply a regularization-approximation scheme that works at the domain level, in the spirit of \cite{CFV,fio}. The local magnetic Sobolev capacity of $\Sigma_0$ is infinite if $n=2$ or zero if $n>2$, see Proposition \ref{prop_capa_ranges}. In both cases, solutions to \eqref{eq_magnetic_strong} can be approximated with families of solutions of the problem
\begin{equation}\label{prob_ue_total_intro}
\begin{cases}
L_{m} u_\e=0, & \text{ on } B_R \setminus \Sigma_{\e,N},\\
u_\e=0, & \text{ on } \partial\Sigma_{\e,N},\\
\end{cases}
\end{equation}
where the domain has been perforated around $\Sigma_0$, being
\begin{equation}
\Sigma_{\e,N}:=\{z=(x,y) \in \R^{d-n}\times \R^n:N_3^{-1}(x,y) y \cdot y \le \e^2\},
\end{equation}
and a homogeneous Dirichlet boundary condition has been prescribed on the boundary of the perforation
\begin{equation}
\partial \Sigma_{\e,N}:=\{z=(x,y) \in \R^{d-n}\times \R^n:N_3^{-1}(x,y) y \cdot y =\e^2\}.
\end{equation}
Notice that it is convenient to perforate the domain with holes whose geometry is determined by the lower-right block $N_3$ of $N$, as in \cite{CFV}.

The approximation with solutions to \eqref{prob_ue_total_intro} is possible since the same vanishing is expected on $\Sigma_0$ for the solutions to the original \eqref{eq_magnetic_strong}. In fact, this is particularly evident from an energetic point of view when the magnetic capacity of $\Sigma_0$ is infinite, but it is also true in the case of zero capacity, since any continuous solution must annihilate on $\Sigma_0$, a posteriori. Then, the regularity proved in the first step ($a\equiv0$) - extended up to the boundary $\partial\Sigma_{\e,N}$ - ensures regularity estimates for solutions to \eqref{prob_ue_total_intro} for any given fixed $\e>0$. Notice that far from the singular set the full potential $A$ behaves as the regular $b$ term. Then, one needs to prove that the constants in the regularity estimates are actually uniform as the parameter $\e\to0$ and the hole $\Sigma_{\e,N}$ shrinks onto $\Sigma_0$. This is done carefully in Section \ref{sec_Holder_estimates}. Both the $C^{0,\alpha}$ and $C^{1,\alpha}$ uniform estimates of Theorems \ref{theorem_reg_C0alpha} and \ref{theorem_reg_C1alpha} are proved by a very delicate blow-up procedure, and rely on the validity of the magnetic Liouville Theorem \ref{theor_liou_y}, which classifies the entire solutions to some Schr\"odinger equation having a bound on the growth rate at infinity. Recall that the singular term of $A$ possesses a critical scaling, and thus it persists after the blow-up. Then, the combined presence of $a$ and of the anisotropy $M$ dictates the admissible growth rates of the entire solutions, in terms of Laplace--Beltrami-type magnetic eigenvalues on $\mathbb S^{n-1}$. This phenomenon strongly affects the H\"older exponents in the local regularity estimates of the solutions.

To introduce our main results, let us consider
\begin{equation}\label{def_tildeax}
\tilde{a}(x,y/|y|):=\frac{N^{\frac{1}{2}}(x,0)(M^{-1}(x,0) a)(x,N_{3}^{\frac{1}{2}}(x,0)y/|y|)}{{|N_3^{\frac{1}{2}}(x,0)y/|y||}},
\end{equation}
and the characteristic exponent
\begin{equation}\label{def_gamma}
\gamma_1(M,a):=\inf_{x \in B_R'}\left\{-\frac{n-2}{2}+
\sqrt{\left(\frac{n-2}{2}\right)^2+\mu_1(\tilde{a}''(x,\cdot),|\tilde{a}'(x,\cdot)|^2)}\right\},
\end{equation}
where $\mu_1(\tilde{a}''(x,\cdot),|\tilde{a}'(x,\cdot)|^2)$ is as in \eqref{def:mu1}, with $p=\tilde{a}''(x,\cdot)$ and $h=|\tilde{a}'(x,\cdot)|^2$. The quantities $\tilde{a}'$ and $\tilde{a}''$ are, respectively, the first $d-n$ and the last $n$ components of the field $\tilde{a}$. 

Our main results can be stated as follows.

\begin{theorem}[H\"older $C^{0,\alpha}$ estimates]\label{theorem_C0alpha_main}
Let $M$ be a $(d\times d)$-dimensional matrix satisfying \eqref{hp_M_elliptic}, \eqref{hp_M_structure}. Let $A$ be as in \eqref{def_A}. Assume that \eqref{hp_transversality} and \eqref{hp_hardy} hold. Let $p>d$. Suppose that $\gamma_1(M,a) \ge \gamma_1^*$ for some $\gamma_1^*>0$ and  let 
\begin{equation}
\alpha\in (0,\gamma_1^*) \cap (0, 1-d/p].
\end{equation}
Assume that $M\in C^{0,\omega_1}(B_R,\R^{d,d})$ for some modulus of continuity $\omega_1$, $a \in C^{0,\omega_2}(B'_R\times\mathbb S^{n-1},\R^d)$ for some modulus of continuity $\omega_2$ and it satisfies \eqref{hp_transversality}, and that $b \in  L^q(B_R,\R^d)$. Let  $u$ be   a weak solution of \eqref{eq_magnetic_strong}.

Then $u \in C^{0,\alpha}_{loc}(B_R,\mb{C})$ and
\begin{equation}\label{eq_C0alpha_u0_main}
u=0 \quad \text{ on }\ \Sigma_0\cap B_R.
\end{equation}
Furthermore, if
\begin{equation}
\norm{M}_{C^{0,\omega_1}(B_R,\R^{d,d})}+\norm{a}_{C^{0,\omega_2}(B'_R\times\mathbb S^{n-1},\R^{d})}+\norm{b}_{L^q(B_R,\R^d)}\le L,
\end{equation}
then for any $r \in (0,R)$ there exists $C>0$, depending only on  $n,d, \la, \Lambda,q, L$, $r,R,\alpha,$ and $ \gamma_1^*$,  such that 
\begin{equation}
\norm{u}_{C^{0,\alpha}(B_r,\mb{C})} \le C \norm{u}_{L^2(B_R,\mb{C})}.
\end{equation}
\end{theorem}

\begin{theorem}[Schauder $C^{1,\alpha}$ estimates]\label{theorem_C1alpha_main}
Let $M$ be a $(d\times d)$-dimensional matrix satisfying \eqref{hp_M_elliptic}, \eqref{hp_M_structure}. Let $A$ be as in \eqref{def_A}. Assume that \eqref{hp_transversality} and \eqref{hp_hardy} hold.
Suppose that $\gamma_1(M,a) \ge \gamma_1^*$ for some $\gamma_1^*>1$ and let 
\begin{equation}
\alpha  \in (0,  \gamma_1^*-1)\cap (0,1).
\end{equation}
Assume that $M \in C^{0,\alpha}(B_R,\R^{d,d})$, $a \in C^{0,\bar \alpha}(B'_R\times\mathbb S^{n-1},\R^d)$ for some $\bar\alpha>\alpha$ and it satisfies \eqref{hp_transversality}, and that $b \in  C^{0,\alpha}(B_R,\R^d)$. Let $u$ be a weak solution of \eqref{eq_magnetic_strong}.

Then $u \in C^{1,\alpha}_{loc}(B_R,\mb{C})$ and 
\begin{equation}
u=0 \quad\text{ and } \quad \nabla u=0 \quad \text{ on } \ \Sigma_0\cap B_R.
\end{equation}
Furthermore, if
\begin{equation}
\norm{M}_{C^{0,\alpha}(B_R,\R^{d,d})}+\norm{a}_{C^{0,\bar \alpha}(B'_R\times\mathbb S^{n-1},\R^{d})}+ \norm{b}_{C^{0,\alpha}(B_R,\R^{d})} \le L,
\end{equation}
then for any $r \in (0,R)$ there exists $C>0$, depending only on   $n,d, \la, \Lambda, L, r,R,\alpha,\overline\alpha$, and  $\gamma_1^*$,  such that
\begin{equation}
\norm{u}_{C^{1,\alpha}(B_r,\mb{C})} \le C \norm{u}_{L^2(B_R,\mb{C})}.
\end{equation}
\end{theorem} 
More general equations can be treated, including lower-order terms in the right-hand side of \eqref{eq_magnetic_strong} and electric potentials with critical Hardy scaling. We omit this level of generality for the sake of readability, as it does not introduce additional difficulties.

\subsection{Physical motivations and applications}
We would like to end this introduction with an application to our theory, which is actually one of the main motivations for the present study. In the particular case $d=3$ and $n=2$, these equations arise naturally when studying three-dimensional Aharonov–Bohm-type (AB) models (see the classical papers \cite{AT,AB_orig}, and also the three-dimensional cylindrical model in \cite{bonheure2016nonlinear}), computing the physically relevant magnetic potential generated by a shrinking solenoid onto the axis $\Sigma_0$, with variable surface current density, and general geometry of the loops where the current circulates, see Appendix \ref{sec:A1}.  Additionally, anisotropy is natural when the solenoid windings are coiled along non-flat curves, see Appendix \ref{sec:A2}.

In such cases, due to asymmetries, anisotropies and variable currents, the problem cannot be reduced to/simplified with a two-dimensional model and therefore requires a general treatment. Moreover, as we shall show, the asymmetry has a twofold effect. On the one hand, it breaks the ideal AB configuration, in the sense that the magnetic field outside the solenoid is no longer vanishing. On the other hand, it yields an unexpected regularizing mechanism on the wave functions.

Let us review the construction of the ideal AB model: the infinite solenoid is made of loops, coiled around the axis $\Sigma_0=\{y_1=y_2=0\}$, and lying on planes which are orthogonal to $\Sigma_0$; that is, $\{x=t\}$ for any $t\in\R$. The circulating current is constant. Then, the AB effect concerns the fact that the magnetic field $B$ is confined inside the solenoid, and it is zero outside, even if the magnetic potential $A$ is not vanishing there, and moreover, it is strongly singular along $\Sigma_0$ as the loops shrink onto the axis. In this particular case, solutions to the two-dimensional AB model, constantly extended in the third variable $x\in\R$, are solutions to the three-dimensional model. Then, as proved in \cite{FFT} for the two-dimensional model, the optimal regularity of solutions is $C^{0,\alpha}$, for some $\alpha\in(0,1/2]$ which depends on the geometry of the loops and it is exactly $1/2$, and yet maximal, in the case of circular loops with centers on $\Sigma_0$, see also \cite{HHOHOO_circ12}.

Then, in Example \ref{ex_big_frequence}, we construct a particular infinite solenoid made of circular loops that lie on planes that are tilted from the orthogonal cross-sections by an angle $\beta\in(0,\pi/2)$. The circulating current is still constant. 

\begin{center}
\includegraphics[page=1,scale=1]{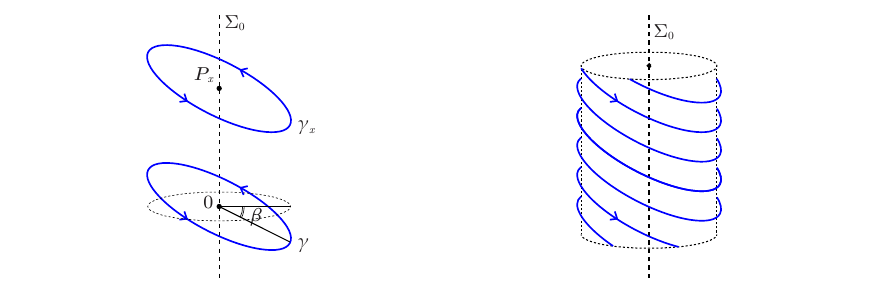}
\captionof{figure}{The picture describes the particular case of an infinite solenoid obtained by translating a planar circular loop $\gamma$ along the axis $\Sigma_0$. For each $x\in\R$, the circle $\gamma_x$ turns around $P_x=(x,0,0)$ and lies in a plane transverse to $\Sigma_0$, and tilted by an angle $\beta\in(0,\pi/2)$ with respect to the plane orthogonal to $\Sigma_0$.}\label{fig:0}
\end{center}

This significant example exhibits the main features that distinguish the asymmetric model from the ideal symmetric configuration. The exterior magnetic field is not vanishing, so the configuration is no longer in the ideal AB regime, and the regularizing mechanism is visible explicitly: the eigenvalue $\mu_1$ undergoes a positive shift, which is consistent with our $C^{1,\alpha}$ estimate in Theorem \ref{theorem_C1alpha_main}, and suggests that even higher Schauder regularity is expected in some cases. In fact, for any $k\in\mathbb N$, there exists $\beta\in(0,\pi/2)$ such that $\gamma_1(\beta)>k$.

\section{Functional Setting}\label{sec_functional_setting}
\subsection{Notations and preliminaries}
Let us define
\begin{equation}\label{def_nabla_A}
\nabla_mu=\nabla_{M,A} u:=iM\nabla u +Au, \quad \nabla_x u:=\left(\pd{u}{x_1},\dots, \pd{u}{x_{d-n}}\right), \quad \nabla_y u:=\left( \pd{u}{y_1},\dots, \pd{u}{y_{n}}\right).
\end{equation}
In the case $M={\rm{Id}_d}$, we write $\nabla_{A} u$ in place of $\nabla_mu$.
Let us fix some more notation. For any vector field $F=(f_1, \dots, f_d)$, we denote its components by
\begin{align}\label{def_'_''}
F'=(f_1, \dots, f_{d-n}),  \quad \text{ and }\quad F''=(f_{d-n+1}, \dots, f_d).
\end{align}
Let $\theta=(\theta',\theta'')$ be the  variable in $\mb{S}^{d-n-1}\times \mb{S}^{n-1} \subset  \R^{d-n}\times \R^n$, 
that is,    $(\theta',\theta''):=(x/|x|,y/|y|)$.
Furthermore, let  $g$  be  the metric on $\mb{S}^{n-1}$ induced by the flat metric on $\R^n$ and  $\nabla_{\mb{S}^{n-1}}$ the correspondent metric gradient.

For any $R>0$ and any $z=(x,y) \in \R^{d-n}\times \R^n$, let 
$B_R'(x)$, $B''_R(y)$, and $B_R(z)$ be the open balls of center $x,y$ or $z$ in  $\R^{d-n}$, $\R^n$ or $\R^d$ respectively.

In all the statement throughout the paper we will always assume the validity of \eqref{def_A},  \eqref{hp_M_elliptic}, \eqref{hp_transversality} and \eqref{hp_hardy} without further notice.

\begin{remark}\label{remark_ineq_magnetic_nablas}
We can easily compare the anisotropic and classical magnetic gradient pointwise.
 Indeed, defining $N=M^TM$, one has 
\begin{equation}
|iM\nabla \varphi+ A \varphi|^2=N (\nabla \varphi+ M^{-1}A \varphi) \cdot \overline{(\nabla \varphi+ M^{-1}A \varphi)}.
\end{equation}
Then 
\begin{equation}
\la  |i\nabla \varphi+ M^{-1}A \varphi|^2\le |iM\nabla \varphi+A \varphi|^2 
\le \Lambda |i\nabla \varphi+ M^{-1}A \varphi|^2.
\end{equation}  
\end{remark}

Thanks to Remark \ref{remark_ineq_magnetic_nablas}, we can prove an anisotropic version of the classical \emph{Diamagnetic inequality}, see for example \cite[Lemma A.1]{FFT}. We also state some useful algebraic identities.

\begin{proposition}[Diamagnetic inequality]\label{prop_ineq_dia}
In a pointwise sense
\begin{equation}\label{ineq_ani_dia}
\sqrt{\la}|\nabla |v||\le|M \nabla |v||\le |\nabla_{M,A}v|.
\end{equation}
Furthermore,
\begin{align}
&\nabla |v|=\Rea\left(\frac{\overline{v}}{|v|}\nabla v\right), &&\Rea(\overline{v} \nabla v)=|v|\nabla |v|,\label{eq_Re_nabla}\\
& \Ima(\overline{v}\nabla v)=\Rea(i(v \nabla \overline{v}- \overline{v} \nabla v)),
&&|iM\nabla v+Av|^2=|\nabla v|^2 -2 \Ima(\overline{v}M\nabla v) \cdot A+|A|^2 v^2.\label{eq_grad_magnetic}
\end{align}
\end{proposition}
\begin{proof}
Both \eqref{eq_Re_nabla} and \eqref{eq_grad_magnetic}
can be proved by direct computations. Furthermore,
\begin{align}
\la |\nabla |v||^2 \le|M \nabla |v||^2&=\left|M\Rea\left(\frac{\overline{v}}{|v|}\nabla v\right)\right|^2\\
&=\left|\Rea\left(\frac{\overline{v}}{|v|}(M\nabla v-iAu)\right)\right|^2 
\le |M\nabla v-iAu|^2=|\nabla_{M,A} v|^2
\end{align}
and so we have proved \eqref{ineq_ani_dia}.
\end{proof}

\subsection{Hardy-Maz'ya inequalities}\label{subsec_Hardy} 
In order to establish a functional setting for \eqref{eq_magnetic_strong}, we need some  Hardy-Maz'ya-type  inequalities  with boundary terms in the spirit of \cite[Lemma 3.1]{FFT} or \cite[Subsection 2.1.6, Corollary 3]{M_book} with respect to the anisotropic magnetic gradient.
Let
\begin{equation}\label{def_aM}
a_M:=M^{-1}a.
\end{equation}
We make use of polar coordinates, hence we need the following lemma.
\begin{lemma}\label{lemma_nabla_polar}
The anisotropic magnetic gradient with respect to the variable $y$ can be expressed in polar coordinates $(r,\theta'')$  of $\R^n$ as follows
\begin{equation}\label{eq_nabla_polar}
\left|i\nabla_y v +|y|^{-1}a_M''\left(x, \theta''\right)v\right|^2=\left|\pd{v}{r} \right|^2 +\frac{1}{r^2}\left|i\nabla_{\mb{S}^{n-1}}v  +  a_M''\left(x, \theta''\right)   v \right|^2.
\end{equation}
\end{lemma}
\begin{proof}
A direct computation yields, by \eqref{hp_M_elliptic} and \eqref{hp_transversality},
\begin{align}
&\left|i \nabla_y v +|y|^{-1}a_M''\left(x, \theta''\right)v\right|^2\\
&=|\nabla_y v|^2 +\frac{\left|a_M''\left(x, \theta''\right)\right|^2}{|y|^2}|v|^2+2\left(\Ima v  \nabla_y \Rea v  \cdot |y|^{-1}a_M''\left(x, \theta''\right)  -\Rea v  \nabla_y \Ima v  \cdot |y|^{-1}a_M''\left(x, \theta''\right)\right)\\
&=\left|\pd{v}{r} \right|^2  +\frac{2}{r}\left(\Ima v  \nabla_{\mb{S}^{n-1}} \Rea v  \cdot r^{-1}a_M''\left(x,  \theta''\right)  -\Rea v  \nabla_{\mb{S}^{n-1}} \Ima v  \cdot r^{-1}a_M''\left(x,  \theta''\right)\right) \\
&+\frac{\left|a_M''(x, \theta'')\right|^2}{|r|^2}|v|^2 +\frac{1}{r^2}|\nabla_{\mb{S}^{n-1}}v|^2 =\left|\pd{v}{r} \right|^2 +\frac{1}{r^2}\left|i\nabla_{\mb{S}^{n-1}}v  +  a_M''\left(x, \theta''\right)   v \right|^2.
\end{align}
Indeed, \eqref{hp_transversality} implies that 
\begin{equation}
 \nabla_y \Rea v  \cdot a_M''\left(x, \theta''\right)= (\partial_r\Rea v\,  \theta''+ r^{-1} \nabla_{\mb{S}^{n-1}} \Rea v)  \cdot a_M''\left(x, \theta''\right)=r^{-1} \nabla_{\mb{S}^{n-1}} \Rea v  \cdot a_M''\left(x, \theta''\right)
\end{equation}
and the same holds for $ \nabla_y \Ima v  \cdot a_M''\left(x, \theta''\right)$.

\end{proof}

Let $c_{n,a,M}$ be as in \eqref{hp_hardy} and recall that we are assuming $c_{n,a,M}>0$. This is particularly relevant for the next propositions, which otherwise would be trivial.

\begin{proposition}[Hardy inequality]\label{prop_hardy}
For any $\varphi \in  C_c^\infty(B_R,\mb{C})$
\begin{equation}\label{ineq_Hardy}
 c_{n,a,M}\int_{B_R}\frac{|\varphi|^2}{|y|^2} \, dz 
\le \int_{B_R} \left|iM\nabla \varphi +|y|^{-1}a(x,y/|y|)\varphi\right|^2\, dz.
\end{equation}
\end{proposition}
\begin{proof}
The integrals can be computed in $\R^d$, but are zero outside the support of $\varphi \in  C_c^\infty(B_R,\mb{C})$. By Remark \ref{remark_ineq_magnetic_nablas} we may estimate, considering only the last $n$ components of the anisotropic magnetic gradient,  
\begin{multline}
\int_{\R^d} \left|iM\nabla \varphi+|y|^{-1}a(x, y/|y|)\varphi \right|^2\, dz\\
\ge\la  \int_{\R^d} |i\nabla \varphi+ |y|^{-1}a_M(x, y/|y|) \varphi|^2 \, dz\\
\ge \la \int_{\R^{d-n}}\int_{\R^n} |i\nabla_y \varphi+ |y|^{-1}a''_M(x, y/|y|) \varphi|^2  \, dy \, dx.
\end{multline}
By Lemma \ref{lemma_nabla_polar}
\begin{multline}
\int_{\R^n} \left|i\nabla_y \varphi+|y|^{-1}a_M''(x, y/|y|)\varphi \right|^2\, dy  \\
 =\int_0^{+\infty} r^{n-1}\int_{\mb{S}^{n-1}}\left(\left|\pd{\varphi}{r} \right|^2\,
+\frac{1}{r^2}\left|i\nabla_{\mb{S}^{n-1}}\varphi  +  a_M''\left(x, \theta''\right)   \varphi \right|^2\right)\, d \sigma \, dr \\
\end{multline}
and as in \cite[Lemma 3.1]{FFT}
\begin{equation}
\int_0^{+\infty} r^{n-1}\int_{\mb{S}^{n-1}}\left|\pd{\varphi}{r} \right|^2 \, dr 
\ge  \frac{n-2}{2}\int_{\R^n}\frac{|\varphi|^2}{|y|^2} \, dz.
\end{equation}
Furthermore, letting $\mu_1(a_M''(x,\cdot),0)$ be as in \eqref{def:mu1},
\begin{multline}
\int_0^{+\infty} r^{n-3}\int_{\mb{S}^{n-1}}\left|i\nabla_{\mb{S}^{n-1}}\varphi + a_M''\left(x, \theta''\right) \varphi \right|^2   \,d \sigma \, dr\\
\ge\mu_1(a_M''(x,\cdot),0)
\int_0^{+\infty} r^{n-3}\int_{\mb{S}^{n-1}}|\varphi|^2 d \sigma \, dr
=\mu_1(a_M''(x,\cdot),0)\int_{\R^n}\frac{|\varphi|^2}{|y|^2} \, dy.
\end{multline}
Hence, integrating over $\R^{d-n}$, we have proved \eqref{ineq_Hardy}.  
\end{proof}

We can prove a version of the Hardy inequality on the $(d-1)$-dimensional unit sphere $\mb{S}^{d-1}$. 
\begin{proposition}[Hardy inequality on spheres]\label{prop_hardy_sphere}
For any $\varphi \in   C^\infty(\mb{S}^{d-1},\mb{C})$
\begin{equation}\label{ineq_Hardy_sphere}
 c_{n,a,M}\int_{\mb{ S}^{d-1}}\frac{|\varphi|^2}{|\theta''|^2} \, d\sigma 
\le \la \left(\frac{d-2}{2}\right)^2 \int_{\mb{S}^{d-1}}|\varphi|^2 \, d\sigma 
+\la\int_{\mb{S}^{d-1}}\left|i \nabla_{\mb{S}^{d-1}}\varphi  +   |\theta''|^{-1}a_M \,   \varphi \right|^2 d \sigma.
\end{equation}
\end{proposition}
\begin{proof}
Let $\varphi \in   C^\infty(\mb{S}^d,\mb{C})$ and $f \in C^\infty_c((0,R), \R)$, $f \neq 0$ and $u(z):=u(r,\theta)=\varphi(\theta) f(r)$.
From the proof of \eqref{ineq_Hardy},  we obtain, considering polar coordinates as in Lemma \ref{lemma_nabla_polar},
\begin{multline}
 c_{n,a,M}\int_0^{+\infty} r^{d-3}|f|^2\, dr \int_{\mb{ S}^{d-1}}\frac{|\varphi|^2}{|\theta''|^2} \, d\sigma 
\le  \la \int_0^{+\infty} r^{d-1}|f'|^2 \, dr \int_{\mb{S}^{d-1}}|\varphi|^2 \, d\sigma \\
+\la \int_0^{+\infty} r^{d-3}|f|^2 \, dr \int_{\mb{S}^{d-1}}\left|i\nabla_{\mb{S}^{d-1}}\varphi  
+ |\theta''|^{-1} a_M \, \varphi \right|^2 d \sigma.
\end{multline}
Since the above inequality holds for any  $f \in C^\infty_c((0,R), \R)$, by the optimality of the classical Hardy constant, 
we obtain \eqref{ineq_Hardy_sphere}.
\end{proof}

\begin{proposition}[Hardy inequality with boundary term]\label{prop_hardy_ball}
For any $\varphi \in C^\infty(\overline{B_R},\mb{C})$ 
\begin{equation}\label{ineq_Hardy_ball}
 c_{n,a,M}\int_{B_R}\frac{|\varphi|^2}{|y|^2} \, d\sigma 
\le \frac{\la(d-2)}{2R} \int_{\partial B_R}|\varphi|^2 \, d\sigma 
+\int_{B_R}\left|iM\nabla\varphi  +  |y|^{-1}a(x,y/|y|) \,   \varphi \right|^2 dz.
\end{equation}
\end{proposition}

\begin{proof}
It is not restrictive to prove \eqref{ineq_Hardy_ball} in the case $R=1$.
Let $\varphi \in C^\infty(\overline{B_1}, \mb{C})$. Passing in polar coordinates, by  Lemma \ref{lemma_nabla_polar},  the  Hardy inequality with boundary terms, and Proposition \ref{prop_hardy_sphere}
\begin{multline}
 c_{n,a,M}\int_{B_1}\frac{|\varphi|^2}{|y|^2} \, dz 
= c_{n,a,M} \int_0^1 \rho^{d-3} \int_{\mb{S}^{d-1}}\frac{|\varphi(\rho\theta)|^2}{|\theta '|^2} \, d\sigma d \rho \\
\le \la   \int_0^1 \rho^{d-3} \int_{\mb{S}^{d-1}}  \left(\left(\frac{d-2}{2}\right)^2|\varphi|^2 
+  \left|i\nabla_{\mb{S}^{d-1}}\varphi +|\theta''|^{-1}a_M \,   \varphi \right|^2 \right) d \sigma d \rho \\
=  \la \left(\frac{d-2}{2}\right)^2\int_{B_1}\frac{|\varphi|^2}{|z|^2} \, dz 
+ \la \int_0^1 \rho^{d-3} 
\int_{\mb{S}^{d-1}}\left(\left|i\nabla_{\mb{S}^{d-1}}\varphi+|\theta''|^{-1}a_M\,\varphi \right|^2\right) d \sigma d\rho\\
\le\la\frac{d-2}{2} \int_{\mb{S}^{d-1}} |\varphi|^2 \, d\sigma+\\
\la \int_0^1 \rho^{d-3} 
\int_{\mb{S}^{d-1}} \left(\rho^2 \left|\pd{\varphi}{\rho} \right|^2+
\left|i\nabla_{\mb{S}^{d-1}}\varphi+|\theta''|^{-1}a_M\,\varphi \right|^2\right) d \sigma d \rho \\
\le\la \frac{d-2}{2} \int_{\mb{S}^{d-1}} |\varphi|^2 \, d\sigma
+\int_{B_1}\left|iM\nabla\varphi+|y|^{-1}a(x,y/|y|)\,\varphi \right|^2  \, dz,
\end{multline}
thus proving \eqref{ineq_Hardy_ball}.
\end{proof}

\begin{proposition}[Equivalence of norms]\label{prop_equiv_norms}
There exist two positive constant $C_1,C_2>0$, depending only on $a,b,\la,\Lambda,d,n$ and $R$ such that for any $\varphi \in C^\infty(\overline{B_R}, \mb{C})$ 
\begin{equation}\label{ineq_equiv_norms}
C_1\int_{B_R} \left(|\nabla\varphi|^2+\frac{|\varphi|^2}{|y|^2}\right) \, dz 
\le \int_{B_R}\left( |\nabla_{M,A} \varphi |^2 +|\varphi|^2 \right)\, dz
\le  C_2\int_{B_R}  \left(|\nabla\varphi|^2+\frac{|\varphi|^2}{|y|^2}\right)\, dz.
\end{equation}
\end{proposition}
\begin{proof}
In view of \eqref{hp_M_elliptic}
\begin{equation}
\int_{B_R} |\nabla\varphi|^2 \, dz \le 
2\int_{B_R}\left( \Lambda|iM\nabla\varphi+A\varphi|^2 +|a|^2\frac{|\varphi|^2}{|y|^2}+|b|^2|\varphi|^2 \, dz\right).
\end{equation}
By \eqref{ineq_Hardy_ball}, classical Sobolev trace theory, and \eqref{ineq_ani_dia} 
\begin{multline}
 c_{n,a,M}\int_{B_R} 
\frac{|\varphi|^2}{|y|^2}\, dz
\le \int_{B_R} \left|iM\nabla\varphi  +  |y|^{-1}a(x,y/|y|) \,   \varphi \right|^2 \, dz +\frac{\la(d-2)}{2} \int_{\partial B_R}|\varphi|^2 \, d\sigma\\
\le 2\int_{B_R} (|\nabla_{M,A}\varphi|^2 +|b|^2 |\varphi|^2)\, dz +\frac{\la(d-2)}{2} \int_{\partial B_R}|\varphi|^2 \, d\sigma
\le C_1  \int_{B_R} (|\nabla_{M,A}\varphi|^2 +(|b|^2+1) |\varphi|^2)\, dz,
\end{multline}
for some constant $C_1>0$. Furthermore, by the classical Sobolev embedding theorem  and \eqref{ineq_ani_dia}, for $d>2$
\begin{multline}
\int_{B_R}|b|^2 |\varphi|^2 \, dz \le \left(\int_{B_R}|b|^d \, dz\right)^{\frac{d}{2}} 
\left(\int_{B_R}|\varphi |^{\frac{2d}{d-2}}\, dz\right)^\frac{d-2}{2}\\
\le C_2 \left(\int_{B_R}|b|^d \, dz\right)^{\frac{d}{2}} \int_{B_R}(|\nabla_{M,A} \varphi |^2 +|\varphi|^2)\, dz,
\end{multline}
for some positive constant $C_2>0$, while if $d=2$ we can prove a similar but simpler estimate.
Summing up, we have proved the first inequality in \eqref{ineq_equiv_norms}. The proof of the second inequality is similar.
\end{proof}

\subsection{Magnetic capacity and Sobolev Spaces}\label{subsec_magnetic_capa_Sobolev}
Let us define for any $R>0$ and any closed set $K$ the magnetic Sobolev capacity
\begin{equation}\label{def_capa_magnetic}
\capa{M,A}{K \cap B_R, \Omega }:=\inf\left\{\int_{\Omega}|\nabla_{M,A}\varphi|^2\, dz 
:\varphi \in C^\infty_c(\Omega,\mb{C}), |\varphi|=1 \text{ on }  K \cap B_R\right\},
\end{equation}
where $\Omega$ is any domain such that   $K \cap B_R \subset \Omega$.
Let us also define 
\begin{equation}\label{def_capa_magnetic_simple}
\overline{\mathop{\rm Cap}}(K \cap B_R, \Omega):=\inf\left\{\int_{\Omega}\left(|\nabla\varphi|^2+\frac{|\varphi|^2}{|y|^2}\right)\, dz 
:\varphi \in C^\infty_c(\Omega,\mb{C}), |\varphi|=1 \text{ on }  K \cap B_R\right\},
\end{equation}
and denote with 
\begin{equation}\label{def_capa_class}
\capa{}{K \cap B_R, \Omega }:=\inf\left\{\int_{\Omega}|\nabla \varphi|^2\, dz 
:\varphi \in C^\infty_c(\Omega), \varphi=1 \text{ on }  K \cap B_R\right\},  
\end{equation}
the classical Sobolev capacity in $\Omega$. 
By Proposition \ref{prop_equiv_norms}, the classical Poincaré inequality and \eqref{ineq_ani_dia}
\begin{equation}
C_1 \overline{\mathop{\rm Cap}}(K \cap B_R, \Omega)\le  \capa{M,A}{K \cap B_R, \Omega } \le C_2 \overline{\mathop{\rm Cap}}(K \cap B_R, \Omega)
\end{equation}
for some constants $C_1,C_2>0$.

\begin{proposition}\label{prop_capa_ranges}
Letting $\Sigma_0$ be as in   \eqref{def_Sigma},
\begin{enumerate}[(i)]
\item if $n = 2$ then $\capa{M,A}{\Sigma_0 \cap B_R, \Omega }=+\infty$,
\item if $n >2$  then $ \capa{M,A}{\Sigma_0 \cap B_R, \Omega}=0$.
\end{enumerate}
\end{proposition}
\begin{proof}
If $n = 2$ and $|\varphi|=1 \text{ on }  \Sigma_0 \cap B_R$, then it is easy to check, using polar coordinates in $\R^n$, that 
\begin{equation}
\int_{\Omega} \frac{|\varphi|^2}{|y|^2}  \, dz =+\infty.
\end{equation}
Hence, by \eqref{ineq_Hardy_ball}, also 
\begin{equation}
\int_{\Omega}|\nabla_{M,A}\varphi|^2\, dz =+\infty
\end{equation}
which implies  $\capa{M,A}{\Sigma_0 \cap B_R, \Omega }=+\infty$.

On the other hand, if $n > 2$,  then, by  \eqref{ineq_ani_dia}, \eqref{ineq_equiv_norms}, the classical Poincaré and Hardy-Maz'ya inequalities (see \cite[Subsection 2.1.6, Corollary 3]{M_book}),
\begin{equation}
 \int_{\Omega} |\nabla_{M,A} \varphi |^2\, dz\le  c\int_{\Omega} |\nabla\varphi|^2\, dz,   
\end{equation}
for a positive constant $c>0$ that do not depend on $\varphi$.
Clearly
\begin{equation}
\inf\left\{\int_{\Omega}|\nabla \varphi|^2\, dz :\varphi \in C^\infty_c(\Omega,\mb{C}), |\varphi|=1 \text{ on }  \Sigma_0 \cap B_R\right\} 
\le   \capa{}{\Sigma_0 \cap B_R, \Omega}.
\end{equation}
Since  $\mc{H}^{d-2}(\Sigma_0 \cap B_R)=0$ it is a classical fact that also $\capa{}{\Sigma_0 \cap B_R, \Omega }=0$, for example arguing as in \cite[Theorem 4.16]{EG_book}. Hence, we conclude that $ \capa{M,A}{\Sigma_0 \cap B_R, \Omega}=0$.
\end{proof}

For any $\varphi \in C^\infty(\overline{B_R}, \mb{C})$ let 
\begin{equation}\label{def_norm_H1A}
\norm{\varphi}_{H^1_{M,A}(B_R)}^2:=\int_{B_R}\left(|\varphi|^2+|\nabla_{M,A}\varphi|^2\right) \, dz
\end{equation}
and let $H^1_m(B_R)=H^1_{M,A}(B_R)$ be the completion of 
\begin{equation}
\{\varphi \in C^\infty(\overline{B_R},\mb{C}): \norm{\varphi}_{ H^1_{M,A}(B_R)}<+\infty\}
\end{equation}
with respect to the norm in \eqref{def_norm_H1A}. 

Motivated by Proposition \ref{prop_equiv_norms}, let us also define the space
$\tilde{H}^1(B_R,\mb{C})$ as  the completion of 
\begin{equation}
\left\{\varphi \in C^\infty(\overline{B_R},\mb{C}): \norm{\varphi}_{\tilde{H}^1(B_R,\mb{C})}<+\infty\right\}
\end{equation}
with respect to the norm 
\begin{equation}\label{def_norm_H10Sigma}
\norm{\varphi}_{\tilde{H}^1(B_R,\mb{C})}^2:=\int_{B_R}\left(\frac{|\varphi|^2 }{|y|^2}+|\nabla\varphi|^2\right) \, dz.
\end{equation}

\begin{proposition}\label{prop_H1MA_H1}
For any $n \ge 2$
\begin{equation}\label{eq_H1MA_H1}
H^1_{M,A}(B_R)=\tilde{H}^1(B_R,\mb{C}).
\end{equation}
Furthermore,
\begin{align}
\tilde{H}^1(B_R,\mb{C})=H^1(B_R,\mb{C}), & \quad \text{ if } n>2,\\
\tilde{H}^1(B_R,\mb{C})\subset H^1(B_R,\mb{C}), & \quad \text{ if } n=2.
\end{align}
\end{proposition}
\begin{proof}
 Notice that if $n=2$ the constant function $1$ belongs to $H^1(B_R,\mb{C})$ but not to $\tilde{H}^1(B_R,\mb{C})$ and therefore the inclusion is strict.
All other claims are an easy consequence of Proposition \ref{prop_equiv_norms} and the classical Hardy-Maz'ya inequality, see for example \cite[Subsection 2.1.6, Corollary 3]{M_book}.
\end{proof}

Arguing as in \cite[Proposition 2.5]{CFV}, we can prove the following density result.
\begin{proposition}\label{prop_density_0}
If $n=2$ then $\tilde{H}^1(B_R,\mb{C})$ corresponds to the completion of
\begin{equation}
\{\varphi \in C^\infty(\overline{B_R}\setminus \Sigma,\mb{C}): \varphi=0 \text{ in a neighbourhood of } \Sigma\}
\end{equation}
with respect to the norm in \eqref{def_norm_H10Sigma}.
\end{proposition}

\begin{proposition}\label{prop_density_n=2_H=W}
If $n>2$ then $\tilde{H}^1(B_R,\mb{C})$ corresponds to 
\begin{equation}
W^{1,\, 2}(B_R,\mb{C}):=\{v \in W^{1,1}_{loc}(B_R,\mb{C}): \norm{v}_{\tilde{H}^1(B_R,\mb{C})}<+\infty\}.
\end{equation}
\end{proposition}
\begin{proof}
The inclusion $\tilde{H}^1(B_R,\mb{C}) \subset W^{1,\, 2}(B_R,\mb{C}) $ is trivial. Let us prove the other inclusion. Let $u \in  W^{1,\, 2}(B_R,\mb{C})$.
We are going to approximate $u$ with a family ${u_h} \in W^{1,1}_{loc}(B_R,\mb{C})$ with support disjoint from $\Sigma$. Then, by a standard mollification, we may approximate each ${u_h}$ with a smooth function $\varphi_h \in C_c^\infty(\overline{B_R}\setminus \Sigma,\mb{C})$ thus concluding the proof.

To this end, 
let us define for any $h \in \mathbb{N} \setminus \{0\}$, the cut-off functions
\begin{equation}\label{def_Psi_h}
\Psi_h(x,y):=1-\eta_h(y), \quad \text{ with }\quad \eta_h(y):= 1-\eta\left(-\frac{\log|y|}{h}\right),
\end{equation}
where $\eta  \in C^{\infty}(\R)$, $0\le \eta \le 1$,  $\eta(t)=1$  if $t\in (2,+\infty)$ and  $\eta(t)=0$  if $t\in (-\infty,1)$.
It is easy to check that $0\le\Psi_h\le1$, $\Psi_h=0$ in $B_R \cap \{|y|\le e^{-2h}\}$ and $\Psi_h=1$ in $B_R \cap \{|y|\ge e^{-h}\}$.
Furthermore,
\begin{equation}\label{proof_density_n=2_H=W_1}
|\nabla \Psi_h|= \left|\frac{\eta'\left(-\frac{\log|y|}{h}\right)}{h |y|^2}(0,y)\right| \le \frac{\norm{\eta'}_{L^\infty(\mb{R})}}{h|y|}.
\end{equation}
Let us define $u_h:= u \Psi_h$. Then  $u_h \in W^{1,1}_{loc}(B_R)$ and 
\begin{equation}
\lim_{h \to +\infty}\int_{B_R}\frac{|u_h-u|^2}{|y|^2} dz=\lim_{h \to +\infty} \int_{B_R \cap\{|y|\le e^{-h}\} }\frac{|u|^2|1-\Psi_h|^2}{|y|^2} dz=0,
\end{equation}
by the Dominated Convergence Theorem.
Furthermore by \eqref{proof_density_n=2_H=W_1} and the Dominated Convergence Theorem
\begin{equation}
\lim_{h \to +\infty}\int_{B_R} |\nabla u- \nabla u_h|^2 \, dz= \lim_{h \to +\infty} \int_{B_R} |1-\Psi_h|^2|\nabla u|^2+|u|^2 |\nabla \Psi_h|^2  \, dz=0.
\end{equation}
Hence, we have proved that $u_h \to u$ strongly with respect to the norm in \eqref{def_norm_H10Sigma}.
\end{proof}

Since in Section \ref{sec_approx_domains} we are going to deal with Dirichlet boundary conditions, for any compact set $K\subset B_R$, let us also define $\tilde{H}_{0,K}^1(B_R,\mb{C})$ as the completion of
\begin{equation}
\left\{\varphi \in C_c^\infty(\overline{B_R}\setminus K,\mb{C}): \norm{\varphi}_{\tilde{H}^1(B_R,\mb{C})}<+\infty\right\}
\end{equation}
with respect to the  $\tilde{H}^1(B_R,\mb{C})$-norm. Equivalently
\begin{equation}\label{Hqe}
\tilde{H}_{0,K}^1(B_R,\mb{C})= \{v \in \tilde{H}^1(B_R,\mb{C}): v=0 \text{ q.e. on } K \},
\end{equation}
where the q.e. condition means quasi-everywhere in the sense of \eqref{def_capa_magnetic_simple}; that is, up to sets of zero magnetic capacity.

\begin{proposition}[Poincar\'e inequality]\label{prop_poinc_boundary}
There exists a constant $C>0$, depending only on  $a,b,\la,\Lambda,d,n$ and $R$, such that
\begin{equation}\label{ineq_poinc_partial_BR}
\int_{B_R} |v|^2 \, dz \le C \int_{B_R} |iM\nabla v+A v|^2 \, dz \quad  \text{ for any }  v\in\tilde{H}_{0, \partial B_R}^1(B_R,\mb{C}).
\end{equation}
\end{proposition}
\begin{proof}
Since $\tilde{H}_{0, \partial B_R}^1(B_R,\mb{C}) \subset H^1(B_R,\mb{C})$, the embedding into $L^2(B_R)$ is compact. Hence, by Proposition \ref{prop_equiv_norms}, we may prove \eqref{ineq_poinc_partial_BR} with a standard contradiction argument.
\end{proof}

\section{Weak solutions and $L^\infty$-bounds}\label{sec_weak_sol_bounded}
In this section, we define the weak solutions of \eqref{eq_magnetic_strong} and obtain $L^\infty$-bounds via a Moser iteration scheme.

We say that $u \in \tilde{H}^1(B_R,\mb{C})$ is a weak solution of \eqref{eq_magnetic_strong} in $B_R$ if
\begin{equation}\label{eq_magnetic_weak}
\int_{B_R} \nabla_{M,A} u \cdot \overline{\nabla_{M,A} v} \, dz=0, \quad  \text{ for any } v \in  \tilde{H}^1_0(B_R,\mb{C}).
\end{equation}
The definition is well posed in view of Subsection \ref{subsec_magnetic_capa_Sobolev}.
The following Cacciopoli-type inequality is the key result of this section.

\begin{proposition}[Caccioppoli inequality]\label{prop_caccio}
Let $u$ be a weak solution of \eqref{eq_magnetic_strong}. Then, for any $r \in (0,R)$ and  for any  $t\in [1,+\infty)$,
\begin{equation}
\text{ if } u \in L^{2t}(B_R,\mathbb C) \quad \text{ then } \quad  |u|^t \in H^1(B_r).
\end{equation}
Furthermore, for any $\tau > 1$
\begin{equation}\label{ineq_caccio}
\int_{B_R}   |\nabla( |u|^t \eta^\tau)|^2 \, dz
\le  2\tau^2 \left(1+4\frac{\Lambda^2}{\la^2}t^2\right)\frac{1}{(R-r)^2}\int_{B_R}     |u|^{2t}  \eta^{2\tau-2}\, d z,
\end{equation}
 where $\eta\in C^{\infty}_c(B_R)$ is a cut-off function  such that $0\leq\eta\leq1$, $\eta=1$ in $B_r$ and $|\nabla \eta| \le 2(R-r)^{-1}$. 
\end{proposition}

\begin{proof}
Let $t\in  [1,+\infty)$, $\tau > 1$,  and $\eta\in C^{\infty}_c(B_R)$ be the cut-off function of the statement. Clearly
\begin{equation}\label{proof_prop_caccio_1}
|\nabla( |u|^t \eta^\tau)|^2 \le 2 [t^2 |u|^{2t-2}\eta^{2 \tau} |\nabla |u||^2+\tau^2|u|^{2t} \eta^{2 \tau-2}|\nabla \eta|^2].
\end{equation}
In order to deal with $ |u|^{2t-2}\eta^{2 \tau} |\nabla |u||^2$, let us consider the test function 
$\varphi \in  \tilde{H}_0^1(B_R , \mb{C})$ defined by
\begin{equation}\label{proof_prop_caccio_2}
\varphi:= (T_k|u|)^{2t-2} u \eta^{2\tau},
\end{equation}
where, for any $k \in \mathbb{N}\setminus \{0\}$, $T_k |u|:= \min\{k, |u|+k^{-1}\}$.
Then
\begin{equation}\label{proof_prop_caccio_3}
\nabla  \varphi= (2t-2)(T_k|u|)^{2t-3} u \eta^{2\tau} \nabla (T_k|u|)+ (T_k|u|)^{2t-2} \eta^{2\tau} \nabla  u 
+2 \tau (T_k|u|)^{2t-2} u \eta^{2\tau-1} \nabla \eta.
\end{equation}
Testing \eqref{eq_magnetic_weak} with $\varphi$ we obtain
\begin{multline}
\int_{B_R}   
[(2t-2)(T_k|u|)^{2t-3} \overline{u} \eta^{2\tau}  N\nabla(T_k|u|) \cdot \nabla u
+ (T_k|u|)^{2t-2} \eta^{2\tau}  N\nabla  u \cdot \nabla \overline{u}\\
+2 \tau (T_k|u|)^{2t-2}\overline{u} \eta^{2 \tau -1}   N\nabla u \cdot \nabla \eta
-i(T_k|u|)^{2t-2} \overline{u} \eta^{2\tau} M^TA \cdot \nabla u\\
+iM^TA \cdot [(2t-2)(T_k|u|)^{2t-3} |u|^2 \eta^{2\tau} \nabla(T_k|u|)+ (T_k|u|)^{2t-2} u \eta^{2\tau} \nabla \overline{u} 
+2 \tau (T_k|u|)^{2t-2} |u|^2 \eta^{2\tau-1} \nabla \eta] \\
+|A|^2 (T_k|u|)^{2t-2} |u|^2 \eta^{2\tau}] \, d z=0.
\end{multline}
Considering the real part and adding $\int_{B_R}   (T_k|u|)^{2t-3} N |u|\eta^{2\tau} \nabla(T_k|u|) \cdot \nabla |u| \, dz$ on both sides we conclude that
\begin{multline}
(2t-1)\int_{B_R}   
(T_k|u|)^{2t-3}  |u|\eta^{2\tau}N\nabla(T_k|u|) \cdot \nabla |u| \,dz\\
=\int_{B_R}    [ (T_k|u|)^{2t-3}  |u|\eta^{2\tau} N\nabla(T_k|u|) \cdot \nabla |u|
- (T_k|u|)^{2t-2} \eta^{2\tau} |M\nabla  u+iAu|^2]\, d z\\
-2 \int_{B_R} \tau (T_k|u|)^{2t-2}|u| \eta^{2 \tau -1} N \nabla |u| \cdot \nabla \eta \, d z,
\end{multline}
thanks to  \eqref{eq_Re_nabla}.
Hence by \eqref{hp_M_elliptic} and \eqref{ineq_ani_dia}
\begin{equation}\label{proof_prop_caccio_4}
\la\int_{B_R}   
(T_k|u|)^{2t-3}  |u|\eta^{2\tau}  \nabla(T_k|u|) \cdot \nabla |u| \,dz\\
\le 2\Lambda \tau \int_{B_R}    [(T_k|u|)^{2t-2}|u| \eta^{2 \tau -1}  |\nabla |u|| |\nabla \eta|\, d z,
\end{equation}
and by the Young inequality 
\begin{multline}\label{proof_caccio_4.5}
2  \tau\int_{B_R}    (T_k|u|)^{2t-2}|u| \eta^{2 \tau -1}  |\nabla |u|| |\nabla \eta|  \, dz \\
\le \frac{\lambda}{2\Lambda} \int_{B_R}    (T_k|u|)^{2t-2}\eta^{2 \tau }  |\nabla |u||^2  \, dz
+  \frac{2\Lambda}{\la}\tau^2\int_{B_R}    |u|^{2t} \eta^{2 \tau -2}  |\nabla \eta|^2  \, dz.
\end{multline}
Hence,
\begin{equation}
\int_{B_R} |u|^{2t-2} \eta^{2\tau} |\nabla |u||^2 \,dz \\
\le \frac{4\Lambda^2}{\la^2}\tau^2\int_{B_R}    |u|^{2t} \eta^{2 \tau -2}  |\nabla \eta|^2  \, dz,
\end{equation}
and so, by \eqref{proof_prop_caccio_1}, it follows that
\begin{multline}
\int_{B_R}   |\nabla |u|^t \eta^\tau|^2 \, dz
\le  2\int_{B_R}     [t^2 |u|^{2t-2}\eta^{2 \tau} |\nabla |u||^2+\tau^2|u|^{2t} \eta^{2 \tau-2}|\nabla \eta|^2]\, dz\\
\le 2\tau^2 \left(1+4\frac{\Lambda^2}{\la^2}t^2\right)\int_{B_R}    |u|^{2t} \eta^{2 \tau -2}  |\nabla \eta|^2  \, dz,
\end{multline}
that is, we have proved \eqref{ineq_caccio}.
\end{proof}

In view of Proposition \ref{prop_H1MA_H1}, we may use  a classical Moser iteration scheme, see for example \cite[Chapter 3]{Beck}, to obtain the boundedness of  weak solutions.
\begin{corollary}\label{corollary_bound_Harn}
There exists a constant $\kappa>0$, depending only on  $d, \la$ and $\Lambda$, such that for any solution $u$ of \eqref{eq_magnetic_strong}
\begin{equation}\label{ineq_u_bounded}
\norm{u}_{L^\infty(B_r,\mathbb C)} \le \frac{\kappa}{(R-r)^d}\norm{u}_{L^2(B_R,\mathbb C)} \quad \text{ for any } r \in (0,R).
\end{equation}
\end{corollary}

We conclude this section with a Caccioppoli-type inequality with lower order terms that is needed in the following sections, see for example the proof of Theorem \ref{prop_blow_up_holder_a=0} below.

\begin{proposition}\label{prop_caccio_simple}
Suppose that $u \in H^1(B_R,\mb{C})$ is a weak solution of the equation
\begin{equation}
\int_{B_R} \left[\nabla_{M,A} u \cdot \overline{\nabla_{M,A} v}  +Vu\overline{v}\right]\, dz=\int_{B_R}  F \cdot \nabla  \overline{v}+ (f_1+f_2)  \overline{v}\, dz \quad \text{ for any } v  \in \tilde{H}_0^1(B_R,\mb{C}),
\end{equation}
where $F \in L^2(B_R,\R^d)$, $f_1 \in L^2(B_R)$ and $f_2 \in L^q(B_R)$, with $q>d/2$. Then, for any $r \in (0,R)$ 
\begin{multline}\label{ineq_caccio_simple}
\norm{\nabla u}_{L^2(B_r,\mb{C})}
\le  C\Bigg(\frac{1}{(R-r)^2}\norm{u}_{L^2(B_R,\mb{C})}+\norm{F}_{L^2(B_R,\R^d)}\\
+\norm{f_1}_{L^2(B_R)}
+\norm{f_2}_{L^q(B_R)}\norm{u}_{L^{q'}(B_R,\mb{C})}\Bigg),
\end{multline}
 where $q'=q/(q-1)$, for some constant $C>0$ depending only on  $d$, $\la$ and $\Lambda$.
\end{proposition}

\begin{proof}
Let $\eta\in C^{\infty}_c(B_R)$ be a cut-off function such that $0\leq\eta\leq1$, $\eta=1$ in $B_r$ and $|\nabla \eta| \le 2(R-r)^{-1}$. Then 
\begin{equation}\label{proof_prop_caccio_simple_1}
|\nabla( u \eta)|^2 \le 2 [ \eta^2 |\nabla u|^2+|u|^2  |\nabla \eta|^2].
\end{equation}
Let us consider the test function $\varphi \in  \tilde{H}_0^1(B_R , \mb{C})$ defined as
\begin{equation}\label{proof_prop_caccio_simple_2}
\varphi:=  u \eta^2,
\end{equation}
so that
\begin{equation}\label{proof_prop_caccio_simple_3}
\nabla  \varphi= \eta^2 \nabla  u +2 u \eta \nabla \eta.
\end{equation}
Testing \eqref{eq_magnetic_weak} with $\varphi$ we obtain
\begin{multline}
\int_{B_R}   \Bigg[ \eta^2  N\nabla  u \cdot \nabla \overline{u}+2 \overline{u} \eta   N\nabla u \cdot \nabla \eta+i \overline{u} \eta^2 M^TA \cdot \nabla u\\
-iM^TA \cdot [u \eta^2 \nabla \overline{u} +2   |u|^2 \eta \nabla \eta] +|A|^2  |u|^2 \eta^2  \Bigg] \, d z\\
=\int_{B_R}  \eta^2 F \cdot \nabla  \overline{u}+ 2\overline{u} \eta F \cdot \nabla  \eta + (f_1+f_2)\overline{u} \eta^2\,dz.
\end{multline}
Considering the real part,  
\begin{multline}
\int_{B_R} \eta^2  N  \nabla u \cdot \nabla \overline{u} \, dz
\le -\int_{B_R}    \left[\eta^2 |M\nabla  u+iAu|^2 \right]\, d z -4 \int_{B_R}  |u| \eta N \nabla u \cdot \nabla \eta \, d z\\
+\int_{B_R}  \eta^2 |F| |\nabla u|+ 2\overline{u} \eta  |F| |\nabla  \eta| + (|f_1|+|f_2|) |u| \eta^2\,dz,
\end{multline}
thanks to  \eqref{eq_Re_nabla}.
It follows that, by \eqref{hp_M_elliptic} and \eqref{ineq_ani_dia},
\begin{equation}\label{proof_prop_caccio_simple_4}
\la\int_{B_R} \eta^2 |\nabla u|^2 \, dz
\le 4 \Lambda\int_{B_R}  |u| \eta  |\nabla u| |\nabla \eta| \, d z
+\int_{B_R}  \eta^2 |F| |\nabla u|+ 2|u| \eta  |F| |\nabla  \eta|  + (|f_1|+|f_2|) \eta^2\,dz.
\end{equation}
Let us fix $\e>0$ to be chosen small later on. By the Young inequality 
\begin{equation}\label{proof_caccio_simple_4.5}
\int_{B_R} |u| \eta|\nabla u|| |\nabla \eta|  \, dz \le \e \int_{B_R}    \eta^2  |\nabla u|^2  \, dz+  \e^{-1}\int_{B_R}    |u|^2  |\nabla \eta|^2  \, dz.
\end{equation}
Similarly
\begin{equation}
\int_{B_R}  \eta^2 |F| |\nabla u|dz \le  \e \int_{B_R}    \eta^2  |\nabla u|^2  \, dz+  \e^{-1}\int_{B_R}   \eta^2|F|^2  \, dz.
\end{equation}
Furthermore, applying the Young inequality again,
\begin{equation}
\int_{B_R}  |u| \eta  |F| |\nabla  \eta|dz \le \frac{1}{2(R-r)}  \int_{B_R}    |u|^2  \, dz+ \frac{1}{2(R-r)} \int_{B_R}   |F|^2  \, dz,
\end{equation}
and 
\begin{equation}
\int_{B_R}   |f_1| |u| \eta^2\le \frac{1}{2}  \int_{B_R} |u|^2 \eta^2  \, dz+ \frac{1}{2} \int_{B_R}   |f_1|^2  \, dz. 
\end{equation}
Letting $1/q+1/q'=1$,  by the H\"older inequality
\begin{equation}
\int_{B_R}   |f_2| |u| \eta^2\le  \norm{f_2}_{L^q(B_R)} \norm{u}_{L^{q'}(B_R,\mb{C})}.
\end{equation}
In conclusion, taking $\e>0$  small enough, 
\begin{multline}
\int_{B_R}\eta^2 |\nabla (\eta u)|^2\, dz\le 2\int_{B_R}[\eta^2 |\nabla u|^2+|u|^2 |\nabla \eta|^2]\, dz\\
\le C\left(\frac{1}{(R-r)^2}\int_{B_R} |u|^2   \, dz + \int_{B_R}   |F|^2  \, dz+ \norm{f_2}_{L^q(B_R)} \norm{u}_{L^{q'}(B_R,\mb{C})}+\int_{B_R}   |f_2|^2  \, dz\right),
\end{multline}
for some constant $C>0$ depending only on $\la, \Lambda$ and $ d$, that is, we have proved \eqref{ineq_caccio_simple}.
\end{proof}

\section{Regularity in the case $a=0$}\label{sec_a=0}
In this section we discuss interior and boundary regularity (under Dirichlet boundary conditions) in H\"older spaces for solutions of \eqref{eq_magnetic_strong} in the simpler case $a=0$, so that $A=b$ (see \eqref{def_A}).

Concerning interior regularity, we have the following results.

\begin{proposition}\label{prop_blow_up_holder_a=0}
Let $R>0$ and $q>d$. Let $\alpha \le 1-d/q$. Suppose that $b \in L^q(B_R,\R^d)$ with $\|b\|_{L^q(B_R,\R^d)}\leq m$, $M$ is continuous with a modulus of continuity $\omega$ and $\norm{M}_{C^{0,\omega}(B_R,\R^{d,d})}\leq L$. 

Then, given $r\in(0,R)$, if $u$ is a weak solution of \eqref{eq_magnetic_strong}, there exists a constant $c>0$ depending only on the dimension $d$, $R$, $r$, $q$, $\alpha$, $\la$, $\Lambda$, $L$ and $m$  such that
\begin{equation}\label{ineq_u_holder_a=0}
\norm{u}_{C^{0,\alpha}(B_r,\mb{C})} \le c \norm{u}_{L^2(B_R,\mb{C})}.
\end{equation}
\end{proposition}

\begin{proposition}\label{prop_blow_up_C1_holder_a=0}
Let $R>0$. Let $\alpha \in (0,1)$. Suppose that $b \in C^{0,\alpha}(B_R,\R^d)$ with $\|b\|_{C^{0,\alpha}(B_R,\R^d)}\leq m$ and  that  $M \in C^{0,\alpha}(B_R,\R^{d,d})$   with  $\norm{M}_{C^{0,\alpha}(B_R,\R^{d,d})}\leq L$.

Then, if $u$ is a weak solution of \eqref{eq_magnetic_strong}, there exists a constant $c>0$ depending only on $d, R, r, \alpha, \la, \Lambda$, $m$ and $L$, such that
\begin{equation}\label{ineq_u_C1_holder_a=0}
\norm{u}_{C^{1,\alpha}(B_r,\mb{C})} \le c \norm{u}_{L^2(B_R,\mb{C})}.
\end{equation}
\end{proposition}

We are not going to provide a detailed proof of Propositions \ref{prop_blow_up_holder_a=0}-\ref{prop_blow_up_C1_holder_a=0},  since, if $a=0$, we can proceed as in  the classical non-magnetic  elliptic setting, see for example \cite{S_reg} and \cite{V_reg}. Furthermore, in Section \ref{sec_Holder_estimates}, we are going to present a detailed argument in the more general case $a \neq 0$. Nevertheless, let us discuss the main ideas behind the proof of Proposition \ref{prop_blow_up_holder_a=0} (being the proof of Proposition \ref{prop_blow_up_C1_holder_a=0} fairly similar). 

The first step is to obtain a priori estimates as in \eqref{ineq_u_holder_a=0}  by the means of a blow-up argument and supposing that $u \in C^{0,\alpha}(B_r,\mb{C})$. More precisely, we argue by contradiction assuming that there exists a sequence of solutions $\{u_k\}$ such that 
\begin{equation}
\norm{u_k}_{C^{0,\alpha}(B_r,\mb{C})} \ge k \norm{u_k}_{L^2(B_R,\mb{C})}.
\end{equation}
Then we show that a suitable scaling of $u_k$ converge to a limit $u$ which is defined in the whole of $\R^d$ and it is globally H\"older continuous, $u \not \equiv 0$ and sublinear in the sense that $|u(z)| \le |z|^\alpha$ globally. Since $a=0$, the magnetic potential is subcritical with respect to the scaling of the Laplacian and so passing to the limit in \eqref{eq_magnetic_weak} we also conclude that both the real and imaginary part of $u$ are entire harmonic functions. Then, classic polynomial Liouville theorems force $u\equiv 0$, which is a contradiction.   

Once the a priori estimates have been established, we consider the regularized problems
\begin{equation}
\begin{cases}
(iM_k\nabla +b_k)^2w=0, &\text{ in } B_r\\
w=u, &\text{ on } \partial B_r,    
\end{cases}
\end{equation}
where $M_k$ and $b_k$ are smooth and converge to $M$ and $b$ respectively.
The existence and uniqueness of solutions follows by standard variational arguments, while the regularity of solutions is a consequence of elliptic regularity theory for linear systems, see for example \cite{GM_reg_book},  since $\Rea w$ and $\Ima w$ solve the system
\begin{equation}
\begin{cases}
-\dive(N_k\nabla \Rea w) -2M_kb_k \cdot \nabla \Ima w -M_k\nabla \cdot b_k \Ima w+|b_k|^2 \Rea w=0, &\text{ in } B_r,\\
-\dive(N_k\nabla \Ima w) -2M_kb_k \cdot \nabla \Rea w -M_k\nabla \cdot b_k \Rea w+|b_k|^2 \Ima w=0, &\text{ in } B_r,\\
\Rea w=\Rea u, \;\Ima w=\Ima u &\text{on } \partial B_r.\\
\end{cases}
\end{equation}
Finally, we can show that any solution $u$ of \eqref{eq_magnetic_strong} can be approximated with solutions of the  regularized problems. Since \eqref{ineq_u_holder_a=0} holds for solutions of the regularized problems and is stable with respect to this regularization (by the a priori regularity estimates) we can pass to the limit  as $k \to \infty$ and conclude that $u \in C^{0,\alpha}(B_r,\mb{C})$  and \eqref{ineq_u_holder_a=0} holds for general weak solutions.

We refer to \cite{V_reg} for all the details of this approach to regularity in the case of the classical elliptic setting.

Next we turn our attention to boundary regularity under Dirichlet boundary conditions. These regularity results are needed for Section \ref{sec_approx_domains}, which in turn introduces the approximating problems on which the proof of Theorems \ref{theorem_C0alpha_main}-\ref{theorem_C1alpha_main} is based.

Let us consider a bounded domain $\Omega\subset \R^d$, $T\subset \partial \Omega$, $T$ open in $\partial \Omega$  and the following problem
\begin{equation}
\begin{cases}
L_mu=0, &\text{ in }  \Omega,\\
u=0, &\text{ on } T,
\end{cases}
\end{equation}
in the sense that $u \in \tilde{H}^1_{0,T}(\Omega,\mb{C})$ and
\begin{equation}\label{eq_dirchlet}
\int_{\Omega} (M\nabla u +ibu) \cdot \overline{(M\nabla v +ibv)} \,dz=0 \quad \text{ for any } v \in \tilde{H}^1_{0,T}(\Omega,\mb{C})
\end{equation}
where 
\begin{equation}
\tilde{H}^1_{0,T}(\Omega):=\{u \in \tilde{H}^1(\Omega,\mb{C}): u=0 \text{ on } T\}.
\end{equation}
Now we state the regularity results for weak solutions, as an extension of Propositions \ref{prop_blow_up_holder_a=0} and \ref{prop_blow_up_C1_holder_a=0}, up to $T\subset \partial\Omega$ under suitable assumptions on $M,b$ and $T$. Also in this occasion, we omit the proofs since they are very similar to the classical non-magnetic case.

We say that $T$ is $C^{1,\omega}$ if for any $x_0 \in T$ there exists a local parametrization $\Psi$ of $T \cap B_R(x_0)$  such that $\Psi$ is $C^1$ and the Jacobian $D \Psi$ has modulus of continuity $\omega$ for some $R>0$.

\begin{proposition}\label{prop_blow_up_holder_a=0_Diri}
Let   $x_0 \in T$ and  $R>0$ be such that $\partial \Omega \cap B_R(x_0)=T \cap B_R(x_0)$  and  $r \in (0,R)$. Suppose that $T$ is $C^{1,\omega_1}$ for some modulus of continuity $\omega_1$ and that 
$\alpha \le 1-d/q$. Suppose that $b \in L^q(B_R,\R^d)$ with $\|b\|_{L^q(B_R(x_0) \cap \Omega,\R^d)}\leq m$, $M$  is 
continuous with a modulus of continuity $\omega_2$ and $\norm{M}_{C^{0,\omega_2}(B_R(x_0)\cap \Omega),\R^{d,d})}\leq L$. Then, if $u$ is weak solution of 
\eqref{eq_dirchlet}, there exists a constant $c>0$ depending only on $d$, $r$, $R$, $q$ $\alpha$, $T$, $\la$, $\Lambda, m$ and  $L$ such that
\begin{equation}
\norm{u}_{C^{0,\alpha}(B_r(x_0) \cap \Omega,\mb{C})} \le c \norm{u}_{L^2(B_R(x_0)\cap \Omega,\mb{C})}.
\end{equation}
\end{proposition}

\begin{proposition}\label{prop_blow_up_C1_holder_a=0_Diri}
Let   $x_0 \in T$  and $R>0$ be such that $\partial \Omega \cap B_R(x_0)=T \cap B_R(x_0)$.
Let  $r \in (0,R)$,  $\alpha \in (0,1)$ and assume  that $T$ is $C^{1,\alpha}$.
Suppose that $b \in C^{0,\alpha}(B_R,\R^d)$ with $\|b\|_{C^{0,\alpha}(B_R(x_0) \cap \Omega,\R^d)}\leq m$, $M \in C^{0,\alpha}(B_R(x_0)\cap \Omega,\R^{d,d})$  with  $\norm{M}_{C^{0,\alpha}(B_R(x_0)\cap \Omega,\R^{d,d})}\leq L$. Then, if $u$ is weak solution of \eqref{eq_dirchlet},  there exists a constant $c>0$ depending only on $d$, $r$, $R$, $\alpha$, $T$, $\la$, $\Lambda, m$ and  $L$ such that
\begin{equation}
\norm{u}_{C^{1,\alpha}(B_r(x_0)\cap \Omega ,\mb{C})} \le c \norm{u}_{L^2(B_R(x_0)\cap \Omega,\mb{C})}.
\end{equation}
\end{proposition}

\section{Approximating with perforated domains}\label{sec_approx_domains}
In order to deal with the far more delicate situation $a \neq 0$, we are going to approximate solutions $u$ to \eqref{eq_magnetic_strong} with solutions of the same equation on a one parameter family of perforated domains around $\Sigma_0$, having a prescribed homogeneous Dirichlet boundary condition on the boundary of the hole. The construction is strongly inspired by the approach developed in \cite{CFV2,CFV,fio}. Then we are going to prove regularity estimates which are stable with respect to this approximation.

Let us define, for $N$ as in \eqref{hp_M_structure} and $\e>0$, the closed set
\begin{equation}\label{def_SigmaeN}
\Sigma_{\e,N}:=\{z=(x,y) \in \R^{d-n}\times \R^n:N_3^{-1}(x,y) y \cdot y \le \e^2\},
\end{equation}
with its boundary
\begin{equation}\label{def_boundary_SigmaeN}
\partial \Sigma_{\e,N}:=\{z=(x,y) \in \R^{d-n}\times \R^n:N_3^{-1}(x,y) y \cdot y =\e^2\}.
\end{equation}
If $N_3=\mathop{\rm Id}_{n}$, we simply let $\Sigma_\e:=\Sigma_{\e,\mathop{\rm Id}}$.
We note that for any matrix $N$ the set $\Sigma_{\e,N}$ shrinks onto $\Sigma_0$ as $\e \to 0^+$, where $\Sigma_0$ is defined in \eqref{def_Sigma}.
Let  $\Sigma_{\e,N,R}:=\Sigma_{\e,N}\cap B_{3R/4}$, and  let us consider the space $\tilde{H}^1_{0, \Sigma_{\e,N,R}}(B_R,\mb{C})$, defined as in \eqref{Hqe}. Then,
\begin{equation}\label{def_norm_H10}
\norm{v}_{\tilde{H}^1_{0, \Sigma_{\e,N,R}}(B_R,\mb{C})}^2:=\int_{B_R} |iM\nabla v+A v|^2 \, dz
\end{equation}
defines a norm on $\tilde{H}^1_{0, \Sigma_{\e,N,R}}(B_R,\mb{C})$ by the next result. Its proof is standard and we omit it.

\begin{proposition}\label{prop_poinc}
For any $\e>0$ there exists a constant $C_\e>0$, depending on $\e$, such that
\begin{equation}\label{ineq_poinc}
\int_{B_R} |v|^2 \, dz \le C_\e \int_{B_R} |iM\nabla v+A v|^2\, dz \quad  \text{ for any }  v \in\tilde{H}_{0, \Sigma_{\e,N,R}}^1(B_R,\mb{C}).
\end{equation}
\end{proposition}

\begin{proposition}\label{prop_ue}
Let $u \in \tilde{H}^1(B_R,\mb{C})$ be a solution of \eqref{eq_magnetic_strong}. Then there exists $\e_0>0$ such that for any $\e \in (0,\e_0)$ the problem
\begin{equation}\label{prob_ue}
\begin{cases}
L_{m} u_\e=0, & \text{ on } B_R \setminus\Sigma_{\e,N,R},\\
u_\e=0, & \text{ on } \Sigma_{\e,N,R},\\
u_\e=u, & \text{ on } \partial B_R,
\end{cases}
\end{equation}
admits a unique weak solution in the sense that  $u_\e \in \tilde{H}_{0, \Sigma_{\e,N,R}}^1(B_R,\mb{C})$, $u_\e=u$ on $\partial B_R$ and
\begin{equation}\label{eq_ue}
\int_{B_R} (iM\nabla u_\e+Au_\e)\overline{(iM\nabla v+Av)}  \, dz=0
\quad  \text{ for any }  v \in \tilde{H}_{0, \Sigma_{\e,N,R}}^1(B_R,\mb{C}).
\end{equation} 
Furthermore,
\begin{equation}\label{limit_ue}
u_\e \to u \quad \text{ strongly in } \tilde{H}^1(B_R,\mb{C}) \text{ as } \e \to 0^+.
\end{equation}
\end{proposition}
\begin{proof}
In view of Proposition \ref{prop_poinc}, minimizing the  coercive functional
\begin{equation}
J(v):=\int_{B_R} |iM\nabla v+A v|^2 \, dz \quad 
\text{ on the convex set  } \{v\in   \tilde{H}_{0, \Sigma_{\e,N,R}}^1(B_R,\mb{C}):v=u \text{ on } \partial B_R\}
\end{equation}
one can conclude that there exists a weak solution $u_\e$ to \eqref{prob_ue}. Furthermore, taking the difference of two solutions of \eqref{prob_ue} and testing \eqref{eq_ue} with the difference, it's clear that $u_\e$ is unique.
Let $w_\e:=u-u_\e$. Then, the function $w_\e$ is a solution of 
\begin{equation}
\begin{cases}
L_{m} w_\e=0, & \text{ in } B_R \setminus\Sigma_{\e,N,R},\\
w_\e=u, & \text{ on } \Sigma_{\e,N,R},\\
w_\e=0, & \text{ on } \partial B_R,
\end{cases}
\end{equation}
in the sense that $w_\e \in \tilde{H}_{0,\partial B_R}^1(B_R,\mb{C})$, $w_\e=u$ on $\Sigma_{\e,N,R}$ and  
\begin{equation}
\int_{B_R} (iM\nabla w_\e+Aw_\e)\overline{(iM\nabla v+Av)}\, dz=0 
\quad  \text{ for any }  v \in \tilde{H}_{0, \partial B_R \cup \Sigma_{\e,N,R}}^1(B_R,\mb{C}).
\end{equation}
Furthermore, $w_\e$ minimizes 
\begin{equation}
\int_{B_R} |iM\nabla v+A v|^2 \, dz 
\quad \text{ on }\{v \in   \tilde{H}_{0,\partial B_R}^1(B_R,\mb{C}):v=u \text{ on } \Sigma_{\e,N,R}\}.
\end{equation}
By \eqref{hp_M_elliptic}, if $(x,y_1) \in \partial \Sigma_{\e,N}$ and $(x,y_2) \in \partial \Sigma_{2\e,N}$
\begin{equation}
\frac{3}{\Lambda}\e^2 \le  \frac{1}{\Lambda}N_3^{-1}(x,y)(y_1-y_2)\cdot(y_1-y_2)\e^2\le |y_1-y_2|^2 \le \frac{1}{\la}N_3^{-1}(x,y)(y_1-y_2) \cdot(y_1-y_2) 
\le \frac{7}{\la}\e^2.
\end{equation}
Hence, there exists a sequence  of cut-off functions   $\eta_\e \in C_c^\infty(B_R)$ such that $\eta_\e=1$ on $\Sigma_{\e,N}$, $\eta_\e=0$ on $B_R\setminus \Sigma_{\e,N}$ and  $|\nabla \eta_\e| \le \frac{c}{\e}$  for some positive constant $C>0$ that does not depend on $\e$. 

We are going to test the minimality of $w_\e$ with $\eta_\e u$ and, after some estimates, pass to the limit as $\e \to 0^+$ to conclude that  $w_\e \to 0$ strongly in $\tilde{H}^1(B_R,\mb{C})$. 
More precisely, by Proposition \ref{prop_equiv_norms}, there exist constants $C_1,C_2, C_3,C_4>0$ (that do not depend on $\e$) such that 
\begin{multline}
\int_{B_R} |iM\nabla w_\e+A w_\e|^2 \, dz \le \int_{B_R} |iM\nabla(\eta_\e u)+A (\eta_\e u)|^2\, dz\\
\le C_1 \int_{B_R} \left(|\nabla (\eta_\e u)|^2 + \frac{|\eta_\e u|^2}{|y|^2}\right)\, dz \le C_2 \frac{1}{\e^2}\int_{B_R\cap \Sigma_{2\e,N} }|u|^2 \, dz\\
+C_3 \int_{B_R\cap \Sigma_{2\e,N} } \left(|\nabla u|^2 + \frac{ |u|^2}{|y|^2}\right)\, dz
\le  C_4\int_{B_R\cap \Sigma_{2\e,N} } \left(|\nabla u|^2 + \frac{ |u|^2}{|y|^2}\right)\, dz,
\end{multline}
since, by \eqref{hp_M_elliptic},  $\frac{1}{\e^2} \le \frac{\Lambda}{|y|^2}$ on $\Sigma_{\e,N}$ for any $\e>0$. Passing to the limit as $\e \to 0^+$, we obtain   $w_\e \to 0$ strongly in $\tilde{H}^1(B_R,\mb{C})$ in view of Proposition \ref{prop_poinc_boundary}. Hence, we have proved \eqref{limit_ue}. 
\end{proof}

\begin{proposition}\label{prop_ue_unif_bounded}
Suppose that $u_\e$ is a solution of \eqref{eq_ue}.
There exists a constant $\kappa>0$, depending only on the dimension $d$, on  $\Lambda$ and $\la$, such that 
\begin{equation}\label{ineq_ue_bounded}
\norm{u_\e}_{L^\infty(B_r,\mathbb C)} \le \frac{\kappa}{(R-r)^d}\norm{u_\e}_{L^2(B_R,\mathbb C)} \quad \text{ for any } r \in (0,R).
\end{equation}
\end{proposition}
\begin{proof}
Noticing that  for  any cut-off function $\eta \in C^\infty_c(B_R)$,  any $t \in [1,+\infty)$ and any $\tau>1$
\begin{equation}
\varphi:= (T_k|u_\e|)^{2t-2} u \eta^{2\tau}, 
\end{equation}
is a test function for \eqref{ineq_ue_bounded}, where $T_k |u_\e|:= \min\{k, |u_\e|+k^{-1}\}$ for any $k \in \mathbb{N}\setminus \{0\}$,
we may prove a Cacciopoli-type inequality as in Proposition \ref{prop_caccio} and then conclude with a Moser iteration scheme.
\end{proof}

\section{Magnetic Liouville theorems}\label{sec_Liou}
In this section we  prove  some magnetic Liouville-type theorems. To deal with the dependence in the variable $y$, we make use of a decomposition in Fourier series with respect to the eigenfunctions of a suitable magnetic problem on the unitary $(n-1)$-dimensional sphere in $\R^n$. This is the content  of the next section.

\subsection{Fourier decomposition on $\mb{S}^{n-1}$}
Let us consider a magnetic potential $P:\R^n \to \R^n$ such that
\begin{equation}
P(y):=\frac{p(y/|y|)}{|y|}.
\end{equation}
Here $p \in L^\infty(\mb{S}^{n-1},\mb{\R}^n)$ and $p\not \equiv 0$. We also assume the following transversality condition 
\begin{equation}\label{hp_transversality_p}
p(y)\cdot y=0.
\end{equation}
Let us consider the eigenvalue problem on the $(n-1)$-dimensional unit sphere $\mb{S}^{n-1}$ 
\begin{equation}\label{prop_eigen_sphere}
(i\nabla_{\mb{S}^{n-1}} \psi+ p \psi)^2+h \psi=\mu \psi,
\end{equation}
where $h \in L^\infty(\mb{S}^{n-1})$ and $h \ge 0$.
More precisely,   $\psi \in H^1(\mb{S}^{n-1}, \mb{C})\setminus\{0\}$ is an eigenfunction if 
\begin{equation}\label{eq_eigen_sphere_liou}
\int_{\mb{S}^{n-1}} (i\nabla_{\mb{S}^{n-1}} \psi +p\psi) \cdot 
\overline{(i\nabla_{\mb{S}^{n-1}} v   +p v)}+ h \psi \overline{v} \, d \sigma=
\mu \int_{\mb{S}^{n-1}}\psi\overline{v}\, d \sigma, \quad \text{ for any } v \in H^1(\mb{S}^{n-1}, \mb{C}).
\end{equation}
In view of the compactness of the embedding of $L^2(\mb{S}^{n-1}, \mb{C})$ into $H^1(\mb{S}^{n-1}, \mb{C})$, the spectrum of \eqref{eq_eigen_sphere_liou} is given by an 
increasing and diverging sequence of non-negative eigenvalues $\{\mu_k(p,h)\}_{k \in \mathbb{N}\setminus \{0\}}$ with finite multiplicity $m_k$. 
Let $\{\psi_{k,j}\}_{j=1,\dots,m_k}$ be an orthonormal basis of the eigenspace associated to $\mu_k$ and 
let $\{\psi_{k,j}\}$, with  $k \in \mathbb{N}\setminus \{0\}$ and $j=1, \dots, m_k$, be the correspondent orthonormal basis   of $L^2(\mb{S}^{n-1})$. 

Finally, it is possible to characterize $\mu_k(p,h)$ with the Courant–Fischer min-max theorem, so that, in particular, $\mu_1(p,h)$ is as in  \eqref{def:mu1}. 
To simplify the statement of the next proposition we set 
\begin{equation}
\tilde{H}^1_{0,B''_0}(B_R'',\mb{C}):=\tilde{H}^1(B_R'',\mb{C}).
\end{equation}

\begin{proposition}\label{prop_liouville_rho}
Let $\e\ge 0$. Assume that $u \in \tilde{H}^1_{0,B''_{\e}}(B_R'',\mb{C})$ for any $R >0$ and it  satisfies the equation 
\begin{equation}\label{eq_magnetic_weak_entire_Rn_e>0}
\int_{\R^n} \left(i\nabla u+\frac{p}{|y|}u\right) \cdot \overline{\left(i\nabla v+\frac{p}{|y|} v\right)} +\frac{h}{|y|^2}u \overline{v} \, dz=0
\quad  \text{ for any } v \in  \tilde{H}^1_{0, B''_\e}(B''_R,\mb{C}) 
\end{equation}
for any $R>0$, with $p\not \equiv 0$. Suppose that \eqref{hp_transversality_p} holds and that there exist $\gamma,c>0$ such that 
\begin{equation}\label{ineq_u_fourier}
|u(y)| \le c(1+|y|)^{\gamma}.
\end{equation}
Then there exists $k(\gamma) \in \mb{N}\setminus \{0\}$ such that for any $y \in \R^n$
\begin{equation}\label{eq_u_Fourier_e>0}
u(y)=\sum_{k=1}^{k(\gamma)} \sum_{j=1}^{m_k}(c_{k,j}^+|y|^{\gamma_k^+}+c_{k,j}^{-}|y|^{\gamma_k^-})\psi_{k,j}(y/|y|)
\end{equation} 
with  $c^{\pm}_{k,j} \in \R$  and 
\begin{equation}
\gamma_k^{\pm}(p,h):=-\frac{n-2}{2} \pm \sqrt{\left(\frac{n-2}{2}\right)^2+\mu_k(p,h)}.
\end{equation}
If $\e=0$ then $c_{k,j}^{-}=0$ for any  $k\in \mathbb{N}\setminus \{0\}$ and  any $j=1,\dots,m_k$.
\end{proposition}

\begin{proof}
Let $u$ be a solution to \eqref{eq_magnetic_weak_entire_Rn_e>0}.
Let $f \in C^\infty_c((\e,+\infty))$, $k \in \mathbb{N}\setminus\{0\}$, $j=1,\dots,m_k$ and $\varphi \in H^1(\mb{S}^{n-1},\mb{C})$. Testing \eqref{eq_magnetic_weak_entire_Rn_e>0} with $f(|y|)\varphi(y/|y|)$ and then passing to polar coordinates $(r,\theta)$, we obtain, in view of \eqref{hp_transversality_p},
\begin{equation}\label{proof_prop_liouville_rho_1}
\int_0^\infty r^{n-1}\int_{\mb{S}^{n-1}}\pd{u}{r} f' \varphi
+\frac{f}{r^2}[(i\nabla_{\mb{S}^{n-1}}u + pu)\overline{(i\nabla_{\mb{S}^{n-1}}\varphi+p\varphi)}+hu \varphi]\, d\sigma dr =0.
\end{equation}
Let us decompose $u$ in Fourier series with respect to $\{\psi_{k,j}\}$, for any $k \in \mathbb{N}\setminus \{0\}$ and $j=1, \dots, m_k$, which is an orthonormal basis of $L^2(\mb{S}^{n-1})$. More precisely, we  write
\begin{equation}
u(r,\theta)= \sum_{k=1}^\infty \sum_{j=1}^{m_k} u_{k,j}(r) \psi_{k,j}(\theta),
\end{equation}
where the Fourier coefficients $u_{k,j}$ are  given by
\begin{equation}
 u_{k,j}(r) := \int_{\mb{S}^{n-1}} u(r \cdot)  \psi_{k,j} \, d\sigma \quad  \text{ and }  \quad  u_{k,j} \in H_{loc}^1((0,+\infty),r^{n-1}).
\end{equation}
Hence, taking $\psi_{k,j}$ as $\varphi$ in \eqref{proof_prop_liouville_rho_1}, we obtain that for any $k \in \mathbb{N}\setminus \{0\}$ and $j=1, \dots, m_k$
\begin{equation}
\int_0^\infty r^{n-1}\int_{\mb{S}^{n-1}} u'_{k,j}  f' \psi^2_{k,j}+\frac{\mu_k(p,h)}{r^2}u_{k,j}f \psi^2_{k,j}\, d\sigma dr =0,
\end{equation}
in view of \eqref{eq_eigen_sphere_liou} tested with $\psi^2_{k,j}$.
Equivalently, since $\norm{\psi_{k,j}}_{L^2(\mb{S}^{n-1},\mb{C})}=1$,
\begin{equation}
\int_0^\infty r^{n-1} [u'_{k,j}  f' +\mu_k(p,h)r^{-2}u_{k,j}f] \, dr =0.
\end{equation}
An integration by parts yields that 
\begin{equation}
\int_0^\infty r^{n-1} [-u''_{k,j}-(n-1) r^{-1} u'_{k,j}+\mu_k(p,h)r^{-2}u_{k,j}]f \, dr =0
\end{equation}
for any $f \in C^\infty_c((\e,+\infty))$, so that 
\begin{equation}
-u''_{k,j}(r)-(n-1) r^{-1} u'_{k,j}(r)+\mu_k(p,h)r^{-2}u_{k,j}(r)=0 \quad  \text{ in } (\e,+\infty).
\end{equation}
Hence, letting
\begin{equation}
\gamma^{\pm}_k:=-\frac{n-2}{2} \pm\sqrt{\left(\frac{n-2}{2}\right)^2 +\mu_k(p,h)},
\end{equation}
we conclude that
\begin{equation}
u_{k,j}(r)=c^-_{k,j} r^{\gamma_k^-}+c^+_{k,j}r^{\gamma_k^+}
\end{equation}
for some constants $c^{\pm}_{k,j} \in \R$.
Let $k(\gamma)$ be the smallest integer such that   $\gamma_k^+>\gamma$. By \eqref{ineq_u_fourier}, we must have $c^+_{k,j}=0$ for any $k\ge k(\gamma)$ and  any $j=1,\dots,m_k$. 
If $\e>0$, the boundary condition $u_{k,j}(\e)=0$ also implies that $c^-_{k,j}=0$ for any $k\ge  k(\gamma)$ and  any $j=1,\dots,m_k$.
On the other hand, if $\e=0$, then  $u_{k,j}' \in L^2((0,1),r^{n-1})$, and so 
\begin{equation}
(c^{-}_{k,j}\gamma_{k}^-)^2 \int_0^1 r^{n-3+2\gamma_{k}^-} \, dr <+\infty,
\end{equation}
which means that $c^{-}_{k,j}=0$ for any  $k \in \mathbb{N}\setminus \{0\}$ and $j=1, \dots, m_k$.
\end{proof}

\subsection{A Liouville-type Theorem}
Suppose that
\begin{equation}\label{hp_AMg_Liouville}
A(x,y)= \frac{a(x,y/|y|)}{|y|},  \quad \text{ and } \quad M \text{ is constant},
\end{equation}
where $a \in L^\infty(B_R'\times\mb{S}^{n-1}, \R^d)$ does not depend on $x$ and satisfies the transversality condition \eqref{hp_transversality}. Suppose that  $M$ satisfies \eqref{hp_M_elliptic}-\eqref{hp_M_structure}.
Recall that $N:=M^{T}M$, and that 
\begin{equation}
\Sigma_{\e,N}=\{(x,y) \in \R^d: N_3^{-1}y \cdot y \le\e^2\}.
\end{equation}
To avoid repetitions, let us set
\begin{equation}\label{def_H10Sigma0}
    \tilde{H}^1_{0,\Sigma_{0,N}}(B_R,\mb{C}):=\tilde{H}^1(B_R,\mb{C}).
\end{equation}
For any $\e\ge 0$, we say that $u$ is an entire weak solution to 
\begin{equation}\label{prob_magnetic_weak_entire}
\begin{cases}
L_mu=0, &\text{ in }\R^d\setminus\Sigma_{\e,N},\\
u=0, &\text{ in } \Sigma_{\e,N},
\end{cases}
\end{equation}
if $u \in \tilde{H}^1_{0,\Sigma_{\e,N}}(B_R,\mb{C})$ for any $R>0$ and
\begin{equation}\label{eq_magnetic_weak_entire}
\int_{\R^d} \nabla_{M,A} u \cdot \overline{\nabla_{M,A} v}  \, dz=0, \quad  \text{ for any } v \in  \tilde{H}^1_{0,\Sigma_{\e,N}}(B_R,\mb{C}) \text{ and any } R>0.
\end{equation}
We are going to classify all the entire solutions having some polynomial bound on the growth rate at infinity. First, we deal we the possible dependence with respect to  the variable $x$. Arguing as in  \cite[Lemma 4.1, Corollary 4.2, Lemma 4.3]{TTV_harnack}, that is,  by  means of a Cacciopoli inequality (Proposition \ref{prop_caccio_simple})  and Nirenberg differential quotients methods, we can  prove the following. 
\begin{proposition}\label{prop_entire_sol_x}
Let $\e\ge 0$. Suppose that $u$ is an entire solution to \eqref{prob_magnetic_weak_entire}. Then
\begin{enumerate}[(i)]
\item $\pd{u}{x_i}$ is an entire weak solution to \eqref{prob_magnetic_weak_entire} for any $i=1, \dots, d-n$;
\item if there exists constants $C,\gamma>0$ such that
\begin{equation}
|u(z)| \le C(1+|z|^{\gamma}) \text{ for a.e. } z \in \R^d,
\end{equation}
then $u$ is a polynomial in $x$ of degree at most $\lfloor \gamma \rfloor$ with coefficients that depend only on the variable $y$.
\end{enumerate}
\end{proposition}
In particular, we remark that the homogeneous Dirichlet boundary condition on $\partial\Sigma_{\e,N}$ is satisfied by $\pd{u}{x_i}$. Indeed $\{y=0\}$ is tangent to $\partial \Sigma_{\e,N}$ since $N$ is constant.

Now we study the dependence in the variable $y$. 
Let us define  $\tilde{a} \in  L^\infty(B_R'\times \mb{S}^{n-1},\R^d)$ as
\begin{equation}\label{def_am_liouville}
\tilde{a}(x,y/|y|):=\frac{N^{\frac{1}{2}}(M^{-1} a)(x,N_{3}^{\frac{1}{2}}y/|y|)}{{|N_3^{\frac{1}{2}}y/|y||}}. 
\end{equation}
Clearly $\tilde{a}$ does not depend on $x$ since $a$ does not, and letting $N_3^{\frac{1}{2}}y= \tilde y$
\begin{equation}\label{hp_trans_lioville_am}
\tilde{a}''(x,y) \cdot y= \frac{N_3^{\frac{1}{2}}(M^{-1} a)''(x,N_{3}^{\frac{1}{2}}y)}{{|N_3^{\frac{1}{2}}y|^2}} \cdot  N_3^{-\frac{1}{2}} \tilde y 
= \frac{(M^{-1}a)''(x,\tilde y)}{{|\tilde y}|^2} \cdot   \tilde y =0
\end{equation}
holds since  $a$ satisfies  the transversality condition \eqref{hp_transversality}.

\begin{theorem}[Magnetic subquadratic Liouville]\label{theor_liou_y}
Let $u$ be an entire solution of \eqref{prob_magnetic_weak_entire} such that 
\begin{equation}\label{def_growth_u}
|u(z)| \le C (1+|z|^\gamma) \quad \text{ for any } z \in \R^d
\end{equation}
for some $C>0$, and some $\gamma \in (0,2)$ such that
\begin{equation}
\gamma < \gamma_1:=-\frac{n-2}{2} + \sqrt{\left(\frac{n-2}{2}\right)^2 +\mu_1(\tilde{a}'',|\tilde{a}'|^2)}.
\end{equation}
Then, $u$ is trivial.
\end{theorem}
\begin{proof}
Let $u$ be an entire solution. Let us define $w(\xi):=u(N^{\frac{1}{2}}\xi)$, and recall that $\Sigma_\e=\Sigma_{{\rm Id}_d,\e}$.
By the definition of $\Sigma_{N,\e}$, see \eqref{def_SigmaeN}, the change of variables $z=N^{\frac{1}{2}}\xi$ in \eqref{eq_magnetic_weak_entire} yields, for any $\e\ge0$,
\begin{equation}\label{proof_theor_liou_y_1}
\int_{\R^d}\left(i\nabla w+\frac{\tilde{a}}{|y|} w\right)\cdot\overline{\left(i\nabla v+\frac{\tilde{a}}{|y|}v\right)}\, dz=0  \quad \text{ for any } v \in  \tilde{H}^1_{0,\Sigma_\e}(B_R,\mb{C}) \text{ and } R>0.
\end{equation}
By Proposition \ref{prop_entire_sol_x}, since $\gamma <2$,
\begin{equation}
w(x,y)=w_0(y)+\sum_{j=1}^{d-n}w_j(y) x_j.   
\end{equation}
For any $j=0,\dots, d-n,$ the function $w_j$ still satisfies the growth condition \eqref{def_growth_u} since
\begin{equation}
w_0(y)= w(0,y) \quad \text{ and  } \quad w_j(y)=w(x_je_j,y)-w(0,y).
\end{equation}
For any  $f \in C^\infty_c(\R^{d-n},\mb{C})$,  and  $\varphi \in  C^\infty_c(\R^n,\mb{C})$ we may test \eqref{proof_theor_liou_y_1} with $v(x,y)=f(x) \varphi(y)$ to get 
\begin{multline}
\int_{\R^d} \left(i\nabla w_j+\frac{\tilde{a}}{|y|} w_j\right) \cdot \overline{\left(i\nabla(f \varphi)+\frac{\tilde{a}}{|y|} f \varphi\right)}  \, dz\, dz\\
=\int_{\R^d}\left[\left(i\nabla w_j+\frac{\tilde{a}''}{|y|} w_j\right) \cdot \overline{\left(i\nabla \varphi+\frac{\tilde{a}''}{|y|}  \varphi\right)}
+ \frac{|\tilde{a}'|^2}{|y|^2} w_j \overline{\varphi}\right] \overline{f}dz
-i \int_{\R^d}   w_j\overline \varphi \frac{\tilde{a}'}{|y|} \cdot \nabla \overline{f} \, dz=0.
\end{multline}
By the Divergence Theorem, since only $\nabla \overline{f}$ depends on $x$,  we conclude that
\begin{equation}
\int_{\R^d}   w_j\overline \varphi \frac{\tilde{a}'}{|y|} \cdot \nabla \overline{f} \, dz=0.
\end{equation}
Hence, since 
\begin{equation}
\int_{\R^{d-n}}\int_{\R^n}\left[\left(i\nabla w_j+\frac{\tilde{a}''}{|y|} w_j\right) \cdot \overline{\left(i\nabla \varphi+\frac{\tilde{a}''}{|y|}  \varphi\right)}
+ \frac{|\tilde{a}'|^2}{|y|^2} w_j \overline{\varphi}\right] \, dy \overline{f}dx=0
\end{equation}
for any $f \in C^\infty_c(\R^{d-n},\mb{C})$, we must have 
\begin{equation}
\int_{\R^n}\left(i\nabla w_j+\frac{\tilde{a}''}{|y|} w_j\right) \cdot \overline{\left(i\nabla \varphi+\frac{\tilde{a}''}{|y|^2}\varphi\right)}
+ \frac{|\tilde{a}'|^2}{|y|^2} w_j \overline{\varphi}=0,
\end{equation}
for any $j=0,\dots, d-n$.
Hence, $w_j$ is a solution to \eqref{eq_magnetic_weak_entire_Rn_e>0} with $p=\tilde{a}''$ and $h=|\tilde{a}'|^2$, thus \eqref{eq_u_Fourier_e>0} and \eqref{def_gamma} readily imply that $w_j=0$ for any $j=0, \dots, d-n$ so that $w$ is trivial. Hence, also  $u$ is trivial.
\end{proof}

\section{H\"older and Schauder estimates}\label{sec_Holder_estimates}
In this section, we prove H\"older and Schauder estimates for solutions to \eqref{eq_magnetic_strong}. We first need to prove stable estimates for solutions of the regularized problems in perforated domains, defined in Section \ref{sec_approx_domains}.  Then, the desired regularity is obtained in the limit of approximating sequences. A similar  strategy has been successfully employed in different contexts when dealing with singular or degenerate elliptic equations, see for instance \cite{CFV,fio,STV_even_sol,TTV_harnack}.

\subsection{Uniform H\"older and Schauder estimates in perforated domains}\label{subsec_unif_hold}
Let us consider the problem
\begin{equation}\label{prob_ue_total}
\begin{cases}
L_{m} u_\e=0, & \text{ on } B_R \setminus \Sigma_{\e,N},\\
u_\e=0, & \text{ on } \Sigma_{\e,N}.\\
\end{cases}
\end{equation}
Weak solutions to the latter problem are functions $u_\e \in \tilde{H}_{0, \Sigma_{\e,N}}^1(B_R,\mb{C})$ such that it holds
\begin{equation}\label{eq_ue_total}
\int_{B_R} (iM\nabla u_\e+Au_\e)\overline{(iM\nabla v+Av)} \, dz=0
\quad  \text{ for any }  v \in \tilde{H}_{0, \Sigma_{\e,N}}^1(B_R,\mb{C}).
\end{equation}
In this section, we prove uniform (with respect to $\e$) H\"older and Schauder estimates for weak solutions. First, we deal with $C^{0,\alpha}$ estimates and then with $C^{1,\alpha}$ estimates.

\begin{theorem}[$\e$-stable $C^{0,\alpha}$ estimate]\label{theorem_reg_C0alpha}
Let $R >0$  and $r \in (0,R)$. Suppose that  $q >d$,    $\gamma_1(M,a) \ge \gamma_1^*$ for some $\gamma_1^* >0$  with  $\gamma_1(M,a)$ as in \eqref{def_gamma} and 
\begin{equation}\label{def_alpha_holder}
\alpha\in (0,1) \cap (0,\gamma_1^*) \cap (0, 1-d/q].
\end{equation}
Let  $M$ be   continuous with modulus of continuity $\omega_1$ and assume that  \eqref{hp_M_structure} holds. Suppose that $a$  is continuous with modulus of continuity $\omega_2$.
Let $b \in L^q(B_R,\R^d)$ with $q >d$.
Suppose that $N_3^{-1}$ is $C^{1,\omega_3}(B_R,\R^{n,n})$  for some modulus of continuity $\omega_3$. 
Finally let $L$ be such that
\begin{equation}
\norm{M}_{C^{0,\omega_1}(B_R,\R^{d,d})}+\norm{a}_{C^{0,\omega_2}(B'_R\times\mathbb S^{n-1},\R^{d})} + \norm{b}_{L^q(B_R,\R^{d})}+ \norm{N_3^{-1}}_{C^{1,\omega_3}(B_R,\R^{n,n})} \le L
\end{equation}
and let  $u_\e$ be a solution of \eqref{prob_ue_total}.

There exist $c>0$ and $\e_1>0$, depending only on  $n,d, \la, \Lambda, L, r,R,\alpha,$ and $ \gamma_1^*$,  
such that  for  any $\e \in (0,\e_1)$ there holds
\begin{equation}
\norm{u_\e}_{C^{0,\alpha}(B_r,\mb{C})} \le c \norm{u_\e}_{L^2(B_R,\mb{C})}.
\end{equation}
\end{theorem}

\begin{proof}
It is not restrictive to assume that $R=1$ and $r=1/2$. Let $\alpha$ be as in \eqref{def_alpha_holder}.
In view of Propositions \ref{prop_blow_up_holder_a=0}-\ref{prop_blow_up_holder_a=0_Diri}, any solution $u_\e$ of \eqref{eq_ue} in $B_1$ belongs to $C^{0,\alpha}(B_{1/2},\mb{C})$ for any $\alpha \in (0,1-d/q]$ and there exists a constant $c_\e$, depending on $\e$, such that
\begin{equation}
\norm{u_\e}_{C^{0,\alpha}(B_{1/2},\mb{C})} \le c_\e \norm{u_\e}_{L^2(B_1,\mb{C})}.
\end{equation}
We are going to prove that $c_\e$ can be chosen uniformly with respect to $\e$ arguing by contradiction. 
Suppose that there exist  a sequence $\e_k \to 0^+$ and a sequence of associated solutions $u_k$ with
\begin{equation}
 \norm{u_k}_{C^{0,\alpha}(B_{1/2},\mb{C})} \ge k \norm{u_k}_{L^2(B_1,\mb{C})}.
\end{equation}
More precisely, $\{u_k\}$ weakly solves \eqref{prob_ue_total}  with 
\begin{equation}
A_k(x,y):=\frac{a_k(x,y/|y|)}{|y|}+b_k(x,y) \quad \text{ and } \quad  d\times d \text{ matrices } M_k
\end{equation}
such that:
\begin{itemize}
    \item $M_k$  is uniformly bounded in $C^{0,\omega_1}(B_1, \R^{d,d})$ and satisfies \eqref{hp_M_elliptic} with  $\la, \Lambda>0$ uniform in $k$;
    \item $a_k$ is uniformly bounded in $C^{0,\omega_2}(B'_1\times\mathbb S^{n-1}, \R^d)$ and $b_k$ is uniformly bounded in $L^q(B_1,\R^{d})$;
    \item  $(N_k)_3^{-1}$ is uniformly bounded in $C^{1,\omega_3}(B_1,\R^{n,n})$;
    \item  $\gamma_1(M_k,a_k)\ge \gamma_1^*$ for any $k \in \mb{N}\setminus \{0\}$.
\end{itemize}

We divide the rest of the proof into several steps.

\textbf{Step 1. Definition of the blow-up sequences.}
In view of Proposition \ref{prop_ue_unif_bounded}, there exists a constant $C>0$, that does not depend on $k$, such that 
\begin{equation}
[u_k]_{C^{0,\alpha}(B_{1/2},\mb{C})} \ge C k   \norm{u_k}_{L^\infty(B_{4/5},\mb{C})}.
\end{equation}
Let $\eta \in C^\infty_c(B_{3/4})$ be a cut-off function with $0\leq\eta\leq1$, $\eta\equiv1$ in $B_{1/2}$ and 
\begin{equation}
E_k:=[\eta u_k]_{C^{0,\alpha}(B_1,\mb{C})}.
\end{equation}
Thus,
\begin{equation}\label{proof_theorem_reg_C0alpha_05}
E_k\ge [u_k]_{C^{0,\alpha}(B_{1/2},\mb{C})} \ge C k \norm{u_k}_{L^\infty(B_{4/5},\mb{C})}.
\end{equation}
Let us consider two sequences $z_k:=(x_k,y_k),\hat z_k:= (\hat x_k, \hat y_k) \in B_{3/4}\setminus \Sigma_{\e_k,N_k}$ such that
\begin{equation}\label{proof_theorem_reg_C0alpha_2}
\frac{|(\eta u_k)(z_k)-(\eta u_k)(\hat z_k)|}{|z_k-\hat z_k|^{\alpha}}\ge \frac{E_k}{2}.
\end{equation}
We need to check that two such sequences actually exist, since a priori we only know that \eqref{proof_theorem_reg_C0alpha_2} holds for some  $z_k,\hat z_k \in 
B_1$. It is clear that we cannot have both $z_k,\hat z_k \in   \Sigma_{\e_k,N_k}$ because then $(\eta u_k)(z_k)-(\eta u_k)(\hat z_k)=0$. Furthermore, if one  
between $z_k$ and $\hat z_k$, say  $\hat z_k$ belongs to $\Sigma_{\e_k,N_k}$, then  we may replace $\hat z_k$ with a projection  
of $z_k$ on $\partial \Sigma_{\e_k, N_k}$. Indeed, with this substitution, the numerator in \eqref{proof_theorem_reg_C0alpha_2} stays the same while the denominator decreases.
Furthermore, since  $\eta \in  C^\infty_c(B_{3/4})$ we cannot have both $z_k,\hat z_k \in   B_1 \setminus B_{3/4} $ otherwise we would have  $(\eta u_k)(z_k)-(\eta u_k)(\hat z_k)=0$. Then we can suppose that $z_k\in B_{3/4}$. Moreover, we may suppose that also $\hat z_k \in   B_{3/4}$ up to replacing it with the projection of $z_k$ on $\partial B_{3/4}$. 

Let us define $r_k:=|z_k-\hat z_k|$. From \eqref{proof_theorem_reg_C0alpha_05} and \eqref{proof_theorem_reg_C0alpha_2}, 
it follows that
\begin{equation}
\frac{C}{2} k \norm{u_k}_{L^\infty(B_{3/4},\mb{C})} \le \frac{|(\eta u_k)(z_k)-(\eta u_k)(\hat z_k)|}{r_k^{\alpha}} \le \frac{2\norm{u_k}_{L^\infty(B_{3/4},\mb{C})} }{r_k^\alpha},
\end{equation}
thus 
\begin{equation}
r_k\le (4/C)^{\frac {1}{\alpha}} k^{-\frac {1}{\alpha}}\to 0^+ \quad  \text{ as } k \to \infty.
\end{equation}
Let $d_k:=\di(z_k,\partial \Sigma_{\e_k, N_k})$. From now on we need to distinguish three cases:
\begin{align}
&\textbf{Case 1: } \frac{d_k}{r_k} \to +\infty\quad \text{ as } k \to \infty, \label{case_blowup_1}\\
&\textbf{Case 2: } \frac{d_k}{r_k} \le C  \quad \text{ and }  \quad \frac{|y_k|}{r_k}\to +\infty  \quad  \text{ as } k \to \infty,\label{case_blowup_2}\\
&\textbf{Case 3: } \frac{|y_k|}{r_k} \le C, \label{case_blowup_3}
\end{align}
for some constant $C>0$ that does not depend on $k$. Let us define
\begin{equation}
Z_k:=(x_k,\tau_k y_k/|y_k|),
\end{equation}
with $\tau_k$ such that $Z_k \in \partial \Sigma_{\e_k, N_k}$. The existence and  uniqueness of $Z_k$ and the fact that $\tau_k <|y_k|$ are proven in
\cite[Lemma A.1]{CFV}. It follows that
\begin{equation}
|y_k| > |y_k|-\tau_k =\left|y_k-\tau_k \frac{y_k}{|y_k|}\right|=|z_k-Z_k| \ge d_k.
\end{equation}
Hence, the cases above are mutually exclusive and cover all the possible scenarios in the asymptotics involving $r_k,d_k,|y_k|$.  Moreover it is easy to check that 
\begin{equation}\label{relation_er}
\begin{cases}
\frac{\e_k}{r_k}\to+\infty, & \mathrm{in} \ \textbf{Case \ 2},\\
\frac{\e_k}{r_k}\leq C, & \mathrm{in} \ \textbf{Case \ 3}.
\end{cases}
\end{equation}
Let us denote by $z_k^0:=(x^0_k,y^0_k)$ a projection of $z_k$ on $\partial \Sigma_{\e_k, N_k}$ (possibly not unique) and let 
\begin{equation}\label{tildezk}
\tilde{z}_k:=(\tilde{x}_k,\tilde{y}_k):=
\begin{cases}
(x_k,y_k), &\text{ in \textbf{Case 1}},\\
(x^0_k,y^0_k), &\text{ in \textbf{Case 2}},\\
(x_k,0), &\text{ in \textbf{Case 3}}.
\end{cases}
\end{equation}
By the definition of the \textbf{Cases 1,\, 2,} and \textbf{3}, $|z_k-\tilde{z}_k|\le C r_k$. 
Let us define the rescaled domains
\begin{multline}\label{def_Omega_k}
\Omega_k:=\frac{B_1\setminus \Sigma_{\e_k,N_k}-\tilde{z}_k}{r_k}\\
=\left\{ z=(x,y) \in \R^d: |\tilde{z}_k+r_kz| <1, (N_k)_3^{-1}(\tilde{z}_k+r_kz)(\tilde{y}_k+r_ky) \cdot (\tilde{y}_k+r_ky)> \e^2_k \right\},
\end{multline}
and the limit domain
\begin{equation}\label{def_Omega_infty}
\Omega_\infty:=\left\{ z=(x,y) \in \R^d: \text{ there exist } \hat k, r>0: B_r(z) \subset \Omega_k \text{ for any } k \ge \tilde{k} \right\}.
\end{equation}
Let us define the limits (up to further subsequences)
\begin{equation}\label{def_bar_x_N_e}
M_\infty:=\lim_{k \to \infty}M_k(\tilde{z}_k), \quad \tilde z_\infty=\lim_{k \to \infty}\tilde z_k,
\end{equation}
and in \textbf{Case 3}  also 
\begin{equation}\label{bareps}
 \bar\e:=\lim_{k \to \infty} \frac{\e_k}{r_k}. 
\end{equation}
By \eqref{relation_er}, $\bar\e \in [0,C]$ and let $N_\infty$ be as in \eqref{hp_M_structure} for the matrix $M_\infty$.
Then, as shown in \cite[Proof of Theorem 1.3, i), Step 2]{CFV}, the domain $\Omega_\infty$ is actually given by 
\begin{equation}\label{proof_theorem_reg_C0alpha_3}
\Omega_\infty=
\begin{cases}
\R^d,   &\text{ in } \textbf{Case 1},\\
\Pi_\nu,   &\text{ in } \textbf{Case 2},\\
\R^d\setminus\Sigma_{\bar\e,N_\infty},  &\text{ in } \textbf{Case 3},
\end{cases}
\end{equation}
where the set $\Pi_\nu$ is the half space defined as 
\begin{equation}
\Pi_\nu:=\{z=(x,y): \nu \cdot y >0\}
\end{equation}
for some $\nu \in \mb{S}^{d-1}$. Furthermore, $\R^d\setminus\Sigma_{\bar\e,N_\infty}$ is the whole space perforated by a $\bar \e$-cylinder defined as in \eqref{def_SigmaeN}. It may happen that  $\bar \e=0$ in which case $\Sigma_{\bar\e,N_\infty}$ is reduced to $\Sigma_0$, see \eqref{def_Sigma}.

Finally, in $\Omega_k$ let 
\begin{equation}
v_k(z):=\frac{(\eta u_k)(\tilde{z}_k+r_k z)-(\eta u_k)(\tilde{z}_k)}{r_k^{\alpha}E_k}, \quad \text{ and } 
\quad w_k(z):=\frac{\eta(\tilde{z}_k)( u_k(\tilde{z}_k+r_k z)-u_k(\tilde{z}_k))}{r_k^{\alpha}E_k}.
\end{equation}
We notice that  $u_k(\tilde{z}_k)=0$ in  \textbf{Cases 2} and \textbf{3}.

\textbf{Step 2. Convergence of the blow-up sequences in H\"older spaces.}
For any  $z, z' \in \Omega_k$
\begin{equation}
|v_k(z)-v_k(z')| =\frac{|(\eta u_k)(\tilde{z}_k+r_k z)-(\eta u_k)(\tilde{z}_k+r_kz')|}{r_k^{\alpha}E_k} \le |z-z'|^{\alpha},
\end{equation}
so that 
\begin{equation}
[v_k]_{C^{0,\alpha}(\Omega_k,\mb{C})}\le 1.
\end{equation}
Furthermore, for any compact set $K \subset \Omega_\infty$, since $K \subset \Omega_k$ for $k$ large enough, and $v_k(0)=0$, 
\begin{equation}
|v_k(z)|\le |z|^{\alpha} \le C(K),
\end{equation}
for some positive constant $C(K)>0$ depending only on $K$. Hence 
\begin{equation}
\norm{v_k}_{C^{0,\alpha}(K,\mb{C})}\le 1+C(K).
\end{equation}
With a diagonal argument over an exhaustion of $\Omega_\infty$ by compact sets $K$  we can show that there exists $v_\infty \in C^{0,\alpha}(\Omega_\infty\cap B_R,\mb{C})$ for any $R>0$ such that $v_k \to v_\infty$ in $C^{0,\beta}(\Omega_\infty\cap B_R,\mb{C})$ for any $R>0$ and any $\beta \in (0,\alpha)$. Furthermore, 
\begin{equation}\label{proof_theorem_reg_C0alpha_2.5}
|v_\infty(z)|\le |z|^{\alpha} \quad \text{ for any } z \in \Omega_\infty.
\end{equation}
For any compact set $K \subset \ \Omega_\infty$ and any $z \in K$,
\begin{multline}
|v_k(z)-w_k(z)|=\frac{\eta(z_k)}{r_k^{\alpha}E_k}|u_k(\tilde z_k+r_k z)||\eta(\tilde z_k+r_k z) -\eta(\tilde z_k)|\\
\le \frac{\norm{u_k}_{L^\infty(B_{4/5},\mathbb C)}}{E_k}\norm{\nabla \eta}_{L^\infty(B_1)}r_k^{1-\alpha} \le C \norm{\nabla \eta}_{L^\infty(B_1)}k^{-1}r_k^{1-\alpha},
\end{multline}
by \eqref{proof_theorem_reg_C0alpha_05}. Since $r_k \to 0^+$, we conclude that also $w_k \to v_\infty$ uniformly on compact sets.
Finally, we note that, letting
\begin{equation}
\hat\xi_k:=\frac{\hat z_k-\tilde{z}_k}{r_k}, \quad \xi_k:=\frac{z_k-\tilde{z}_k}{r_k}
\end{equation}
we have
\begin{equation}
|\hat \xi_k| = \frac{|\hat z_k-\tilde{z}_k|}{r_k} \le \frac{|\hat z_k-z_k|}{r_k}  +\frac{|\tilde{z}_k-z_k|}{r_k} \le (1+C),
\end{equation}
thus, up to passing to a subsequences,  $\hat\xi_k \to \hat \xi$ for some $\hat \xi \in \R^d$ as $k \to \infty$. 
Similarly we can show that $\xi_k$ is bounded and so $\xi_k \to \xi$ for some $\xi \in \R^d$ as $k \to \infty$. 
Then, since $\xi_k,\hat \xi_k \in \Omega_k$,
\begin{equation}\label{proof_theorem_reg_C0alpha_2.8}
|v_\infty(\xi)-v_\infty(\hat \xi)|=\lim_{k\to \infty}|v_k(\xi_k)-v_k(\hat \xi_k)|=\lim_{k \to \infty}\frac{|(\eta u_k)(z_k)-(\eta u_k)( \hat z_k)|}{r_k^{\alpha}E_k} \ge \frac{1}{2},
\end{equation}
in view of \eqref{proof_theorem_reg_C0alpha_2}. In conclusion, $v_\infty\not \equiv 0$. 

\textbf{Step 3. The limit is an entire solution.}
With a change of variables in \eqref{eq_ue}, we can see that for any $\varphi \in C^\infty_c(\Omega_\infty,\mb{C})$, for large enough $k$
\begin{multline}\label{proof_theorem_reg_C0alpha_3_5}
\int_{\Omega_k} \Bigg[i M_k(\tilde{z}_k+r_kz)\nabla w_k(z)+\left(r_k\frac{a_k(\tilde{x}_k+r_kx,(\tilde{y}_k+r_ky)/|\tilde{y}_k+r_ky|)}{|\tilde{y}_k+r_ky|}
+ r_kb_k(\tilde{z}_k+r_kz)\right)w_k(z)\\
+\left(\frac{a_k(\tilde{x}_k+r_kx,(\tilde{y}_k+r_ky)/|\tilde{y}_k+r_ky|)}{|\tilde{y}_k+r_ky|} 
+b_k(\tilde{z}_k+r_kz)\right) \frac{r_k^{1-\alpha}\eta(\tilde z_k)u_k(\tilde{z}_k)}{E_k}\Bigg] \\
\cdot \Big(-i M_k(\tilde{z}_k+r_kz) \nabla \overline{\varphi}(z)\\
+ \left(r_k\frac{a_k(\tilde{x}_k+r_kx,(\tilde{y}_k+r_ky)/|\tilde{y}_k+r_ky|)}{|\tilde{y}_k+r_ky|}
+ r_kb_k(\tilde{z}_k+r_kz)\right)\overline{\varphi}(z)\Big)\, dz=0.
\end{multline}
Let us show that $\{w_k\}_{k \in \mb N \setminus \{0\}}$ is bounded in $\tilde{H}^1(\Omega_\infty\cap B_R,\mb{C})$ for any $R>0$  so that, in particular,
$w_k \rightharpoonup v_\infty$ weakly in $\tilde{H}^1(\Omega_\infty\cap B_R,\mb{C})$ as $k \to \infty$ for any $R>0$. 

We recall that $u_k(\tilde{z}_k)$=0 in \textbf{Case 2} and in  \textbf{Case 3}, so that 
\begin{equation}
\norm{\nabla w_k}_{L^2(B_R\cap\Omega_k,\mb{C})} \le C_1 \norm{w_k}_{L^2(B_{2R}\cap\Omega_k,\mb{C})},
\end{equation}
thanks to \eqref{ineq_caccio_simple}.

On the other hand, in \textbf{Case 1}, $\tilde{z}_k=z_k$ and for $k$ large enough
\begin{equation}
|y_k+r_ky| \ge |y_k| -r_k|y| \ge \frac{1}{2} |y_k|,
\end{equation}
since $|y_k|/r_k \ge d_k/r_k \to +\infty$ which implies $r_k|y| \le \frac{1}{2} |y_k|$ for $k$ large enough. It follows that
\begin{multline}
r_k^{1-\alpha}\frac{|a_k(x_k+r_kx,(y_k+r_ky)/|y_k+r_ky|)||u_k( z_k)|}{E_k|y_k+r_ky|}\\
\le C_2\norm{a_k}_{L^\infty(B_1,\R^d)}\frac{r_k^{1-\alpha}|u_k(z_k)-u_k(x_k,0)|}{E_k|y_k|} 
\le C_2\norm{a_k}_{L^\infty(B_1,\R^d)}\left(\frac{r_k}{|y_k|}\right)^{1-\alpha}\le C_3\left(\frac{r_k}{d_k}\right)^{1-\alpha},
\end{multline}
for some positive constants $C_1,C_2,C_3$ that do not depend on $k$. 
Hence,  by \eqref{ineq_caccio_simple},
\begin{multline}
\norm{\nabla w_k}_{L^2(B_R\cap \Omega_k,\mb{C})} \le C \norm{w_k}_{L^2(B_{2R}\cap \Omega_k,\mb{C})} \\
+C  \left(\frac{r_k}{d_k}\right)^{1-\alpha}+C  \left(\frac{r_k}{d_k}\right)^{2-\alpha}
+C\left(\frac{r_k^{1-\alpha}|u_k(z_k)|}{E_k}
+\frac{r_k^{2-\alpha}}{d_k^{1-\alpha}}\right)\left(\int_{B_{2R\cap \Omega_k}}|b_k(z_k+r_kz)|^2 \,dz\right)^{\frac{1}{2}}\\
+C\frac{r_k^{1-\alpha}|u_k(z_k)|}{E_k}\left(\int_{B_{2R\cap \Omega_k}}|b_k(z_k+r_kz)|^q \,dz\right)^{\frac{2}{q}},
\end{multline}
for some positive constant $C$ that does not depend on $k$.
Furthermore, by the H\"older inequality and \eqref{proof_theorem_reg_C0alpha_05},
\begin{multline}
\left(\frac{r_k^{1-\alpha}|u_k(z_k)|}{E_k}
+\frac{r_k^{2-\alpha}}{d_k^{1-\alpha}}\right)\left(\int_{B_{2R}}|b_k( z_k+r_kz)|^2 \,dz\right)^{\frac{1}{2}}\\
\le \left(\frac{r_k^{1-\alpha-\frac{d}{2}}}{k}+\frac{r_k^{2-\alpha-\frac{d}{2}}}{d_k^{1-\alpha}}\right) \left( \int_{B_{2r_kR}( z_k)}|b_k(y)|^2 \,dy\right)^{\frac{1}{2}} \\
\le C  r_k^{1-\frac{d}{q}-\alpha} \norm{b_k}_{L^q(B_1,\R^d)}k^{-1}+ r_k^{1-\frac{d}{q}} \norm{b_k}_{L^q(B_1,\R^d)} \to 0^+, \text{ as } k \to \infty 
\end{multline}
and
\begin{multline}
\frac{r_k^{2-\alpha}|u_k(z_k)|}{E_k} \left(\int_{B_{2R}}|b_k(z_k+r_kz)|^q \,dz\right)^{\frac{2}{q}}\\
\le\frac{r_k^{2-\alpha-\frac{2d}{q}}}{k}\left( \int_{B_{2r_kR}(z_k)}|b_k(y)|^q \,dy\right)^{\frac{2}{q}}
\le Cr_k^{1-\frac{d}{q}} \norm{b_k}_{L^q(B_1,\R^d)}^2k^{-1} \to 0^+, \text{ as } k \to \infty. 
\end{multline}
Hence, we conclude that $\{w_k\}_{k \in \mb N \setminus \{0\}}$ is bounded in $\tilde{H}^1(\Omega_\infty\cap B_R,\mb{C})$  also in \textbf{Case 1}.

Since  $M_k(\tilde z_k+r_k\cdot)$ is bounded in $C^{0,\omega_1}(B_R,\R^{d,d})$   for any $R>0$
\begin{equation}
M_k(\tilde z_k+r_kz) \to M_\infty, \quad \text{ uniformly in  } B_R,
\end{equation}
where $M_\infty$ is the $d\times d$ constant matrix defined in \eqref{def_bar_x_N_e}.
Then, the same computations that we made to prove the boundedness of $\{w_k\}$ in $\tilde{H}^1(\Omega_\infty \cap B_R,\mathbb{C})$ for any $R>0$ allow us to pass to the limit in \eqref{proof_theorem_reg_C0alpha_3_5}. Let us treat separately \textbf{Cases 1,\, 2 } and  \textbf{3}.

\textbf{Case 1.} In this case, passing to the limit as $k \to \infty$ in \eqref{proof_theorem_reg_C0alpha_3_5},
we obtain, 
\begin{equation}\label{proof_theorem_reg_C0alpha_4_5}
\int_{\R^d}   M_\infty\nabla v_\infty \cdot  M_\infty\nabla \overline\varphi \, dz=0
\end{equation}
for any $\varphi \in C_c^\infty(\R^d,\mb{C})$.

\textbf{Case 2.}
Passing to the limit  as $k \to \infty$ in \eqref{proof_theorem_reg_C0alpha_3_5}, we have  
\begin{equation}\label{proof_theorem_reg_C0alpha_5}
\int_{\Pi_\nu} M_\infty\nabla v_\infty \cdot M_\infty\nabla \overline{\varphi} \, dz=0
\end{equation}
for any $\varphi \in C_c^\infty(\Pi_\nu, \mb{C})$ and $v_\infty=0$ on $\partial \Pi_\nu$ by uniform convergence of $w_k$ to $v_\infty$.

\textbf{Case 3.} 
In this case, since $\tilde{y}_k=0$,  up to passing to a subsequence,
\begin{equation}
a_k(\tilde{x}_k+r_kx,y/|y|) \to a_\infty( y/|y|),  \quad \text{ uniformly in  } B_R,
\end{equation}
for some  continuous function $a_\infty $  with  modulus of continuity $\omega_2$. Then passing to the limit as $k \to \infty$ in \eqref{proof_theorem_reg_C0alpha_3_5}
\begin{equation}\label{proof_theorem_reg_C0alpha_6}
\int_{\R^d}\left(i  M_\infty\nabla v_\infty+\frac{a_\infty(y/|y|))}{|y|}v_\infty\right)
\cdot\left(\overline{i  M_\infty\nabla \varphi+\frac{a_\infty(y/|y|))}{|y|}\varphi}  \right) \, dz=0
\end{equation}
for any  $\varphi \in C^\infty_c(\R^d\setminus \Sigma_{\bar \e,N_\infty},\mb{C})$  with $\bar \e \ge 0$ and  $\Sigma_{\bar \e,N_\infty}$ as in \eqref{proof_theorem_reg_C0alpha_3}.  Furthermore,  $v_\infty=0$ on $\partial \Sigma_{\bar \e,N_\infty}$ by uniform convergence of $w_k$ to $v_\infty$.

\textbf{Step 4. Liouville theorems and conclusion.}
In \textbf{Case 1}, since $M_\infty$ is real valued, \eqref{proof_theorem_reg_C0alpha_4_5} is equivalent to 
\begin{equation}
\int_{\R^d}   M_\infty\nabla \Rea v_\infty \cdot  M_\infty\nabla \varphi \, dz=0 \quad \text{ and } \int_{\R^d}   M_\infty\nabla \Ima v_\infty \cdot  M_\infty\nabla \varphi \, dz=0,
\end{equation}
for any $\varphi \in C_c^\infty(\R^d)$.
The change of variables $t=N_\infty^{-\frac{1}{2}}z$ yields
\begin{equation}
\int_{\R^d}   \nabla \Rea (v_\infty\circ N_\infty^{\frac{1}{2}}) \cdot  \nabla \varphi \, dt=0 \quad \text{ and } \int_{\R^d} \nabla \Ima (v_\infty\circ N_\infty^{\frac{1}{2}}) \cdot  \nabla \varphi \, dt=0,
\end{equation}
for any $\varphi \in C_c^\infty(\R^d)$. By classical Liouville theorems, see for example \cite[Theorem 4.2]{V_reg}, it follows that $v_\infty  \equiv 0$.

In \textbf{Case 2} we can do the same after an even reflection around the plane $\Pi_\nu$.
Finally, in \textbf{Case 3}, we can invoke Theorem \ref{theor_liou_y}  to  conclude that $v_\infty  \equiv 0$. 
In any case we have reached a contradiction with \eqref{proof_theorem_reg_C0alpha_2.5} or \eqref{proof_theorem_reg_C0alpha_2.8}.
\end{proof}

Now we turn to stable Schauder estimates. We need a preliminary lemma, which is a result about $C^{1,\beta}$ functions vanishing on $\partial \Sigma_{\e,N}$. We borrow some ideas from \cite[Theorem 5.1]{fio} carefully adapted for our setting.

\begin{lemma}\label{lemma_gradient_u}
Suppose that $u \in C^{1,\beta}(B_R\setminus \Sigma_{\e,N},\mb{C})$ for some $\beta \in (0,1)$ and that $u=0$ on $\partial\Sigma_{\e,N} \cap B_R$. Then there exists a constant $C>0$, depending only on $\la,\Lambda, d$ and $n$ such that
\begin{equation}\label{ineq_gradient_u}
\norm{\nabla u}_{L^\infty( \partial  \Sigma_{\e,N} \cap B_R,\mathbb C)} \le C [\nabla u]_{C^{0,\beta}(\partial  \Sigma_{\e,N} \cap B_R,\C)} \e^\beta.
\end{equation}
\end{lemma}
\begin{proof}
We begin the proof with a simple and general geometric observation. For any open bounded set $D \subset \R^n$ with  $C^1$ boundary  
and any $\nu \in  \mb{S}^{n-1}$ there exists  a point $y \in \partial D$ such that the unit outer normal vector to $D$ in $y$ is $\nu$. Indeed, it is enough to take a hyperplane  orthogonal to $\nu$  and disjoint  from $D$ and choose $y$ as a point of minimal distance on $\partial D$. Since $D$ is a $C^1$ domain the outer normal vector to $D$ in $y$ exists and it is $\nu$.

Let $\nu(z)$  be  the unit outer normal vector to $\Sigma_{\e,N}$ in $z \in \partial \Sigma_{\e,N}$ and $\{e_i\}_{i=1,\dots, d}$ be the canonical basis of $\R^d$. For any  $x\in \R^{d-n}$, if there exists $(x,y) \in \partial \Sigma_{\e,N}$ then there also exists $(x,y^{(i)}) \in \partial \Sigma_{\e,N}$ such that 
$\nu(x,y^{(i)})=(\nu_1(x,y^{(i)}), \dots, \nu_{d-n}(x,y^{(i)}), 0, \dots, 0 )+e_i$ for any $i=d-n+1,\dots , d$. Indeed, 
the section $\{x\} \times \R^{n}$ of $\Sigma_{\e,N}$ is a $C^1$ bounded domain and so we can  apply the geometric observation.
Hence, in particular, $e_j$ is tangent to $\Sigma_{\e,N}$ in $(x, y^{(i)})$ for any $j=d-n+1,\dots , d$, with $j \neq i$.

Let $j,i\in\{d-n+1,\dots , d\}$ with $i \neq j$, $(x,y) \in \partial\Sigma_{\e,N}$ and let $(x, y^{(i)}) \in \partial \Sigma_{\e,N}$  be such that  $e_j$ is tangent to $\partial\Sigma_{\e,N}$ in $(x, y^{(i)})$. Since 
\begin{equation}
\nabla u(z)\cdot q =0\quad  \text{ for any  } q \in \R^d \text{ such that } q\cdot  \nu(z)=0,
\end{equation}
we have that 
\begin{equation}
|\nabla u(x,y) \cdot e_j| = |(\nabla u(x,y)-\nabla u(x,y^{(i)}))\cdot e_j| \le C [\nabla u]_{C^{0,\beta}(\partial \Sigma_{\e,N}\cap B_R,\mathbb C)} \e^\beta.
\end{equation}
By \cite[Lemma A.1]{CFV}, $|\nu_i(z)| \le C \e$ for any $i=1, \dots, d-n$ for some constant $C>0$ depending only on $\la,\Lambda, d$ and $n$.
Hence,
\begin{equation}
|\nabla u(x,y)| \le \sum_{i=1}^d |\nabla u(x,y) \cdot e_i| \le C \e  \norm{\nabla u}_{L^\infty( \partial\Sigma_{\e,N}\cap B_R,\mathbb C)}
+C [\nabla u]_{C^{0,\beta}(\partial \Sigma_{\e,N}\cap B_R,\mathbb C)} \e^\beta.
\end{equation}
Passing to the supremum over $(x,y) \in    \partial\Sigma_{\e,N} \cap B_R$ 
\begin{equation}
 \norm{\nabla u}_{L^\infty( \partial\Sigma_{\e,N}\cap B_R,\mathbb C)}\le C [\nabla u]_{C^{0,\beta}( \partial\Sigma_{\e,N}\cap B_R,\C)} \e^\beta,
\end{equation}
for some $C>0$, depending only on $\la,\Lambda, d$ and $n$.
\end{proof}

We are now in position to prove $C^{1,\alpha}$ uniform estimates.

\begin{theorem}[$\e$-stable $C^{1,\alpha}$ estimate]\label{theorem_reg_C1alpha}
Let $R>0$ and $r \in (0,R)$.  Suppose that $\gamma_1(M,a) \ge \gamma_1^*$ for 
some $\gamma_1^*>1$, with  $\gamma_1(M,a)$ as in \eqref{def_gamma}. 
Let $\alpha \in (0,1)\cap (0,\gamma_1^*-1)$.

Assume $M \in C^{0,\alpha}(B_R,\R^{d,d})$, that  \eqref{hp_M_structure} holds and  that $N_3^{-1}$ is $C^{1,\alpha}(B_R,\R^{n,n})$.   Suppose  $a \in C^{0,\bar \alpha}(B'_R\times\mathbb S^{n-1},\R^d)$  for some $\bar \alpha \in (\alpha,1)$, and $b \in  C^{0,\alpha}(B_R,\R^d)$.
Finally let $L$ be such that
\begin{equation}
\norm{M}_{C^{0,\alpha}(B_R,\R^{d,d})}+\norm{a}_{C^{0, \bar \alpha}(B'_R\times\mathbb S^{n-1},\R^{d})} + \norm{b}_{C^{0,\alpha}(B_R,\R^{d})}+ \norm{N_3^{-1}}_{C^{1,\alpha}(B_R,\R^{n,n})} \le L
\end{equation}
and let  $u_\e$ be a solution of \eqref{prob_ue_total}.

There exist $c>0$ and $\e_1>0$, depending only on  $n,d, \la, \Lambda, L, r,R,\alpha,\overline\alpha$, and $\gamma_1^*$   such that for  any $\e \in (0,\e_1)$ there holds
\begin{equation}\label{ineq_reg_C1alpha}
\norm{u_\e}_{C^{1,\alpha}(B_r\setminus \Sigma_{\e,N},\mb{C})} \le c \norm{u_\e}_{L^2(B_R,\mb{C})},
\end{equation}
and 
\begin{equation}\label{ineq_nabla_C1alpha}
|\nabla u_\e(z)| \le c \norm{u_\e}_{L^2(B_R,\mb{C})} \e^{\alpha} \quad \text{ for any } z \in \partial \Sigma_{\e,N} \cap B_r.
\end{equation}
\end{theorem} 

\begin{proof}
It is not restrictive to assume that $R=1$ and $r=1/2$.

We are going to prove the result by a contradiction argument which shares the same spirit of the proof of the previous Theorem \ref{theorem_reg_C0alpha}. However, here we do not have a uniform $L^\infty$-bound for the gradients of the solutions (i.e. a uniform Lipschitz bound). This fact forces us to prove the result in two steps. First, we prove the uniform estimate in \eqref{ineq_reg_C1alpha} for a slightly smaller exponent $0<\beta<\alpha$, where $\beta$ denotes the exponent of regularity of the solution and $\alpha$ still denotes the exponent of regularity of the data. In this way, the uniform Lipschitz bound is unlocked, and we can redo the same contradiction argument by choosing the sharp $\beta=\alpha$. We would like to stress that \textbf{Step 6} is the only part in the proof which needs this two-steps procedure.

Thanks to  Propositions \ref{prop_blow_up_holder_a=0_Diri}-\ref{prop_blow_up_C1_holder_a=0_Diri}, any weak solution $u_\e$ of \eqref{prob_ue_total} in $B_1$ belongs to $C^{1,\beta}(B_{1/2}\setminus \Sigma_{\e,N},\mb{C})$. In particular, there exists a constant $c_\e$, depending on $\e$, such that
\begin{equation}
\norm{u_\e}_{C^{1,\beta}(B_{1/2}\setminus \Sigma_{\e,N},\mb{C})} \le c_\e \norm{u_\e}_{L^2(B_1,\mb{C})}.
\end{equation}
We are going to show that $c_\e$ can be chosen uniformly with respect to $\e$ arguing by contradiction.

Suppose that there exist a sequence $\e_k \to 0^+$ and a sequence of associated solutions $u_k$ with
\begin{equation}
\norm{u_k}_{C^{1,\beta}(B_{1/2}\setminus \Sigma_{\e_k,N_k},\mb{C})}\ge k \norm{u_k}_{L^2(B_1,\mb{C})}.
\end{equation}
More precisely, $\{u_k\}$ weakly solves \eqref{prob_ue_total}  with 
\begin{equation}
A_k(x,y):=\frac{a_k(x,y/|y|)}{|y|}+b_k(x,y) \quad \text{ and } \quad  d\times d \text{ matrices } M_k
\end{equation}
such that:
\begin{itemize}
    \item $M_k$  is uniformly bounded in $C^{0,\alpha}(B_1, \R^{d,d})$ and satisfies \eqref{hp_M_elliptic} with  $\la, \Lambda>0$ uniform in $k$;
    \item $a_k$ uniformly bounded in $C^{0,\overline{\alpha}}(B'_1\times\mathbb S^{n-1}, \R^d)$ and  $b_k$ uniformly bounded in  $C^{0,\alpha}(B_1,\R^d)$;
    \item  $(N_k)_3^{-1}$ is uniformly bounded in $C^{1,\alpha}(B_1,\R^{n,n})$;
    \item  $\gamma_1(M_k,a_k)\ge \gamma_1^*$ for any $k \in \mb{N}\setminus \{0\}$.
\end{itemize}
We divide the rest of the proof into several steps.  To simplify the exposition we will denote with $C>0$ any constant that does not depend on $k$.

\textbf{Step 1. Definition of the blow-up sequences.}
In view of Theorem \ref{theorem_reg_C0alpha}, for any $\gamma \in (0,1)$, there exists a constant $C>0$ such that
\begin{equation}\label{boundcgamma}
\norm{u_k}_{C^{0,\gamma}(B_{4/5},\mb{C})} \le C   \norm{u_k}_{L^2(B_1,\mb{C})}.
\end{equation}
Let $\eta \in C^\infty_c(B_{3/4})$ be a cut-off function with $0\leq\eta\leq1$, $\eta\equiv1$ in $B_{1/2}$ and let
\begin{equation}
E_k:=[|\nabla( \eta u_k)|]_{C^{0,\beta}(B_1\setminus \Sigma_{\e_k,N_k},\mb{C})}.
\end{equation}
Since $\eta u_k=0$ in $B_1\setminus B_{3/4}$ it cannot happen that
\begin{equation}
\frac{\norm{ \nabla (\eta u_k)}_{L^\infty(B_1\setminus\Sigma_{\e_k,N_k},\mb{C})}}{\norm{u_k}_{L^2(B_1,\mb{C})}}\to +\infty \quad \text{ as } k \to \infty,
\end{equation}
while 
\begin{equation}
\frac{[|\nabla( \eta u_k)|]_{C^{0,\beta}(B_1\setminus \Sigma_{\e_k,N_k},\mb{C})}}{ \norm{u_k}_{L^2(B_1,\mb{C})}}\le C.
\end{equation}
Hence, we must have 
\begin{equation}
E_k \to +\infty.
\end{equation}

Let us consider two sequences $z_k:=(x_k,y_k),\hat z_k:= (\hat x_k, \hat y_k) \in B_1\setminus \Sigma_{\e_k,N_k}$ such that 
\begin{equation}\label{proof_theorem_reg_C2alpha_2}
\frac{|\nabla (\eta u_k)(z_k)-\nabla(\eta u_k)(\hat z_k)|}{|z_k-\hat z_k|^{\beta}}\ge \frac{E_k}{2}.
\end{equation}
Just as in Theorem \ref{theorem_reg_C0alpha}, it is not restrictive to assume that $z_k, \hat z_k \in B_{3/4}$. We are going to use the same notation of Theorem \ref{theorem_reg_C0alpha}. Let $r_k:=|z_k-\hat z_k|$ and  $d_k:=\di(z_k,\partial \Sigma_{\e_k, N_k})$.
From now on, we need to distinguish the same three cases of Theorem \ref{theorem_reg_C0alpha}; that is,  \eqref{case_blowup_1}, \eqref{case_blowup_2}  and \eqref{case_blowup_3}. Furthermore, also \eqref{relation_er} holds.  
Unlike Theorem \ref{theorem_reg_C0alpha}, at the moment we can only prove that $r_k \to 0^+$ in \textbf{Cases 1-2}.

Let us denote with $z_k^0:=(x^0_k,y^0_k)$ a projection of $z_k$ on $\partial \Sigma_{\e_k, N_k}$ and let $\tilde z_k$ be the points defined in \eqref{tildezk}.

By the definition of the \textbf{Cases 1,\, 2,} and \textbf{3}, $|z_k-\tilde{z}_k|\le C r_k$. Furthermore, in \textbf{Cases 1-2} we have
\begin{equation}
|z_k^0| \le |z_k|+d_k \le \frac{3}{4}+Cr_k,
\end{equation}
and so for any $\tau \in (3/4,1)$ we have that $z_k^0 \in B_\tau\setminus \Sigma_{\e_k,N_k}$ for $k$ large enough. 
Let $\Omega_k$ be as in \eqref{def_Omega_k} and the limit domain $\Omega_\infty$ (defined in \eqref{def_Omega_infty}).

Finally, in $\Omega_k$ let us define in \textbf{Cases 1} and \textbf{2}
\begin{align}
&v_k(z):=\frac{(\eta u_k)(\tilde{z}_k+r_k z)-(\eta u_k)(\tilde{z}_k)-\nabla (\eta u_k)(\tilde{z}_k)\cdot r_kz}{r_k^{1+\beta}E_k},\\
&w_k(z):=\frac{\eta(\tilde{z}_k)(u_k(\tilde{z}_k+r_k z)-u_k(\tilde{z}_k)-\nabla  u_k(\tilde{z}_k)\cdot r_kz)}{r_k^{1+\beta}E_k},
\end{align}
while in \textbf{Case 3} we let
\begin{equation}
v_k(z):=\frac{(\eta u_k)(\tilde{z}_k+r_k z)}{r_k^{1+\beta}E_k}, \quad \text{ and } 
\quad w_k(z):=\frac{\eta(\tilde{z}_k)u_k(\tilde{z}_k+r_k z)}{r_k^{1+\beta}E_k}.
\end{equation}
Notice that in \textbf{Case 3} we do not need to subtract the first order expansion since $u_k(\tilde z_k)=0$ and $\nabla u_k(\tilde z_k)=0$.

\textbf{Step 2. Convergence of the blow-up sequences in H\"older spaces.}
For any $z,z' \in \Omega_k$
\begin{equation}
|\nabla v_k(z)-\nabla v_k(z')|\le\frac{|\nabla(\eta u_k)(\tilde{z}_k+r_k z)-\nabla(\eta u_k)(\tilde{z}_k+r_k z')|}{r_k^{\beta}E_k}\le  |z-z'|^{\beta},
\end{equation}
so that
\begin{equation}\label{proof_theorem_reg_C2alpha_2_6}
[\nabla v_k]_{C^{0,\beta}(\Omega_k,\mb{C})}\le 1.   
\end{equation}
In \textbf{Cases 1} and \textbf{2} for any compact set $K\subset \Omega_\infty$, noticing that $K \subset \Omega_k$ for $k$ large enough,
\begin{equation}
\norm{\nabla v_k}_{L^\infty(K,\mathbb C)}\le \sup_{z \in\Omega_k \cap K} |\nabla v_k(z)- \nabla v_k(0)| \le\sup_{z \in K} |z|^\beta,
\end{equation}
since $\nabla v_k(0)=0$ and $0 \in \overline{\Omega}_k$.
Furthermore, arguing as in  \cite[Proof of Theorem 1.3, ii), Step 2]{CFV}, we can  prove that for any compact set $K \subset \Omega_\infty$, there exists a compact set $K'\subset\R^d$ with $K \subset K'$ and $k_0\in\mathbb N$ such that  for $k\geq k_0$ and for any $z \in K$ there exists a curve $\gamma_k$, depending on $z$, satisfying
\begin{equation}\label{eq_curves}
\gamma_k:[0,1] \to \overline{\Omega}_k, \;\;\gamma_k(0)=0,\; \;\gamma_k(1)=z,\;\; \norm{\gamma_k}_{C^{0,1}(0,1)} \le C(1+|z|),
\;\; \mathop{\rm spt}(\gamma_k) \subset K'\cap \overline \Omega_k.
\end{equation}
Hence, thanks to \eqref{eq_curves}, for any $z \in K\cap \Omega_k$ 
\begin{equation}\label{eq_vk_ineq_cureves}
|v_k(z)|=\left|\int_0^1\nabla v_k(\gamma_k(t))\cdot \gamma_k'(t)\, dt \right|\le \norm{\nabla v_k}_{L^\infty(K'\cap \Omega_k,\mathbb C)}\norm{\gamma_k'}_{L^\infty(0,1)}
\le C(1+|z|^{\beta+1}).
\end{equation}
We conclude that for any compact set $K\subset \Omega_\infty$, since $K \subset \Omega_k$ for $k$ large enough,
\begin{equation}
\norm{v_k}_{C^{1,\beta}(K,\mb{C})}\le C.
\end{equation}
We still need to  deal with \textbf{Case 3}. In this case, it is harder to show a uniform bound on the gradient of $v_k$ and we do so 
by the means of Lemma \ref{lemma_gradient_u} which implies that 
\begin{equation}\label{proof_theorem_reg_C2alpha_2_7}
|\nabla u_k(z)| \le CE_k \e_k^{\beta} \quad \text{ for any } z \in B_1 \cap \partial \Sigma_{\e_k,N_k}.
\end{equation}
In particular, recalling that in \textbf{Case 3} $\e_k/r_k\leq C$ by \eqref{relation_er}, 
\begin{equation}
|\nabla v_k((z_k^0-\tilde{z}_k)/r_k)|=\frac{|\nabla(\eta u_k)(z_k^0)|}{r_k^{\beta} E_k} \le C \frac{\e_k^{\beta}}{r_k^\beta}\le C.
\end{equation}
Hence, for any compact set $K \subset \Omega_\infty$  we have that, by \eqref{proof_theorem_reg_C2alpha_2_6},
\begin{multline} 
\norm{\nabla v_k}_{L^\infty(K,\mathbb C)}\le \sup_{z \in K} | \nabla v_k(z)- \nabla v_k((z_k^0-\tilde{z}_k)/r_k)|+|\nabla v_k((z_k^0-\tilde{z}_k)/r_k)|
\le C\sup_{z \in K} (1+|z|^\beta),
\end{multline}
since 
\begin{equation}
\frac{|z_k^0-\tilde{z}_k|}{r_k} \le \frac{|z_k^0-z_k|}{r_k}+\frac{|z_k-\tilde{z}_k|}{r_k}=\frac{d_k}{r_k}+\frac{|y_k|}{r_k} \le C.
\end{equation}
As in  \textbf{Cases 1} and \textbf{2},  we can show that for any compact set $K \subset \Omega_\infty$, there exists a compact set $K'\subset\R^d$ with $K \subset K'$ and $k_0\in\mathbb N$ such that  for $k\geq k_0$ and for any $z \in K$ there exists a curve $\gamma_k$, depending on $z$, satisfying
\begin{equation}
\gamma_k:[0,1] \to \overline{\Omega}_k, \;\;\gamma_k(0)=\frac{z_k^0-\tilde{z}_k}{r_k},\; \;\gamma_k(1)=z,\;\; \norm{\gamma_k}_{C^{0,1}(0,1)} \le C(1+|z|),
\;\; \mathop{\rm spt}(\gamma_k) \subset K'\cap \overline \Omega_k.
\end{equation}
It follows that  also in \textbf{Case 3} we may estimate  $|v_k|$ as in \eqref{eq_vk_ineq_cureves} to conclude that
\begin{equation}
\norm{v_k}_{C^{1,\beta}(K,\mb{C})}\le C.
\end{equation}
Hence, up to passing to subsequences and with a standard diagonal argument, there exists a function $v_\infty \in C^{1,\beta}(\Omega_\infty \cap B_R,\mathbb C)$ for any $R>0$ such that $v_k \to v_\infty$ in $C^{1,\gamma}(\Omega_\infty \cap B_R,\mathbb C)$ for any $R>0$ and any $\gamma<\beta$. Furthermore,
\begin{equation}
[\nabla v_\infty]_{C^{0,\beta}(\Omega_\infty,\mathbb C)}\le 1
\end{equation}
so that 
\begin{equation}
|v_\infty(z)| \le C(1+|z|^{\beta+1}) \quad \text{ in } \Omega_\infty.
\end{equation}
Letting 
\begin{equation}
\hat\xi_k:=\frac{\hat z_k-\tilde{z}_k}{r_k}, \quad \xi_k:=\frac{z_k-\tilde{z}_k}{r_k}
\end{equation}
as in Theorem \ref{theorem_reg_C0alpha}, we can show that, 
up to passing to a subsequences,  $\hat\xi_k \to \hat \xi$ 
and  $\xi_k \to \xi$ for some $\hat \xi, \xi \in \mb{R}^{d}$ as $k \to \infty$. 
Then, since $\xi_k,\hat \xi_k \in \Omega_k$,
\begin{equation}\label{proof_theorem_reg_C2alpha_2_8}
|\nabla v_\infty(\xi)-\nabla v_\infty(\hat \xi)|=\lim_{k\to \infty}|\nabla v_k(\xi_k)-\nabla v_k(\hat \xi_k)|=\lim_{k \to \infty}\frac{|\nabla(\eta u_k)(z_k)-\nabla(\eta u_k)( \hat z_k)|}{r_k^{\beta}E_k} \ge \frac{1}{2}.
\end{equation}
In conclusion, $v_\infty\not \equiv \rm{const}$ and in particular  $v_\infty\not \equiv 0$.

\textbf{Step 3. We have that  $r_k \to 0^+$ in Case 3.} If by contradiction $r_k \to r_\infty\in(0,3/2]$,  then for any $z \in \Omega_\infty$, since $\eta \in C^\infty_c(B_{3/4})$,
\begin{equation}
|v_k(z)|=\left|\frac{(\eta u_k)(\tilde z_k+r_kz)}{E_k r_k^{1+\beta}}\right|\le 
\frac{2 \norm{u_k}_{L^\infty(B_{3/4}\setminus \Sigma_{\e_k,N_k},\mathbb C)}}{E_k r_k^{1+\beta}}\le \frac{C}{E_k} \to 0^+,
\end{equation}
a contradiction since $|v_\infty(z)|=\lim_{k \to \infty}|v_k(z)|=0$ but $v_\infty\not \equiv 0$.

\textbf{Step 4. $w_k \to v_\infty$ uniformly on compact sets.}
In \textbf{Cases 1} and \textbf{2}, if we choose $\gamma =\beta$ in \eqref{boundcgamma},
\begin{multline}
|v_k(z)-w_k(z)|=\frac{|(\eta u_k)(\tilde z_k+r_kz)-\eta(\tilde{z}_k)u_k(\tilde z_k+r_kz)-(u_k\nabla \eta)(\tilde z_k)\cdot r_k z|}{r_k^{1+\beta}E_k}\\
\le\frac{|u_k(\tilde z_k+r_kz)(\eta(\tilde z_k+r_kz)-\eta(\tilde z_k)-\nabla \eta(\tilde z_k)\cdot r_k z)|}{r_k^{1+\beta}E_k}
+C\frac{|\nabla \eta(\tilde z_k)||u_k(\tilde z_k+r_kz)-u_k(\tilde z_k)|}{r_k^{\beta}E_k}\\
\le C\frac{\norm{u_k}_{C^{0,\beta}(B_{4/5},\mathbb C)}r_k^{1-\beta}}{E_k}+C\frac{\norm{u_k}_{C^{0,\beta}(B_{4/5},\mathbb C)}}{E_k}
\le \frac{C}{E_k}\to 0^+.
\end{multline}
In \textbf{Case 3}, again with the choice $\gamma=\beta$ in \eqref{boundcgamma},
\begin{multline}
|v_k(z)-w_k(z)|=\frac{|u_k(\tilde z_k+r_kz) (\eta(\tilde z_k+r_kz) -\eta(\tilde{z}_k))|}{r_k^{1+\beta}E_k}
\le C \frac{|u_k(\tilde z_k+r_kz) -u_k(z_k^0)|}{r_k^{\beta}E_k}\\
\le C\frac{\norm{u_k}_{C^{0,\beta}(B_{4/5},\mathbb C)}|\tilde{z}_k+r_k z -z_k^0|^\beta}{r_k^{\beta}E_k}
\le C\frac{\norm{u_k}_{C^{0,\beta}(B_{4/5},\mathbb C)}}{E_k}\to 0^+,
\end{multline}
since $|\tilde{z}_k+r_k z -z_k^0| \le |y_k^0|+r_k|z|\le Cr_k+\e_k \le Cr_k$.

\textbf{Step 5. $v_\infty$ satisfies homogeneous Dirichlet boundary conditions in Cases 2 and 3.}
In \textbf{Case 3}, since $v_k=0$ on $\partial \Sigma_{\e_k/r_k,N_k}$ from the uniform convergence of $v_k \to v_\infty$ we conclude that $v_\infty=0$ on $\partial \Sigma_{\bar \e, N_\infty}$, where $\bar\e=\lim_{k\to\infty}\e_k/r_k$ as defined in \eqref{bareps}.

In \textbf{Case 2}, the situation is far more delicate since $v_k \not \equiv 0$ on $\partial \Sigma_{\e_k/r_k,N_k}$. 
Take any $z \in \partial (\Sigma_{\e_k,N_k}-\tilde{z}_k)/r_k$, and let $P_kz$ be the projection of $z$ on the tangent hyperplane to $\partial \Sigma_{\e_k,N_k}$ in $z_k^0$.

Since $\tilde{z}_k=z_k^0 \in  \Sigma_{\e_k,N_k}$, it follows that $\nabla (\eta u_k)(z_k^0)\cdot z=\nabla (\eta u_k)(z_k^0)\cdot (z-P_kz)$ and so 
\begin{equation}\label{proof_theorem_reg_C2alpha_71}
|v_k(z)|=\frac{|\nabla (\eta u_k)(z_k^0)\cdot z|}{r_k^{\beta}E_k}\le \frac{|\nabla (\eta u_k)(z_k^0)|}{r_k^{\beta}E_k}|z-P_kz| 
\le C \left(\frac{\e_k}{r_k}\right)^{\beta}|z-P_kz| 
\end{equation}
thanks to \eqref{proof_theorem_reg_C2alpha_2_7}.
We claim that
\begin{equation}\label{proof_theorem_reg_C2alpha_3}
|z-P_kz| \le C \frac{r_k}{\e_k},
\end{equation}
which, combined with the  uniform convergence of $v_k \to v_\infty$ on compact sets, \eqref{relation_er}  and \eqref{proof_theorem_reg_C2alpha_71}, will yield 
$v_\infty\equiv0$ on $\partial \Pi_\nu$.

To prove \eqref{proof_theorem_reg_C2alpha_3}, let $\Psi_k:\R^d \to \R$
\begin{equation}\label{def_Psi_k}
\Psi_k(z):= \sqrt{(N_k)_3^{-1}(z)y\cdot y}
\end{equation}
so that 
\begin{equation}
\Psi_k(z)=\e_k \text{ for any } z \in \partial \Sigma_{\e_k,N_k}.
\end{equation}
We can compute $\nabla \Psi_k$ as 
\begin{equation}
\nabla \Psi_k(z)=\frac{1}{2}\frac{(0,(N_k)_3^{-1}(z)y)}{\Psi_k(z)}+\frac{G_k(z)}{\Psi_k(z)}
\end{equation}
where $G_k:\R^d \to \R$ is defined as
\begin{equation}\label{def_G_k}
G_k(z):=\frac{1}{2}\left(\sum_{i,j=1}^n \pd{(q_k)_{i,j}}{z_l}y_iy_j\right)_{l=1, \dots, d} \text{ where } (N_k)_3^{-1}(z)=:(q_k)_{i,j=1, \dots n}.
\end{equation}
In particular,  $G_k \in C^{0,1}(B_1)\cap C^1(B_1\setminus \Sigma)$ and $\norm{\nabla \Psi_k}_{L^\infty(B_1)}\le C_1L$, for some positive constant $C_1>0$ that does not depend on $k$ in view of the uniform ellipticity of $N_k$.

We notice that  the outer normal vector to $\partial \Sigma_{\e_k,N_k}$ in $\tilde{z}_k$ is given by
\begin{equation}
\nu_k:=\frac{\nabla \Psi_k(\tilde{z}_k)}{|\nabla \Psi_k(\tilde{z}_k)|}
=\frac{1}{|\nabla \Psi_k(\tilde{z}_k)|}\left[\frac{1}{2}\frac{(0,(N_k)_3^{-1}(\tilde{z}_k)y_k^0)}{\Psi_k(z_k^0)}+\frac{G_k(z_k^0)}{\Psi_k(z_k^0)}\right].
\end{equation}
Furthermore, $\Psi_k(z_k^0)=\e_k$ and
\begin{equation}
|\nabla \Psi_k(\tilde{z}_k)| \ge  \frac{C}{\e_k}(\e_k-\e_k^2) \ge C_1
\end{equation}
for some constant $C_1>0$ that does not depend on $k$, since $z_k^0 \in \partial \Sigma_{\e_k,N_k}$.  Moreover, 
\begin{equation}
(N_k)_3^{-1}(z_k^0+r_kz)(y_k^0+r_ky)\cdot (y_k^0+r_ky)=\e_k^2 \quad  \text{ or equivalently } \quad \Psi_k(z_k^0+r_kz)=\e_k
\end{equation}
since $z \in \partial (\Sigma_{\e_k,N_k}-z_k^0)/r_k$. Hence, by the Lagrange Theorem there exists $z_k^*=z_k^0+\theta r_kz$ with $\theta \in [0,1]$ such that
\begin{multline}
|z-P_kz| =|z\cdot \nu_k| \le C\left|\frac{\Psi_k(z_k^0+r_kz)-\Psi_k(z_k^0)}{r_k}-\left[\frac{1}{2}\frac{(0,(N_k)_3^{-1}(z_k^0)y_k^0)}{\Psi_k(z_k^0)}+\frac{G_k(z_k^0)}{\Psi_k(z_k^0)}\right]\cdot z\right| \\
\le C\left|\left[\frac{1}{2}\frac{(0,(N_k)_3^{-1}(z_k^*)y_k^*)}{\Psi_k(z_k^*)}+\frac{G_k(z_k^*)}{\Psi_k(z_k^*)}\right]\cdot z
-\left[\frac{1}{2}\frac{(0,(N_k)_3^{-1}(z_k^0)y_k^0)}{\Psi_k(z_k^0)}+\frac{G_k(z_k^0)}{\Psi_k(z_k^0)}\right]\cdot z\right|
\end{multline}
By \eqref{def_G_k}, letting
\begin{equation}
Q_k:=\frac{1}{2}\left(\sum_{i,j=1}^n \pd{(q_k)_{i,j}}{z_l}\right)_{l=1, \dots, d},
\end{equation}
 we thus have
\begin{multline}
\left|\frac{G_k(z_k^*)}{\Psi_k(z_k^*)}-\frac{G_k(z_k^0)}{\Psi_k(z_k^0)}\right| 
\le \left|\frac{Q_k(z_k^*) y_k^*\cdot y_k^*}{\Psi_k(z_k^*)}-\frac{Q_k(z_k^0) y_k^*\cdot y_k^*}{\Psi_k(z_k^0)}\right|
+\left|\frac{Q_k(z_k^0) y_k^*\cdot y_k^*}{\Psi_k(z_k^0)}-\frac{Q_k(z_k^0) y_k^0\cdot y_k^0}{\Psi_k(z_k^0)}\right|\\
\le \left|\frac{Q_k(z_k^*) y_k^*\cdot y_k^*}{\Psi_k(z_k^*)}-\frac{Q_k(z_k^0) y_k^*\cdot y_k^*}{\Psi_k(z_k^*)}\right|
+\left|\frac{Q_k(z_k^0) y_k^*\cdot y_k^*}{\Psi_k(z_k^*)}-\frac{Q_k(z_k^0) y_k^*\cdot y_k^*}{\Psi_k(z_k^0)}\right|+\frac{C}{\e_k}|y_k^*-y_k^0|\\
\le C\left(|z_k^*-z_k^0| \frac{|y_k^*|^2}{|\Psi(z_k^*)|}+c |y_k^*|\frac{|\Psi_k(z_k^0)-\Psi_k(z_k^*)|}{|\Psi_k(z_k^0)||\Psi_k(z_k^*)|}+\frac{r_k}{\e_k}\right)
\le C\left(r_k +\frac{r_k}{\e_k}\right).
\end{multline}
Then 
\begin{multline}\label{ineq_N31_ek_rk}
\left|\frac{(N_k)_3^{-1}(z_k^*) y_k^*}{\Psi_k(z_k^*)}-\frac{(N_k)_3^{-1}(z_k^0) y_k^0}{\Psi_k(z_k^0)}\right|\\
\le \left|\frac{(N_k)_3^{-1}(z_k^*) y_k^*}{\Psi_k(z_k^*)}-\frac{(N_k)_3^{-1}(z_k^0) y_k^*}{\Psi_k(z_k^0)}\right|
+\left|\frac{(N_k)_3^{-1}(z_k^0) y_k^*}{\Psi_k(z_k^0)}-\frac{(N_k)_3^{-1}(z_k^0) y_k^0}{\Psi_k(z_k^0)}\right|\\
\le \left|\frac{(N_k)_3^{-1}(z_k^*) y_k^*}{\Psi_k(z_k^*)}-\frac{(N_k)_3^{-1}(z_k^0) y_k^*}{\Psi_k(z_k^*)}\right|
+\left|\frac{(N_k)_3^{-1}(z_k^0) y_k^*}{\Psi_k(z_k^*)}-\frac{(N_k)_3^{-1}(z_k^0) y_k^*}{\Psi_k(z_k^0)}\right|+\frac{C}{\e_k}|y_k^*-y_k^0|\\
\le C\left(|z_k^*-z_k^0| \frac{|y_k^*|}{|\Psi(z_k^*)|}+c |y_k^*|\frac{|\Psi_k(z_k^0)-\Psi_k(z_k^*)|}{|\Psi_k(z_k^0)||\Psi_k(z_k^*)|}+\frac{r_k}{\e_k}\right)
\le C\left(r_k +\frac{r_k}{\e_k}\right).
\end{multline}
In conclusion, we have proved  \eqref{proof_theorem_reg_C2alpha_3}.

\textbf{Step 6. Boundedness in $H^1$.}
With a change of variables in \eqref{eq_ue}, we can see that for any $\varphi \in C^\infty_c(\Omega_\infty,\mb{C})$,  for any $R>0$ such that the support of $\varphi$ is contained in $B_R$ and  for large enough $k$, in \textbf{Cases 1} and \textbf{2}
\begin{multline}\label{proof_theorem_reg_C2alpha_4}
\int_{\Omega_k} \Bigg[i M_k(\tilde{z}_k+r_kz)\nabla w_k(z)+\left(r_k\frac{a_k(\tilde{x}_k+r_kx,(\tilde{y}_k+r_ky)/|\tilde{y}_k+r_ky|)}{|\tilde{y}_k+r_ky|}
+ r_kb_k(\tilde{z}_k+r_kz)\right)w_k(z)\\
+ i  M_k(\tilde z_k+r_kz) \frac{r_k^{-\beta}\eta(\tilde z_k)\nabla u_k(\tilde z_k)}{E_k}
+\left(\frac{a_k(\tilde{x}_k+r_kx,(\tilde{y}_k+r_ky)/|\tilde{y}_k+r_ky|)}{|\tilde{y}_k+r_ky|} +b_k(\tilde{z}_k+r_kz)\right)\\
\times\left(\frac{r_k^{-\beta}\eta(\tilde z_k)u_k(\tilde z_k)}{E_k}+  \frac{r_k^{1-\beta}\eta(\tilde z_k)\nabla u_k(\tilde z_k)\cdot z}{E_k}\right)\Bigg] \\
\cdot \left(-i M_k(\tilde{z}_k+r_kz) \nabla \overline{\varphi}(z)
+ \left(r_k\frac{a_k(\tilde{x}_k+r_kx,(\tilde{y}_k+r_ky)/|\tilde{y}_k+r_ky|)}{|\tilde{y}_k+r_ky|}
+ r_kb_k(\tilde{z}_k+r_kz)\right)\overline{\varphi}(z)\right) \\
+r_k^2\frac{g_k(\tilde{x}_k+r_kx,(\tilde{y}_k+r_ky)/|\tilde{y}_k+r_ky|)}{|\tilde{y}_k+r_ky|^2}w_k(z) \overline{\varphi}(z)\\
+\frac{g_k(\tilde{x}_k+r_kx,(\tilde{y}_k+r_ky)/|\tilde{y}_k+r_ky|)}{|\tilde{y}_k+r_ky|^2}\\
\times\left(\frac{r_k^{1-\beta}\eta(\tilde z_k)u_k(\tilde{z}_k)}{E_k}
+\frac{r_k^{2-\beta}\eta(\tilde z_k)\nabla u_k(\tilde z_k)\cdot z}{E_k}\right)
\overline{\varphi}(z)\Bigg]\, dz=0.
\end{multline}
In \textbf{Case 3}
\begin{multline}\label{proof_theorem_reg_C2alpha_5}
\int_{\Omega_k} i M_k(\tilde{z}_k+r_kz)\nabla w_k(z)+\left(r_k\frac{a_k(\tilde{x}_k+r_kx,(\tilde{y}_k+r_ky)/|\tilde{y}_k+r_ky|)}{|\tilde{y}_k+r_ky|}
+ r_kb_k(\tilde{z}_k+r_kz)\right)w_k(z)\\
\cdot \Big(-i M_k(\tilde{z}_k+r_kz) \nabla \overline{\varphi}(z)\\
+ \left(r_k\frac{a_k(\tilde{x}_k+r_kx,(\tilde{y}_k+r_ky)/|\tilde{y}_k+r_ky|)}{|\tilde{y}_k+r_ky|}
+ r_kb_k(\tilde{z}_k+r_kz)\right)\overline{\varphi}(z)\Big) \, dz=0.
\end{multline}
We now prove that $\{w_k\}_{k \in \mb N \setminus \{0\}}$ is bounded in $\tilde{H}^1(\Omega_\infty\cap B_R,\mb{C})$ for any $R>0$  so that,
in particular, $w_k \rightharpoonup v_\infty$ weakly in $\tilde{H}^1(\Omega_\infty\cap B_R,\mb{C})$ as $k \to \infty$ for any $R>0$.

Since \eqref{proof_theorem_reg_C2alpha_4} is in the form of \eqref{ineq_caccio_simple} after some multiplications, it is enough to estimate each  terms on the right hand side of   \eqref{ineq_caccio_simple} in a suitable manner. To this end, it's clearly sufficient to control the $L^\infty$ norms of all the terms involved since we are integrating over  $B_{2R}\cap \Omega_k$ which has finite measure.

Let us start with \textbf{Case 1} and some preliminary considerations. In this case $\tilde z_k=z_k$ thus, since $u_k(z_k^0)=0$, by \eqref{proof_theorem_reg_C2alpha_2_7}
\begin{equation}\label{proof_theorem_reg_C2alpha_6}
|u_k(z_k)| \le C |\nabla u_k(z_k^0)| |z_k-z_k^0| \le C E_kd_k\e_k^\beta \le C E_k|y_k|^{1+\beta}
\end{equation}
and similarly
\begin{equation}\label{proof_theorem_reg_C2alpha_6_5}
|\nabla u_k(z_k)| \le |\nabla u_k(z_k)-\nabla u_k(z_k^0)|+|\nabla u_k(z_k^0)| \le E_kd_k^\beta+CE_k \e_k^\beta\le CE_k|y_k|^{\beta}.
\end{equation}
Furthermore, as in Theorem \ref{theorem_reg_C0alpha},
\begin{equation}
|y_k+r_ky| \ge \frac{1}{2} |y_k|,
\end{equation}
since $d_k/r_k \to +\infty$,  thus
\begin{equation}\label{proof_theorem_reg_C2alpha_7}
\frac{|a_k(x_k+r_kx,(y_k+r_ky)/|y_k+r_ky|)|}{|y_k+r_ky|}
\le \frac{C}{|y_k|}.
\end{equation}
Moreover, since $\{a_k\}$ are equi-H\"older continuous,
\begin{multline}\label{proof_theorem_reg_C2alpha_8}
\left|\frac{a_k(x_k+r_kx,(y_k+r_ky)/|y_k+r_ky|)}{|y_k+r_ky|}-\frac{a_k(x_k,y_k/|y_k|)}{|y_k|}\right|\\
=\frac{1}{|y_k+r_ky||y_k|}\left||y_k|a_k(x_k+r_kx,(y_k+r_ky)/|y_k+r_ky|)-|y_k+r_ky|a_k(x_k,y_k/|y_k|)\right| \\
\le C\frac{||y_k|-|y_k+r_ky||}{|y_k+r_ky||y_k|}+\frac{|a_k(x_k+r_kx,(y_k+r_ky)/|y_k+r_ky|)-a_k(x_k,y_k/|y_k|)|}{|y_k|}\\
\le C\frac{r_k}{|y_k|^2}+C\frac{r_k^{\bar \alpha} +\left|\frac{(y_k+r_ky)}{|y_k+r_ky|}-\frac{y_k}{|y_k|}\right|^{\bar \alpha}}{|y_k|}\\
\le C\frac{r_k}{|y_k|^2}+C\frac{r_k^{\bar \alpha}}{|y_k|}+C\frac{\left||y_k|(y_k+r_ky)-|y_k+r_ky|y_k\right|^{\bar \alpha}}{|y_k|^{1+2\bar \alpha}}\\
\le C\frac{r_k}{|y_k|^2}+C\frac{r_k^{\bar \alpha}}{|y_k|}+C\frac{\left|r_k |y_k| y +||y_k|-|y_k+r_ky|| y_k\right|^{\bar \alpha}}{|y_k|^{1+2\bar \alpha}}
\le  C\frac{r_k}{|y_k|^2}+C\frac{r_k^{\bar \alpha}}{|y_k|^{1+\bar \alpha}}.
\end{multline}
Finally, since we are testing \eqref{proof_theorem_reg_C2alpha_4} with test functions $\varphi$ with compact support, we may add terms of the form 
\begin{equation}\label{proof_theorem_reg_C2alpha_9}
\int_{\Omega_k}\nu\cdot \nabla \overline{\varphi} \, dz =0
\end{equation}
to \eqref{proof_theorem_reg_C2alpha_4} for any  $\nu \in \R^d$.

Using this last idea with $\nu =-N_k(z_k) \nabla u_k( z_k)$, we may estimate, by \eqref{proof_theorem_reg_C2alpha_6_5}
\begin{equation}\label{proof_theorem_reg_C2alpha_alpha_crit_1}
\frac{r_k^{-\beta}\eta(z_k)|\nabla u_k( z_k)|}{E_k}
|N_k(z_k+r_kz)-N_k(z_k)|\le Cr_k^{\alpha-\beta}.
\end{equation}
Furthermore,  by \eqref{proof_theorem_reg_C2alpha_6}, \eqref{proof_theorem_reg_C2alpha_8} and \eqref{proof_theorem_reg_C2alpha_9} with $\nu= -\frac{a_k(x_k,y_k/|y_k|)}{|y_k|}$,
\begin{multline}
\left|M_k(z_k+r_kz) \frac{\eta(z_k)r_k^{-\beta}u_k(z_k)}{E_k}\right|
\cdot\left|\frac{a_k(x_k+r_kx,(y_k+r_ky)/|y_k+r_ky|)}{|y_k+r_ky|} -\frac{a_k(x_k,y_k/|y_k|)}{|y_k|}\right|\\
\le Cr_k^{-\beta}|y_k|^{1+\beta}\left[\frac{r_k}{|y_k|^2}+C\frac{r_k^{\bar \alpha}}{|y_k|^{1+\bar \alpha}}\right] 
\le C\left[\frac{r_k^{1-\beta}}{|y_k|^{1-\beta}}+\frac{r_k^{\bar \alpha -\beta}}{|y_k|^{\bar \alpha -\beta}}\right],
\end{multline}
while, similarly, by  \eqref{proof_theorem_reg_C2alpha_6_5}, \eqref{proof_theorem_reg_C2alpha_8} and \eqref{proof_theorem_reg_C2alpha_9} 
\begin{multline}
\left|M_k(z_k+r_kz) \frac{r_k^{1-\beta}\eta(z_k)\nabla u_k(z_k)}{E_k}
\cdot\left[\frac{a_k(x_k+r_kx,(y_k+r_ky)/|y_k+r_ky|)}{|y_k+r_ky|} -\frac{a_k(x_k,y_k/|y_k|)}{|y_k|}\right]\right|\\
\le Cr_k^{1-\beta}|y_k|^{\beta}\left[\frac{r_k}{|y_k|^2}+C\frac{r_k^{\bar \alpha}}{|y_k|^{1+\bar \alpha}}\right] 
\le C\left[\frac{r_k^{2-\beta}}{|y_k|^{2-\beta}}+\frac{r_k^{1+\bar \alpha -\beta}}{|y_k|^{1+\bar \alpha -\beta}}\right].
\end{multline}
In simpler fashion, by \eqref{proof_theorem_reg_C2alpha_6},  \eqref{proof_theorem_reg_C2alpha_6_5} and   \eqref{proof_theorem_reg_C2alpha_9} with $\nu= -b_k(z_k)$,
\begin{equation}\label{proof_theorem_reg_C2alpha_alpha_crit_2}
\left|M_k(z_k+r_kz) \frac{\eta(z_k)[r_k^{-\beta}u_k(z_k)+r_k^{1-\beta}\nabla u(z_k)]}{E_k}
\cdot\left(b_k(z_k+r_kz)-b_k(z_k)\right)\right|\le Cr_k^{\alpha-\beta}.
\end{equation}
Moreover, by  \eqref{proof_theorem_reg_C2alpha_6_5} and \eqref{proof_theorem_reg_C2alpha_7}
\begin{equation}
\left|M_k(z_k+r_kz) \frac{\eta(z_k)r_k^{1-\beta}\nabla u_k(z_k)}{E_k}
\cdot \frac{a_k(x_k+r_kx,(y_k+r_ky)/|y_k+r_ky|)}{|y_k+r_ky|}\right|\le C\frac{r_k^{1-\beta}}{|y_k|^{1-\beta}}.
\end{equation}
By \eqref{proof_theorem_reg_C2alpha_6} and \eqref{proof_theorem_reg_C2alpha_7}
\begin{equation}
\left|\frac{r_k^{1-\beta}\eta(z_k)u_k(z_k)}{E_k}
\frac{|a_k(x_k+r_kx,(y_k+r_ky)/|y_k+r_ky|)|^2}{|y_k+r_ky|^2}\right|
\le C\frac{r_k^{1-\beta}}{|y_k|^{1-\beta}}
\end{equation}
and similarly  thanks to \eqref{proof_theorem_reg_C2alpha_6_5} and \eqref{proof_theorem_reg_C2alpha_7}
\begin{equation}
\left|\frac{r_k^{2-\beta}\eta(z_k)\nabla u_k(z_k)\cdot z}{E_k}
\frac{|a_k(x_k+r_kx,(y_k+r_ky)/|y_k+r_ky|)|^2}{|y_k+r_ky|^2}\right|
\le C\frac{r_k^{2-\beta}}{|y_k|^{2-\beta}}.
\end{equation}
In the exact same way  by \eqref{proof_theorem_reg_C2alpha_6} and \eqref{proof_theorem_reg_C2alpha_7}
\begin{equation}
\left|\frac{r_k^{1-\beta}\eta(z_k)u_k(z_k)}{E_k}
\frac{a_k(x_k+r_kx,(y_k+r_ky)/|y_k+r_ky|)\cdot b_k(z_k)}{|y_k+r_ky|}\right|
\le Cr_k^{1-\beta}
\end{equation}
and by \eqref{proof_theorem_reg_C2alpha_6_5} and \eqref{proof_theorem_reg_C2alpha_7}
\begin{equation}
\left|\frac{r_k^{2-\beta}\eta(z_k)\nabla u_k(z_k)\cdot z}{E_k}
\frac{a_k(x_k+r_kx,(y_k+r_ky)/|y_k+r_ky|) \cdot b_k(z_k)}{|y_k+r_ky|}\right|
\le C\frac{r_k^{2-\beta}}{|y_k|^{1-\beta}}.
\end{equation}
Furthermore, by  \eqref{proof_theorem_reg_C2alpha_6_5} 
\begin{equation}
\left|M_k(z_k+r_kz) \frac{\eta(z_k)r_k^{1-\beta}\nabla u_k(z_k)}{E_k} \cdot b_k(z_k+r_kz)\right|
\le Cr_k^{1-\beta},
\end{equation}
and also thanks to \eqref{proof_theorem_reg_C2alpha_6},
\begin{equation}
\left|\frac{\eta(z_k)[r_k^{1-\beta}u_k(z_k)+r_k^{1-\beta}\nabla u_k(z_k)\cdot z]}{E_k}
|b_k(z_k+r_kz)|^2\right|\le Cr_k^{1-\beta}.
\end{equation}
We conclude that $\{w_k\}_{k \in \mb N \setminus \{0\}}$ is bounded in $\tilde{H}^1(\Omega_\infty\cap B_R,\mb{C})$ for any $R>0$ in \textbf{Case 1}.
We can proceed in a similar manner in \textbf{Case 2}
while in \textbf{Case 3}  by  \eqref{ineq_caccio_simple} and \eqref{proof_theorem_reg_C2alpha_5} we directly get
\begin{equation}
\norm{\nabla w_k}_{L^2(B_R\cap\Omega_k,\mb{C})} \le C \norm{w_k}_{L^2(B_{2R}\cap \Omega_k,\mb{C})}.
\end{equation}

\textbf{Step 7. Passing to the limit.}
Just as in Theorem \ref{theorem_reg_C0alpha} the computation made in \textbf{Step 5} allow us to pass to the limit in \eqref{proof_theorem_reg_C2alpha_4} and \eqref{proof_theorem_reg_C2alpha_5}. Distinguishing the $3$ cases we obtain the following.

\textbf{Case 1.} In this case, passing to the limit as $k \to \infty$ in \eqref{proof_theorem_reg_C2alpha_5},
we obtain, 
\begin{equation}
\int_{\R^d}   M_\infty\nabla v_\infty \cdot  M_\infty\nabla \overline\varphi \, dz=0
\end{equation}
for any $\varphi \in C_c^\infty(\R^d,\mb{C})$.

\textbf{Case 2.}
Passing to the limit  as $k \to \infty$ in \eqref{proof_theorem_reg_C0alpha_3_5}, we have  
\begin{equation}
\int_{\Pi_\nu} M_\infty\nabla v_\infty \cdot M_\infty\nabla \overline{\varphi} \, dz=0
\end{equation}
for any $\varphi \in C_c^\infty(\Pi_\nu, \mb{C})$ and $v_\infty=0$ on $\partial \Pi_\nu$ by \textbf{Step 5}.

\textbf{Case 3} 
In this case, since $\tilde{y}_k=0$,  up to passing to a subsequence,
\begin{equation}
a_k(\tilde{x}_k+r_kx,y/|y|) \to a_\infty( y/|y|)  \quad \text{ uniformly in  } B_R,
\end{equation}
for some   $a_\infty \in C^{0,\beta}(B_R,\R^d)$. Then passing to the limit as $k \to \infty$ in \eqref{proof_theorem_reg_C2alpha_5}
\begin{equation}\
\int_{\R^d}\left(i  M_\infty\nabla v_\infty+\frac{a_\infty(y/|y|))}{|y|}v_\infty\right)
\cdot\left(\overline{i  M_\infty\nabla \varphi+\frac{a_\infty(y/|y|))}{|y|}\varphi}  \right)\, dz=0
\end{equation}
for any  $\varphi \in C^\infty_c(\R^d\setminus \Sigma_{\bar \e,N_\infty},\mb{C})$  with $\bar \e \ge 0$ and  $\Sigma_{\bar \e,N_\infty}$ as in \eqref{proof_theorem_reg_C0alpha_3}.

Hence, arguing as in Theorem  \ref{theorem_reg_C0alpha}, in all cases $v_\infty\equiv 0$ by Theorem \ref{theor_liou_y}  
which contradicts \eqref{proof_theorem_reg_C2alpha_2_8}.

To finish the proof we need to  sharpen \eqref{ineq_reg_C1alpha}, which for the moment has been proved for any $\beta \in(0,\alpha)$, to the the case $\beta=\alpha$. 
We can do the same contradiction argument, which can be adapted simply substituting $\alpha$ to $\beta$ from \textbf{Step 1} to \textbf{Step 5}. Concerning \textbf{Step 6}, by 
\eqref{ineq_reg_C1alpha} with any $\beta<\alpha$, we must have 
\begin{equation}\label{proof_theorem_reg_C1alpha_1}
E_k\ge C k \norm{u_k}_{L^2(B_1,\mb{C})}.
\end{equation}
It follows that, by \eqref{ineq_reg_C1alpha} with any $\beta<\alpha$, for any $\gamma \in (0,1)$
\begin{equation}\label{proof_theorem_reg_C1alpha_1_5}
\norm{u_k}_{C^{0,\gamma}(B_{1/2},\mb{C})} +\norm{ \nabla u_k}_{L^\infty(B_{1/2}\setminus\Sigma_{\e_k,N_k},\mb{C})} \le Ck^{-1} E_k.
\end{equation}
Using \eqref{proof_theorem_reg_C1alpha_1_5}, estimate \eqref{proof_theorem_reg_C2alpha_alpha_crit_1} becomes
\begin{equation}
\frac{r_k^{-\beta}\eta(\tilde z_k)|\nabla u_k(\tilde z_k)|}{E_k}
|N_k(\tilde z_k+r_kz)-N_k(\tilde z_k)|\le \frac{C}{k},
\end{equation}
while \eqref{proof_theorem_reg_C2alpha_alpha_crit_2}
\begin{equation}
\left|M_k(\tilde z_k+r_kz) \frac{\eta(\tilde z_k)[r_k^{-\beta}u_k( \tilde z_k)+r_k^{1-\beta}\nabla u( \tilde z_k)]}{E_k}
\cdot\left(b_k(\tilde z_k+r_kz)-b_k(\tilde z_k)\right)\right|\le \frac{C}{k}.
\end{equation}
Hence, we can also show the validity of \textbf{Step 6}. Since  \textbf{Step 7} needs no modifications, we have completed the proof of \eqref{ineq_reg_C1alpha}.
Finally, \eqref{ineq_reg_C1alpha} and Lemma \ref{lemma_gradient_u} readily imply \eqref{ineq_nabla_C1alpha}.
\end{proof}

\subsection{A priori estimates, approximation and proof of the main results}\label{subsec_proof_main}

The next two results are needed to prove Theorems \ref{theorem_C0alpha_main} and \ref{theorem_C1alpha_main} without any extra assumption on the regularity of the block $N_3$ of $N$. The proof is a much simpler version of the one of Theorems \ref{theorem_reg_C0alpha}  and \ref{theorem_reg_C1alpha}, since we are not perforating the domain, and hence we skip it.

\begin{proposition}\label{proposition_reg_C0alpha_matrix}
Let $R >0$  and $r \in (0,R)$. Suppose that  $q >d$,    $\gamma_1(M,a) \ge \gamma_1^* >0$  with  $\gamma_1(M,a)$ as in \eqref{def_gamma} and 
\begin{equation}
\alpha\in (0,1) \cap (0,\gamma_1^*) \cap (0, 1-d/q].
\end{equation}
Let $M$ be  continuous with modulus of continuity $\omega_1$  and assume that  \eqref{hp_M_structure} holds.   Suppose that $a$ is continuous with modulus of continuity $\omega_2$.
Let $b \in L^q(B_R,\R^d)$ with $q >d$. Finally let $L$ be such that
\begin{equation}
\norm{M}_{C^{0,\omega_1}(B_R,\R^{d\times d})}+\norm{a}_{C^{0,\omega_2}(B'_R\times\mathbb S^{n-1},\R^{d})} + \norm{b}_{L^q(B_R,\R^d)} \le L
\end{equation}
and let  $u \in C^{0,\alpha}(B_r,\mb{C})$ for any $r<R$  be a solution of \eqref{eq_magnetic_strong} with 
\begin{equation}\label{eq_C0alpha_u0}
u=0 \quad \text{ on } \Sigma \cap B_R.
\end{equation}
Then, there exists $C>0$, depending only on  $n,d,q, \la, \Lambda, L$, $r,R,\alpha$, and $\gamma_1^*$,  such that  
\begin{equation}
\norm{u}_{C^{0,\alpha}(B_r,\mb{C})} \le C \norm{u}_{L^2(B_R,\mb{C})}.
\end{equation}
\end{proposition}

\begin{proposition}\label{proposition_reg_C1alpha_matrix}
Let $R>0$ and $r \in (0,R)$.  Suppose that $\gamma_1(M,a) \ge \gamma_1^* >0$ for 
some $\gamma_1^*>1$, with  $\gamma_1(M,a)$ as in \eqref{def_gamma}. 
Let $\alpha \in (0,1)\cap (0,\gamma_1^*-1)$.

Assume that $M \in C^{0,\alpha}(B_R,\R^{d,d})$, that  \eqref{hp_M_structure} holds and that $a \in C^{0,\bar \alpha}(B'_R\times\mathbb S^{n-1},\R^d)$.
Suppose that $b \in  C^{0,\alpha}(B_R,\R^d)$.
Finally let $L$ be such that
\begin{equation}
\norm{M}_{C^{0,\alpha}(B_R,\R^{d\times d})}+\norm{a}_{C^{0, \bar \alpha}(B'_R\times\mathbb S^{n-1},\R^{d})} + \norm{b}_{C^{0,\alpha}(B_R,\R^{d})}\le L
\end{equation}
and let  $u \in C^{1,\alpha}(B_r,\mb{C})$ for any $r<R$  be a solution of \eqref{eq_magnetic_strong} with
\begin{equation}\label{eq_C1alpha_u0_nablau0}
u=0 \text{ and } \nabla u=0 \quad \text{ on } \Sigma \cap B_R.
\end{equation}
Then, there exists $C>0$, depending only on  $n,d, \la, \Lambda, L, r,R,\alpha$, and $\gamma_1^*$,  such that
\begin{equation}\label{ineq_reg_C1alpha_matrix}
\norm{u}_{C^{1,\alpha}(B_r,\mb{C})} \le C \norm{u}_{L^2(B_R,\mb{C})}.
\end{equation}
\end{proposition} 

To prove Theorems \ref{theorem_C0alpha_main}--\ref{theorem_C1alpha_main} we need an approximation lemma on the matrix $M$.
\begin{lemma}\label{lemma_approximation_smooth}
Let  $M \in C^{0,\omega}(B_R,\R^{d,d})$ for some modulus of continuity $\omega$ and suppose that $M$ satisfies \eqref{hp_M_structure}. Let $u$ be a solution of \eqref{eq_magnetic_strong} in $B_R$.
Let $r \in (0,R)$ and let $\{M_k\}_{k \in \N} \subset C^\infty(B_r,\R^{d,d})$, uniformly bounded in $C^{0,\omega}(B_r,\R^{d,d})$ be such that $M_k \to M$ in $L^{\infty}(B_r,\R^{d,d})$. For any $k \in \N$ the problem  
\begin{equation}\label{prob_approx_matrix}
\begin{cases}
(iM_k\nabla+A)^2 v= 0,  &\text{ in } B_r,\\
v=u, &\text{ on } \partial B_r,
\end{cases}
\end{equation}
admits a unique weak solution $u_k \in  \tilde{H}^1_0(B_r,\mb{C})$ and 
\begin{equation}
u_k \to u  \text{ strongly  in }  \tilde{H}^1_0(B_r,\mb{C}) \quad \text{ as } k \to \infty.
\end{equation}
\end{lemma}
\begin{proof}
The existence and uniqueness of $u_k$ can be proved by standard variational methods. 
Indeed we can simply minimize the functional 
\begin{equation}
J_k:\tilde{H}^1_0(B_r,\mb{C}) \to [0,+\infty), \quad  J_k(w):=\int_{B_r} |iM_k\nabla w+Aw|^2 \, dz,
\end{equation}
over the convex set $\{w\in \tilde{H}^1_0(B_r,\mb{C}): w=u \text{ on } \partial B_r\}$.
Furthermore, \eqref{hp_M_elliptic} holds with constants uniform with respect to $k$ since $\{M_k\}$ is bounded in  $C^{0, \omega}(B_r,\R^{d,d})$. Thus, ${u_k}$ is bounded in $\tilde{H}^1_0(B_r,\mb{C})$. Hence, up to passing to a subsequence, it converges weakly in  $\tilde{H}^1_0(B_r,\mb{C})$ to some $u_\infty \in \tilde{H}^1_0(B_r,\mb{C})$. Passing to the limit in the weak formulation of \eqref{prob_approx_matrix}, we conclude that $u_\infty$ is a solution of \eqref{eq_magnetic_weak} and that $u=u_\infty$ on $\partial B_r$. Thus, $u=u_\infty$. Finally, testing with $u-u_k$ the equation of $u_k$ we obtain 
\begin{equation}
\int_{B_r}|iM_k\nabla u_k+Au_k|^2 \, dz=
\int_{B_r}(iM_k\nabla u_k+Au_k) \cdot\overline{(iM_k\nabla u+Au)}   \, dz \to \int_{B_r}|iM\nabla u+Au|^2  \, dz.
\end{equation}
Thus, $u_k \to u$ strongly in  $\tilde{H}^1_0(B_r,\mb{C})$ as $k \to \infty$. By the Urysohn subsequence principle such convergence holds as $k \to\infty$ and not just along a subsequence.
\end{proof}

\begin{proof}[\textbf{Proof of Theorem \ref{theorem_C0alpha_main}}]
Suppose that $M \in C^\infty(B_R,\R^{d,d})$.  Let $u_\e$ be the weak solution of \eqref{prob_ue}. Then by Theorem \ref{theorem_reg_C0alpha}, for any $r<R$ 
\begin{equation}
\norm{u_\e}_{C^{0,\alpha}(B_r,\C)} \le c \norm{u_\e}_{L^2(B_R,\C)}
\end{equation}
In particular, up to passing to a subsequence, $u_\e \to u$ in $C^{0,\beta}(B_r,\C)$ for any $\beta \in (0,\alpha)$, by Proposition \ref{prop_ue} and $u \in C^{0,\alpha}(B_r,\C)$. 
Furthermore  the uniform convergence of $u_\e$ to $u$ implies that $u(x,0)=0$ for any $(x,0) \in B_R$.
Hence, by Proposition \ref{proposition_reg_C0alpha_matrix}, we have shown that 
\begin{equation}
\norm{u}_{C^{0,\alpha}(B_r,\C)} \le c \norm{u}_{L^2(B_R,\C)}
\end{equation}
if  $M \in C^\infty(B_R,\R^{d,d})$. To complete the proof, let $M \in C^{0,\omega_1}(B_R,\R^{d,d})$, $r \in (0,R)$, $r_1 \in (r,R)$ and a sequence $M_k \to M$ in $ L^\infty(B_{r_1},\R^{d,d})$ bounded in $C^{0,\omega_1}(B_{r_1},\R^{d,d})$ with $\{M_k\}_{k \in \N} \subset C^\infty(B_{r_1},\R^{d,d})$. 
Letting $u_k$ as in Lemma \ref{lemma_approximation_smooth},
\begin{equation}
\norm{u_k}_{C^{0,\alpha}(B_r,\C)} \le c \norm{u_k}_{L^2(B_{r_1},\C)}.
\end{equation}
Hence,  up to passing to a subsequence, by Lemma \ref{lemma_approximation_smooth}, $u_k \to u$ in $C^{0,\beta}(B_r,\C)$ for any $\beta \in (0,\alpha)$ thus $u \in C^{0,\alpha}(B_r,\C)$.  Since $u_k(x,0)=0$ for any $(x,0) \in B_R$ the uniform convergence of $u_\e$ to $u$ yields $u(x,0)=0$ for any $(x,0) \in B_R$. Hence, by Proposition \ref{proposition_reg_C0alpha_matrix} also in this case
 \begin{equation}
\norm{u}_{C^{0,\alpha}(B_r,\C)} \le c \norm{u}_{L^2(B_R,\C)},
\end{equation}
which finishes the proof.
\end{proof}

\begin{proof}[\textbf{Proof of Theorem \ref{theorem_C1alpha_main}}]
Assume that $M \in C^\infty(B_R,\R^{d,d})$.  Letting $u_\e$ be  the weak solution of \eqref{prob_ue},  by Theorem \ref{theorem_reg_C1alpha}, for any $r<R$
\begin{equation}
\norm{u_\e}_{C^{1,\alpha}(B_r\setminus \Sigma_{\e,N},\C)} \le c \norm{u_\e}_{L^2(B_R,\C)}.
\end{equation}
Let $\delta>0$. Up to passing to a subsequence,  $u_\e \to u$ in $C^{1,\beta}(B_r\setminus \Sigma_{\delta,N},\C)$ for any $\beta \in (0,\alpha)$  and $u \in C^{1,\alpha}(B_r\setminus \Sigma_{\delta,N},\C)$, thanks to  Proposition \ref{prop_ue}. By a diagonal argument, $u \in C^{1,\alpha}(B_r,\C)$. 
Furthermore, $u_\e \to u$ in $C^{0,\beta}(B_r,\C)$ for any $\beta \in (0,1)$ so that
 $u(x,0)=0$ for any $(x,0) \in B_R$ since $u_\e(x,0)=0$. Letting $\delta>0$ and  $(x,y) \in \partial \Sigma_{\delta,N}$,
\begin{equation}
|\nabla u(x,0)|\le |\nabla u(x,0)-\nabla u(x,y)|+ |\nabla u(x,y)-\nabla u_\e(x,y)|+ |\nabla u_\e(x,y)| \le C [\nabla u]_{C^{1,\alpha}(B_r,\C)} \delta^{\alpha}+ \delta +C\delta^{\alpha}
\end{equation}
for $\e$ small enough taking \eqref{ineq_nabla_C1alpha} and the convergence of  $u_\e \to u$ in $C^{1,\beta}(B_r\setminus \Sigma_{\delta/2,N},\C)$ into account. Hence, also  $\nabla u(x,0)=0$ for any $(x,0) \in B_R$.
Proposition \ref{proposition_reg_C1alpha_matrix} then yields
\begin{equation}
\norm{u}_{C^{1,\alpha}(B_r,\C)} \le c \norm{u}_{L^2(B_R,\C)},
\end{equation}
under the  hypothesis  $M \in C^\infty(B_R,\R^{d,d})$. To complete the proof, let us consider  $M \in C^{0,\alpha}(B_R,\R^{d,d})$,  $r \in (0,R)$, $r_1 \in (r,R)$
and a sequence $M_k \to M$ in $ L^\infty(B_{r_1},\R^{d,d})$ bounded in $C^{0,\alpha}(B_{r_1},\R^{d,d})$ with $\{M_k\}_{k \in \N} \subset C^\infty(B_{r_1},\R^{d,d})$. 
If $u_k$ is as in Lemma \ref{lemma_approximation_smooth},
\begin{equation}
\norm{u_k}_{C^{1,\alpha}(B_r,\C)} \le c \norm{u_k}_{L^2(B_{r_1},\C)}.
\end{equation}
Hence, by Lemma \ref{lemma_approximation_smooth}, up to passing to a subsequence,  $u_k \to u$ in $C^{1,\beta}(B_r,\C)$ for any $\beta \in (0,\alpha)$ and $u \in C^{1,\alpha}(B_r,\C)$.  Since $u_k(x,0)=0$ and $\nabla u_k(x,0)=0$ for any $(x,0) \in B_R$ the $C^1$ convergence of $u_\e$ to $u$ implies that also  $u(x,0)=0$ and $\nabla u(x,0)=0$ for any $(x,0) \in B_R$. By Proposition \ref{proposition_reg_C0alpha_matrix}
 \begin{equation}
\norm{u}_{C^{1,\alpha}(B_r,\C)} \le c \norm{u}_{L^2(B_R,\C)},
\end{equation}
thus completing the proof.
\end{proof}

\appendix
\section{Physical models}\label{sec:solenoid} 
In this section, we compute the magnetic potentials generated by the current passing through shrinking solenoids with the Biot-Savart law to underline the physical relevance of a magnetic covariant gradient of the form $\nabla_mu=(M\nabla+iA)u$ with a magnetic potential as in \eqref{def_A} and an anisotropy matrix $M$.

\subsection{Shrinking solenoid with variable surface current density and general loop geometry}\label{sec:A1}
In this first section, we are going to compute the magnetic potential generated by an infinite solenoid with general geometry of the turns and possibly variable circulating current.
Consider a loop $\gamma(\theta)$ around the axis $\Sigma_0=\{y=0\}$; that is, $\gamma=(\gamma_1,\gamma_2,\gamma_3) \in  AC([0,2 \pi], \R^3)$ is a closed curve that represents a single turn of the coil. We are going to compute the magnetic potential at a point $z=(x,y_1,y_2)\in\R^3$ generated by a solenoid obtained by translating the curve $\gamma$ along $\Sigma_0$. In other words, for each $x\in\R$, we consider the translated loop $\gamma_x=(x+\gamma_1,\gamma_2,\gamma_3)$. Notice that $\gamma_0=\gamma$. One can refer to Figure \ref{fig:1}, which describes the particular case of planar loops. Let us stress the fact that any single loop does not need to be planar.

\begin{center}
\includegraphics[page=1,scale=1]{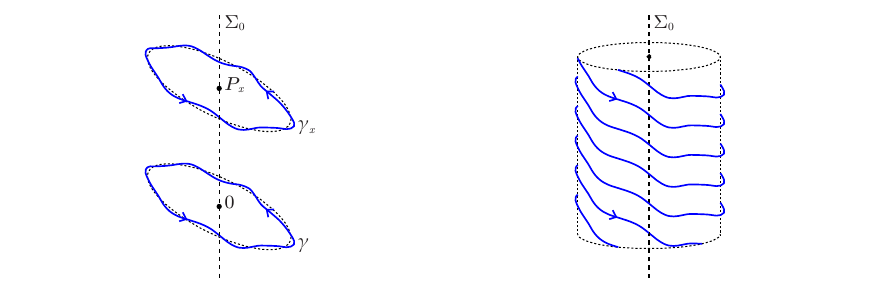}
\captionof{figure}{The picture describes the particular case of an infinite solenoid obtained by translating a planar closed curve $\gamma$ along the axis $\Sigma_0$. For each $x\in\R$, the curve $\gamma_x$ turns around $P_x=(x,0,0)$ and lies in a plane transverse to $\Sigma_0$.}\label{fig:1}
\end{center}

Let us define the physical quantities
\begin{align}
    & \mu_0>0 \quad \text{ permeability of free space},\\
    & I(t) \quad \text{ the current},\\
    & n(t) \quad \text{ the turn density (number of turns of the coil per unit length)},\\
    & K(t,\theta):= n(t) I(t)(\gamma_1'(\theta),\gamma_2'(\theta),\gamma_3'(\theta)) \quad \text{ the surface current density.}
\end{align}
Then, the Biot-Savart law states that the magnetic potential $\tilde A$ can be computed at $z$ as follows 
\begin{equation}\label{def_Adelta}
\tilde A(z):=\frac{\mu_0}{4 \pi}
\int_{-\infty}^{+\infty} \int_0^{2 \pi} \frac{K(t,\theta)}{|z-(t+\gamma_1(\theta),\gamma_2(\theta),\gamma_3(\theta))|}  d\theta dt.
\end{equation}
If the size of the turns is very small, that is, we scale $\gamma_x$ by a small parameter $\delta>0$ (i.e. 
$\gamma^\delta_x=(x+\delta\gamma_1,\delta\gamma_2,\delta\gamma_3)$), we are going to show that a magnetic potential as in \eqref{def_A} is obtained by considering, as $\delta\to0^+$, the normalized leading term in the Taylor expansion of 
\begin{equation}
\tilde A_\delta(z):=\frac{\mu_0}{4 \pi}\delta\int_{-\infty}^{+\infty} \int_0^{2 \pi}
J(t)\frac{\gamma'(\theta)}{|z-(t+\delta\gamma_1(\theta),\delta\gamma_2(\theta),\delta\gamma_3(\theta))|}  d\theta dt.
\end{equation}
Here $J(t):=n(t) I(t)$ is the vertical component of the current density, and we assume
\begin{equation}
tJ,J \in C^{1,1}_{loc}(\R) \cap L^\infty(\R), \quad \text{ and } \quad (tJ)',J' \in L^\infty(\R).
\end{equation}
Letting $f(\delta):=|z-(t+\delta\gamma_1(\theta),\delta\gamma_2(\theta),\delta\gamma_3(\theta))|^{-1}$, a Taylor expansion in $\delta=0$ yields 
\begin{equation}
f(\delta)=|z-(t,0)|^{-1} +\frac{(x-t)\gamma_1(\theta)+\gamma_2(\theta) y_1+\gamma_3(\theta) y_2}{(|y|^2+(x-t)^2)^{\frac{3}{2}}}\delta +  O(\delta^2),  \quad \text{ as } \delta \to 0^+.
\end{equation}
Hence, since the first term in the expansion gives no contribution being $\gamma$ a closed curve,
\begin{multline}
\tilde A_\delta(z)=\frac{\mu_0}{4 \pi} \delta^2
\int_{-\infty}^{+\infty}\frac{J(t)}{(|y|^2+(x-t)^2)^{\frac{3}{2}}} \, dt  
\int_0^{2 \pi} [\gamma_1(\theta)x+\gamma_2(\theta) y_1+\gamma_3(\theta) y_2]\gamma'(\theta)  d\theta\\
-\frac{\mu_0}{4 \pi} \delta^2\int_{-\infty}^{+\infty}\frac{tJ(t)}{(|y|^2+(x-t)^2)^{\frac{3}{2}}} \, dt  
\int_0^{2 \pi} \gamma_1(\theta)\gamma'(\theta)  d\theta
+ O(\delta^3),  \quad \text{ as } \delta \to 0^+.
\end{multline}
With two changes of variables, we obtain
\begin{equation}
\int_{-\infty}^{+\infty}\frac{J(t)}{(|y|^2+(x-t)^2)^{\frac{3}{2}}} \, dt 
=\int_{-\infty}^{+\infty}\frac{J(x+w)}{(|y|^2+w^2)^{\frac{3}{2}}} \, dw
=\frac{1}{|y|^2}\int_{-\infty}^{+\infty}\frac{J(x+|y|\tau)}{(1+\tau^2)^{\frac{3}{2}}} \, d\tau.  
\end{equation}
Let us define $\rho:=|y|$ and 
\begin{equation}
h(\rho ):=\int_{-\infty}^{+\infty}\frac{J(x+\rho\tau)}{(1+\tau^2)^{\frac{3}{2}}} \, d\tau.
\end{equation}
It is easy to see that $h \in C^1([0,+\infty))$ and that

\begin{equation}
|h'(\rho_1)-h'(\rho_2)| \le C\left(\int_{-\infty}^{+\infty}\frac{|\tau|}{(1+\tau^2)^{\frac{3}{2}}}\, d \tau\right)  |\rho_1-\rho_2| \quad \text{ for any } \rho_1,\rho_2 \in [0,+\infty),
\end{equation}
for some constant $C>0$, that is $h \in C_{loc}^{1,1}([0,+\infty))$.
Since we are interested in local results around $0$, with a Taylor expansion around $\rho=0$,
\begin{multline}
h(\rho)= \left(\int_{-\infty}^{+\infty}\frac{1}{(1+\tau^2)^{\frac{3}{2}}}\, d \tau \right) J(x)
- \left(\int_{-\infty}^{+\infty}\frac{\tau}{(1+\tau^2)^{\frac{3}{2}}}\, d \tau\right)J'(x) \rho + O(\rho^2) \\
=2 J(x) + O(\rho^2) \quad \text{ as } \rho \to 0^+.
\end{multline}
Similar computations also yield
\begin{equation}
\int_{-\infty}^{+\infty}\frac{tJ(t)}{(|y|^2+(x-t)^2)^{\frac{3}{2}}} \, dt  = \frac{2x}{|\rho|^2} J(x) + O(1) \quad \text{ as } \rho \to 0^+.
\end{equation}
Let us introduce the coefficients
\begin{equation}
c_{i,j} := \int_0^{2\pi} \gamma_i(\theta)\,\gamma_j'(\theta)\, d\theta,
\qquad i,j \in \{1,2,3\}.
\end{equation}
Since $\gamma$ is closed, integration by parts gives
\[
c_{i,j}
=
\gamma_i \gamma_j\Bigg|^{2 \pi}_0
-
\int_0^{2\pi} \gamma_i' \gamma_j \, d\theta
=
-
c_{j,i}.
\]
Hence the matrix $C=(c_{i,j})$ is antisymmetric, and in particular for $i=1,\dots, 3$
\begin{equation}
c_{i,i}=\int_0^{2 \pi}\gamma_i \gamma_i'\, d \theta= \frac{\gamma_i^2}{2}\Bigg|^{2 \pi}_0=0,
\end{equation}
Then, we may write the normalized leading term in the Taylor expansion of $\tilde A_\delta$ as $\delta\to0^+$ (normalized means divided by a factor $\delta^2$) as
\begin{multline}
\frac{\mu_0 }{2\pi}\frac{J(x)}{\rho^2}(-c_{1,2}y_1-c_{1,3}y_2,c_{1,2}x-c_{2,3}y_2,c_{1,3}x+c_{2,3}y_1)
-\frac{\mu_0 }{2\pi}\frac{xJ(x)}{\rho^2}(0,c_{1,2},c_{1,3})+O(1)\\
= \frac{\mu_0 }{2\pi}\frac{J(x)}{\rho^2}(-c_{1,2}y_1-c_{1,3}y_2,-c_{2,3}y_2,c_{2,3}y_1)+O(1).
\end{multline}
Hence, if we let $b$ be the reminder term in the above equation,
\begin{equation}\label{def_A_loops}
A(x,y_1,y_2):= \frac{1}{\rho}a\left(x,\frac{y_1}{\rho},\frac{y_2}{\rho}\right) +b(x,y)
=\frac{\mu_0 }{2\pi}\frac{J(x)}{\rho^2}(-c_{1,2}y_1-c_{1,3}y_2,-c_{2,3}y_2,c_{2,3}y_1)+b(x,y),
\end{equation}
we obtain a potential as in \eqref{def_A}.
Furthermore,  letting $\theta$ be the angular variable  $\theta:=\left(\frac{y_1}{\rho},\frac{y_2}{\rho}\right)$, the transversality condition \eqref{hp_transversality} is reduced to 
\begin{equation}
a\left(x,\theta\right) \cdot (0,\theta) =0,
\end{equation}
which is easily verified. 
Furthermore, \eqref{hp_hardy} holds as soon as
\begin{equation}
\inf_{x \in B_R'} {\rm{dist}}\left(\frac{c_{2,3}\mu_0J(x) }{2\pi},\mathbb{Z}\right)>0,
\end{equation}
see for example \cite[Section 7]{FFT}.

In the particular case of loops lying on planes $\{x=t\}$, i.e. orthogonal to $\Sigma_0$, and $J\equiv const$ we obtain
\begin{equation}
A(x,y_1,y_2)=\frac{\mu_0 c_{2,3}}{2\pi}\frac{J}{\rho^2}(0,-y_2,y_1)
\end{equation}
which corresponds to the well-known case of the ideal Aharonov--Bohm magnetic potential.

\begin{remark}[Deviation from the ideal AB configuration due to asymmetry and geometric meaning of the spectral shift]\label{rem:phys}
We can canonically identify the antisymmetric matrix $C=(c_{i,j})$ with a vector $\mathbf{C}$ as
\[
C \longleftrightarrow \mathbf{C}
=
\left(
c_{2,3},\,
c_{3,1},\,
c_{1,2}
\right).
\]
Then
\begin{equation}
\mathbf{C}
=
\frac12
\int_0^{2\pi}
\gamma(\theta)\wedge\gamma'(\theta)\, d\theta
\end{equation}
is the oriented area vector associated with the closed curve $\gamma$. Its components represent the oriented areas of the projections of $\gamma$ onto the coordinate planes.

If a steady current $J$ flows along the loop $\gamma$, the associated magnetic dipole moment is given by
\begin{equation}
m=\frac{J}{2}\int_\gamma z \wedge dz=J\,\mathbf{C}.
\end{equation}
Thus, the vector $\mathbf{C}$ coincides, up to the multiplicative factor $J$, with the magnetic moment of the current loop.

Let us consider the magnetic field in \eqref{def_A_loops}. If we assume that $J$ is constant then $b\equiv 0$ and 
\begin{equation}
A(x,y_1,y_2)=\frac{\mu_0 }{2\pi}\frac{J}{\rho^2}(-c_{1,2}y_1-c_{1,3}y_2,-c_{2,3}y_2,c_{2,3}y_1).
\end{equation}
Let us remark that the magnetic field $B=\nabla\times A$ is zero (apart from $\Sigma_0$) if and only if $c_{1,2}=c_{1,3}=0$. In particular, any nontrivial transverse component of the magnetic dipole moment with respect to the chosen blow-up axis breaks the ideal AB regime. Then, the asymmetry of the loops implies a deviation from the pure AB configuration. The same asymmetry induces a spectral shift, which in turn increases the optimal regularity of the wave functions. In fact, no spectral shift occurs if and only if $|a'|=0$; that is, $c_{1,2}=c_{1,3}=0$. We will comment further on this fact in the next Example \ref{ex_big_frequence}.
\end{remark}

\begin{example}[an arbitrarily large spectral shift]\label{ex_big_frequence}
In this example, we show that there is a physical motivation to investigate higher regularity results for solutions to \eqref{eq_magnetic_strong}. In other words, for any $k \in \mathbb{N}$, for a suitable choice of the surface current density and the loops of the solenoid, we show that the associated magnetic potential satisfies $\gamma_1({\rm Id}_d,a)>k$, where ${\rm Id}_d$ is the $d$-dimensional identity matrix and $\gamma_1({\rm Id}_d,a)$ is as in \eqref{def_gamma}. This is consistent with our Theorem \ref{theorem_C1alpha_main}, and would unlock possible $C^k$ regularity results for any $k\geq 2$.

Let us consider the loop $\gamma:[0,2\pi] \to \R^3$ defined as 
\begin{equation}
\gamma(\theta):=(\sin(\beta) \cos(\theta),\cos(\beta) \cos(\theta),\sin(\theta))
\end{equation}
for some chosen $\beta \in (0,\pi/2)$. The support of the loop $\gamma([0,2\pi])$ is a circle centered at $0$ and lying on the plane whose normal vector is $(\cos(\beta),-\sin(\beta),0)$. Notice that the latter plane is tilted from the orthogonal plane (with normal vector $e_1$), exactly by an angle $\beta$. In particular, the requirement $\beta \in (0,\pi/2)$ implies that the plane is transverse ($\beta\neq\pi/2$) but not orthogonal to the axis $\Sigma_0$ ($\beta\neq0$). One can refer to Figure \ref{fig:0} and previous discussions for the construction of the solenoid.
Furthermore,
\begin{align}
&c_{1,2}=\int_0^{2 \pi}\gamma_1 \gamma_2'\, d \theta=-\sin(\beta)\cos(\beta) \int_0^{2 \pi}\cos(\theta) \sin(\theta) \, d \theta=0,\\
&c_{1,3}=\int_0^{2 \pi}\gamma_1 \gamma_3'\, d \theta=\sin(\beta)\int_0^{2 \pi}|\cos(\theta)|^2 \, d \theta=\sin(\beta)\pi,\\
&c_{2,3}=\int_0^{2 \pi}\gamma_2 \gamma_3'\, d \theta=\cos(\beta) \int_0^{2 \pi}|\cos(\theta)|^2\, d \theta=\cos(\beta)\pi.
\end{align}
Thus, if we consider a steady current circulating in the solenoid, the magnetic potential has the form 
\begin{equation}
A(x,y_1,y_2)=\frac{1}{\rho}a\left(x,\frac{y_1}{\rho},\frac{y_2}{\rho}\right)= \frac{\mu_0 }{2}\frac{J}{\rho^2}(-\sin(\beta) y_2,-\cos(\beta)y_2,\cos(\beta)y_1).
\end{equation}
Let us choose the constant current $J$ as follows
\begin{equation}
\frac{\mu_0}{2}J\cos(\beta)= \frac{1}{2},
\end{equation}
so that, with a slight abuse of notations, 
\begin{equation}
a\left(\frac{y_1}{\rho},\frac{y_2}{\rho}\right)= \frac{1}{2\rho}(-\tan(\beta) y_2,-y_2,y_1).
\end{equation}
Letting $\mu_1(a'',0)$ be as in \eqref{def:mu1}, it is clear that $\mu_1(a'',0)>0$, actually $\mu_1(a'',0)=\frac{1}{4}$, see for example \cite[Section 7]{FFT}.
Furthermore, since $n=2$ and $a$ does not depend on $x$, 
\begin{equation}
\gamma_1({\rm Id}_d,a):=\sqrt{\mu_1(a'',|a'|^2)}.
\end{equation}
Moreover, since $a$ depends only on $\beta$, we can denote $\mu_1(\beta):=\mu_1(a'',|a'|^2)$.

Recall that $\mu_1(\beta)$ is actually the first eigenvalue of a magnetic eigenvalue problem on $\mb{S}^1$, see Section \ref{sec_Liou} for details, which admits a well-defined eigenfunction $\varphi \in H^1(\mb{S}^1,\mb{C})\setminus\{0\}$, that is, a  minimizer of  \eqref{def:mu1}. 
More precisely, 
\begin{equation}
(i\nabla_{\mb{S}^1} +a'')^2\varphi+|a'|^2 \varphi=\mu_1(\beta) \varphi.
\end{equation}
By a standard gauge transformation, that is defining $\psi:=e^{-i\Theta} \varphi$, where
\begin{equation*}
\Theta: \mathbb{S}^1\to \R,\quad 
\Theta(\cos t,\sin t)=\frac{1}{2}t \quad  \text{ for any }t\in[0,2\pi),
\end{equation*}
we obtain that $\psi$ solves the problem 
\begin{equation}
\begin{cases}
&\psi''(\theta)+(\mu_1(\beta)-g(\beta)-g(\beta)\cos(2\theta) )\psi(\theta)= 0, \quad \text{ in } (0,2\pi),\\
&\psi(0)=- \psi(2 \pi),
\end{cases}
\end{equation}
where
\begin{equation}
g(\beta):=\frac{1}{8}|\tan(\beta)|^2.
\end{equation}
The latter is the well-know Mathieu equation with parameters 
\begin{equation}
a=\mu_1(\beta)-g(\beta) \quad \text{ and } \quad  q=\frac{1}{2}g(\beta),
\end{equation}
see for example \cite[Chapter 20]{MS_hand}.
Hence, we must have 
\begin{equation}
\mu_1(\beta)=2q+a_{h+\frac{1}{2}}(q)
\end{equation}
for some $h \in \mathbb{N}$, where $a_{h+\frac{1}{2}}(q)$ are the characteristic numbers associated with solutions subject to antiperiodic boundary conditions, see \cite[Chapter 17]{M_book} and \cite[Chapter 20]{MS_hand}.
For any given $h \in \mathbb{N}$, the asymptotic of $a_{h+\frac{1}{2}}(q)$ as $q \to \infty$ is given by
\begin{equation}
a_{h+\frac{1}{2}}(q)=-2q+2(2h+1)\sqrt{q}+o(\sqrt{q}), \quad \text{as } q \to \infty.
\end{equation}
We refer to \cite[Chapter 20, Section 20.2.30]{MS_hand} and \cite[Chapeter 17]{M_book}. Hence,
\begin{equation}
\mu_1(\beta)=2(2h+1)\sqrt{q}+o(\sqrt{q}), \quad \text{as } q \to \infty.
\end{equation}
Let $k \in \mathbb{N}$. In view of the definition of $g$, we conclude that $\mu_1(\beta)>k^2$ for $\beta$ close enough to $\pi/2$ so that 
$\gamma_1({\rm Id}_d,a)>k$.
\end{example}

\subsection{Shrinking solenoid onto a curve}\label{sec:A2}
In this subsection we would like to motivate the presence of the anisotropy matrix $M$ in the covariant magnetic gradient $\nabla_m$.
We consider the more general situation where the solenoid is coiled around a curve $\eta$ and any turn is given by a loop lying in the normal-binormal plane to the tangent to $\eta$. More precisely, let us consider a curve $\eta \in C^2(\R,\R^3)$ such that $\eta'(x)\neq0$, $\eta''(x)\neq0$ for any $x \in \R$. It is not restrictive to suppose that $\eta(0)=0$ and that  $\eta$ is parametrized with respect to the arc length.
Let us define at any $x\in\R$ the triad: unit tangent vector $T(x)$, unit principal normal vector $N(x)$, and unit binormal vector $B(x)$, which are given by
\begin{equation}
T(x):=\frac{\eta'(x)}{|\eta'(x)|}=\eta'(x), \quad N(x):=\frac{\eta''(x)}{|\eta''(x)|} \quad \text{ and }   \quad B(x):=T(x) \wedge   N(x).
\end{equation}
The latter is called the Frenet–Serret frame and is a well-defined orthonormal basis of $\R^3$ for any $x \in \R$. We remark that the Frenet–Serret frame is well defined as long as the curvature of $\eta$ never vanishes. For the sake of simplicity, we use the classic Frenet–Serret frame in this work. Alternatively, one may employ other framings that remain valid even when the curvature vanishes at some points (e.g., the Bishop frame \cite{Bis75} or the Beta frame \cite{Car13}). Without loss of generality, we may also suppose that   $T(0)=e_1$, $N(0)=e_2$, $B(0)=e_3$, up to a rotation.

Let us consider a closed  planar curve $\gamma\in AC([0,2 \pi],\R^2)$ with support in the plane $\{x=0\}$ and coiled around $0$.

Starting from $\gamma$, let us define the  loops coiled around the support of $\eta$ in the plane spanned by $\{N(x),B(x)\}$ (the normal plane to $\eta$ at $x$) defined as  
\begin{equation}
\gamma_x(\theta):=\gamma_1(\theta)N(x)+\gamma_2(\theta) B(x)
\end{equation}
for any $x \in \R$ and any $\theta \in [0,2 \pi]$. The construction of the solenoid is detailed in Figure \ref{fig:2}. In this simplified picture $\gamma$ is a circle; that is, $\gamma(\theta)=(\cos\theta,\sin\theta)$.

\begin{center}
\includegraphics[page=1,scale=1]{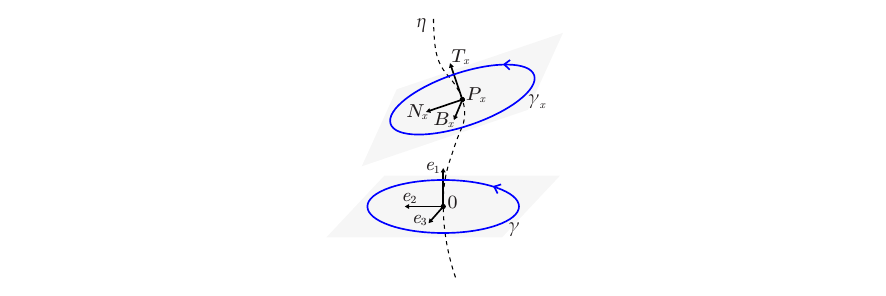}
\captionof{figure}{The picture describes the particular case of an infinite solenoid given by the family of circles $\gamma_x$ coiled along the curve $\eta$ and lying on its normal-binormal plane. In particular, the Frenet–Serret frame is given by $\{T_x,N_x,B_x\}$.}\label{fig:2}
\end{center}

Let us consider 
\begin{equation}\label{def_Phi}
\Phi: \R^3 \to \R^3, \quad \Phi(x,y_1,y_2):= \eta(x)+y_1N(x)+y_2B(x),
\end{equation}
that is, the map $\Phi$ gives  the Fermi-type tubular coordinates  relative to the  Frenet–Serret frame of $\eta$.
A direct computation yields
\begin{equation}\label{def_DPhi}
D\Phi(x,y_1,y_2)= 
\begin{bmatrix}
\eta'(x)+y_1N(x)+y_2B(x), &N(x), &B(x)
\end{bmatrix}
\end{equation}
so that $\det D\Phi \neq 0$ for any $y$ small enough. 
More precisely, there exist $r_0,\e>0$ small enough so that $\Phi$ is a diffeomorphism in $(-\e, \e) \times B''_{r_0}$ onto its image, where 
$B''_{r_0}:=\{y \in \R^2: |y| < r_0\}$.

We are going to shrink the size of the coils with a small parameter $\delta$, and compute the magnetic potential as $\delta \to 0^+$ in tubular coordinates. 
As in the previous subsection, the following are the physical quantities involved in the computation by the Biot-Savart law of the magnetic potential:
\begin{align}
    &\mu_0>0 \quad \text{ permeability of free space},\\
    & I(t) \quad \text{ the current},\\
    & n(t) \quad \text{ the turn density (number of turns of the coil per unit length)},\\
    &K(t,\theta):= n(t) I(t)[\gamma'_1(\theta)N(t)+\gamma'_2(\theta) B(t)] \quad \text{ the surface current density.}
\end{align} 
For any  $z \in (-\e, \e) \times B''_{r_0}$ we have, by the Biot-Savart law, taking as coils $\delta \gamma_x$,
\begin{equation}
\tilde A_\delta(\Phi(z)):= \frac{\mu_0}{4 \pi}
\delta\int_{-\infty}^{+\infty}\int_0^{2 \pi} \frac{K(t,\theta)}{|\Phi(z)-\eta(t)-\delta[\gamma_1(\theta)N(x)+\gamma_2(\theta) B(x)] |} \, d\theta \, dt.
\end{equation}
Let us define 
\begin{equation}
f(\delta):=|\Phi(z)-\eta(t)-\delta[\gamma_1(\theta)N(x)+\gamma_2(\theta) B(x)] |^{-1}.
\end{equation}
Then, a Taylor expansion in $\delta=0$ yields
\begin{equation}
f(\delta)= |\Phi(z)-\eta(t) |^{-1}+ \frac{[\Phi(z)-\eta(t)]\cdot[\gamma_1(\theta)N(t)+\gamma_2(\theta)B(t)]}{|\Phi(z)-\eta(t)|^{3}}\delta+o(\delta), 
\text{ as } \delta \to0^+.
\end{equation}
Since $\gamma$ is closed,  letting $J(t):= n(t) I(t)$,
\begin{equation}
\int_0^{2 \pi} K(t,\theta) \, d\theta
=  J(t)\left[\left(\int_0^{2 \pi} \gamma'_1(\theta) \, d\theta\right) N(t)+\left(\int_0^{2 \pi} \gamma'_2(\theta) \, d \theta\right)  B(t)\right]=0,
\end{equation}
so that
\begin{multline}
\tilde A_\delta(\Phi(z))= \frac{\mu_0}{4 \pi}\delta^2
\int_{-\infty}^{+\infty}\int_0^{2 \pi} \frac{[\Phi(z)-\eta(t)]\cdot[\gamma_1(\theta)N(t)+\gamma_2(\theta)B(t)]}{|\Phi(z)-\eta(t)|^{3}}\\
\times J(t)[\gamma'_1(\theta)N(t)+\gamma'_2(\theta) B(t)] \, d\theta \, dt +o(\delta^2),
\text{ as } \delta \to0^+.
\end{multline}
Letting 
\begin{equation}
c_\gamma:=c_{1,2}=\int_0^{2\pi} \gamma_1 \gamma_2' \, d\theta=-\int_0^{2\pi} \gamma_1' \gamma_2 \, d\theta,
\end{equation}
we conclude that, as $\delta \to 0^+$,
\begin{multline}
\tilde A_\delta(\Phi(z))
= \frac{\mu_0}{4 \pi}c_\gamma\delta^2
\int_{-\infty}^{+\infty}J(t)\frac{-[\Phi(z)-\eta(t)]\cdot B(t) N(t)+[\Phi(z)-\eta(t)]\cdot N(t) B(t)}{|\Phi(z)-\eta(t)|^{3}}\, dt +o(\delta^2).
\end{multline}
Let us now study the normalized leading term of the Taylor expansion of $\tilde A_\delta$ (normalized means divided by a factor $\delta^2$),
\begin{equation}\label{def_A_solenoid_curved}
A(\Phi(z)):= \frac{\mu_0}{4 \pi}c_\gamma
\int_{-\infty}^{+\infty}J(t)\frac{-[\Phi(z)-\eta(t)]\cdot B(t) N(t)+[\Phi(z)-\eta(t)]\cdot N(t) B(t)}{|\Phi(z)-\eta(t)|^{3}}\, dt.
\end{equation}
 A Taylor expansion in $t=x$ with an integral reminder on each component yields
\begin{equation}
 \eta(t)=\eta(x) +(t-x) \eta'(x)+ R(x,t), \quad \text{ where }R(x,t):= \int_x^t\eta''(\sigma)(t-\sigma) \, d \sigma.
\end{equation}
Then, letting  $(y_1,y_2)=\rho (\cos \vartheta, \sin \vartheta)$ and 
\begin{equation}\label{def_Uvartheta}
U(\vartheta,x):=\cos \vartheta N(x)+ \sin \vartheta B(x),
\end{equation}
 with the change of variables  $t=v+x$ we obtain, by \eqref{def_Phi} and  letting $\sigma:=v/\rho$,
\begin{multline}
\int_{-\infty}^{+\infty}J(t)\frac{-[\Phi(z)-\eta(t)]\cdot B(t) N(t)+[\Phi(z)-\eta(t)]\cdot N(t) B(t)}{|\Phi(z)-\eta(t)|^{3}}\, dt\\
=\int_{-\infty}^{+\infty}J(t)\frac{-[\Phi(z)-\eta(t)]\cdot B(t) N(t)+[\Phi(z)-\eta(t)]\cdot N(t) B(t)}{|\rho U(\vartheta,x)+(t-x) \eta'(x)+ R(x,t)|^3}\, dt\\
=-\int_{-\infty}^{+\infty}J(v+x)\frac{[\rho U(\vartheta,x)+v \eta'(x)+ R(x,v+x)]\cdot B(v+x) N(v+x)}{|\rho U(\vartheta,x)+v \eta'(x)+R(x,v+x)|^{3}}\, dv\\
+\int_{-\infty}^{+\infty}J(v+x)\frac{[\rho U(\vartheta,x)+v \eta'(x)+ R(x,v+x)]\cdot N(v+x) B(v+x)}{|\rho U(\vartheta,x)+v \eta'(x)+ R(x,v+x)|^{3}}\, dv\\
=\frac{1}{\rho}\int_{-\infty}^{+\infty}\frac{H(x, \vartheta, \sigma, \rho)}{| U(\vartheta,x)+\sigma \eta'(x)+ \rho^{-1}R(x,\rho \sigma+x)|^{3}}\, d\sigma\\
\end{multline}
where 
\begin{multline}\label{def_J_curved}
H(x,\vartheta, \sigma, \rho):=J(\rho \sigma+x) \{- [U(\vartheta,x)+ \sigma \eta'(x)+ \rho^{-1} R(x,\rho \sigma+x)]\cdot N(\rho \sigma+x) B(\rho \sigma+x)\\
+  [U(\vartheta,x)+ \sigma \eta'(x)+ \rho^{-1} R(x,\rho \sigma+x)]\cdot B(\rho \sigma+x) N(\rho \sigma+x)\}.
\end{multline}
For any $x, \sigma \in \R$ 
\begin{equation}
\rho^{-1}|R(x,\rho \sigma+x)| \le \frac{1}{2} \norm{\eta''}_{L^\infty(\R,\R^3)}\rho \sigma^2,
\end{equation}
and 
\begin{multline}
\lim_{\rho \to 0^+}\pd{}{\rho}[\rho^{-1}R(x,\rho \sigma+x)]
=\lim_{\rho \to 0^+} \left( -\frac{1}{\rho^2}\int_x^{\rho \sigma+x} \eta''(t)(\rho \sigma+x-t) \, dt
+\frac{\sigma}{\rho}\int_x^{\rho \sigma+x} \eta''(t) \ dt\right) \\
=\lim_{\rho \to 0^+} \left(\int_0^{-\sigma} \eta''(x-\rho v)( \sigma+v) \, dv+\frac{\sigma}{\rho} (\eta'(\rho \sigma+x)-\eta'(x)) \ dt\right)\\
=\left(\int_0^{-\sigma}  (\sigma+v) \, dv\right)\eta''(x)+ \sigma^2 \eta''(x)
=\frac{\sigma^2}{2}\eta''(x).
\end{multline}
By \eqref{def_Uvartheta} we also have that 
\begin{equation}
|U(\vartheta,x)+\sigma \eta'(x)|^2 =1+\sigma^2.
\end{equation}
Hence, recalling \eqref{def_J_curved} and \eqref{def_Uvartheta}
\begin{multline}
H(x,\vartheta,\sigma,0)=J(x) \{- [U(\vartheta,x)+ \sigma \eta'(x)]\cdot N(x) B(x)+  [U(\vartheta,x)+ \sigma \eta'(x)]\cdot B(x) N(x)\}\\
=J(x) (- \cos \vartheta B(x)+  \sin \vartheta  N(x)),
\end{multline}
while
\begin{multline}
\pd{H}{\rho}(x,\vartheta, \sigma,0)=\sigma J'(x) \{- [U(\vartheta,x)+ \sigma \eta'(x)]\cdot N(x) B(x)
+ [U(\vartheta,x)+ \sigma \eta'(x)]\cdot B(x) N(x)\}\\
- J(x)\frac{\sigma^2}{2}\eta''(x)\cdot N(x) B(x)+J(x)\frac{\sigma^2}{2}\eta''(x)\cdot B(x) N(x)\\
-\sigma J(x)[U(\vartheta,x)+ \sigma \eta'(x)] \cdot N'(x) B(x)-\sigma J(x)[U(\vartheta,x)+ \sigma \eta'(x)] \cdot N(x) B'(x)\\
+\sigma J(x)[U(\vartheta,x)+ \sigma \eta'(x)] \cdot B'(x) N(x)
+\sigma J(x)[U(\vartheta,x)+ \sigma \eta'(x)] \cdot N(x) B'(x)\\
=\sigma J'(x) [-\cos\vartheta B(x) +\sin \vartheta N(x)] - J(x)\frac{\sigma^2}{2}|\eta''(x)| B(x).
\end{multline}
Furthermore, letting,
\begin{equation}
h(\rho):=\frac{H(x,\vartheta, \sigma, \rho)}{| U(\vartheta,x)+\sigma \eta'(x)+ \rho^{-1}R(x,\rho \sigma+x)|^{3}},
\end{equation}
a Taylor expansion in $\rho=0$ yields
\begin{equation}
h(\rho)= h(0)+\rho h'(0) +o(\rho), \quad  \text{ as } \rho \to 0^+
\end{equation}
with 
\begin{equation}
h(0)=\frac{J(x) (- \cos\vartheta B(x)+  \sin \vartheta  N(x))}{(1+\sigma^2)^{\frac{3}{2}}},
\end{equation}
and 
\begin{multline}
h'(0)=\frac{H'(x,\vartheta,\sigma,0) | U(\vartheta,x)+\sigma \eta'(x)|^3}{| U(\vartheta,x)+\sigma \eta'(x)|^6}\\
\times\frac{-3H(x,\vartheta,\sigma,0)| U(\vartheta,x)+\sigma \eta'(x)|(U(\vartheta,x)+\sigma \eta'(x))\cdot\sigma^2\eta''(x)}{| U(\vartheta,x)+\sigma \eta'(x)|^6}\\
=\frac{H'(x,\vartheta,\sigma,0) (1+\sigma^2)-3H(x,\vartheta,\sigma,0)U(\vartheta,x)\cdot\sigma^2 \eta''(x)}{(1+\sigma^2)^\frac{5}{2}}.
\end{multline}
In conclusion, 
\begin{equation}
A(\Phi(z))= \frac{\mu_0c_\gamma}{2 \pi \rho } J(x) (- \cos\vartheta B(x)+  \sin \vartheta  N(x)) +b(x,y),
\end{equation}
where $b$ is a bounded function.

If we consider now a wave function of a charged particle in the magnetic field $B=\nabla\times A$; that is, a solution $u$ to the equation
\begin{equation}
(i\nabla +A)^2 u=0, \text{ in } B_R
\end{equation}
for some $R>0$, taking $R$ small enough so that  $\Phi:B_R \to \Phi(B_R)$ is a diffeomorphism and letting
\begin{equation}
\tilde{u}:= u \circ \Phi,  \quad M:=D\Phi^{-1} \sqrt{\det D\Phi^{-1}} \quad \text{ and } \quad \tilde{A}=A\circ \Phi,
\end{equation}
it follows that 
\begin{equation}
-\dive(N \nabla \tilde{u})+2i M \tilde{A}\cdot \nabla  \tilde{u} +i M \nabla \cdot \tilde{A} \tilde{u}+ |\tilde{A}|^2 \tilde{u}=0,
\end{equation}
in the weak sense of Section \ref{sec_weak_sol_bounded} with $A\circ \Phi$ as in \eqref{def_A}. This is why the anisotropy $M$ in the operator \eqref{def_L_Af} is particularly interesting from a physical point of view.
Furthermore, by \eqref{def_DPhi}, since $\eta'(x)=T(x)$,
\begin{equation}
D\Phi^{-1}=\begin{bmatrix}
T(x), &N(x), &B(x)
\end{bmatrix}^T+o(|\rho|), \text{ as } \rho \to 0^+,
\end{equation}
thus
\begin{multline}
N=\det D\Phi^{-1} [D\Phi^{-1}]^TD\Phi^{-1} \\
=\begin{bmatrix}
T(x), &N(x), &B(x)
\end{bmatrix}\begin{bmatrix}
T(x), &N(x), &B(x)
\end{bmatrix}^T+o(|\rho|)=\rm{Id}_3+o(|\rho|),  \text{ as } \rho \to 0^+,
\end{multline}
and so, in particular, it is uniformly elliptic, up to taking a smaller $R>0$, that is, it satisfies \eqref{hp_M_elliptic} and \eqref{hp_M_structure}.
We conclude noticing that
\begin{equation}
M^{-1} A(\Phi(z)) =\frac{\mu_0c_\gamma}{2 \pi \rho } J(x) (0, \sin \vartheta , - \cos\vartheta ) +O(1) \quad \text{ as } \rho \to 0^+.
\end{equation}
Hence, the transversality condition \eqref{hp_transversality} is verified and  \eqref{hp_hardy} holds as soon as
\begin{equation}
\inf_{x \in B_R'} {\rm{dist}}\left(\frac{c_{\gamma}\mu_0J(x) }{2\pi},\mathbb{Z}\right)>0,
\end{equation}
see for example \cite[Section 7]{FFT}.\\

{\bf Acknowledgments.}
The authors are research fellows of Istituto Nazionale di Alta Matematica INDAM group GNAMPA. S.V. is supported by the GNAMPA project E5324001950001 \emph{PDE ellittiche che degenerano su variet\`a di dimensione bassa e frontiere libere molto sottili} and supported by the GNAMPA project E53C25002010001 \emph{Struttura fine e regolarita' in problemi variazionali non-lineari}. 
G.S. is supported by the GNAMPA project E53C25002010001  \emph{Asymptotic analysis of variational problems}.
The authors thank Susanna Terracini for many discussions on the physical models in Appendix \ref{sec:solenoid} and Manni for the figures.

\bibliographystyle{acm}
\bibliography{references}	
\end{document}